\documentclass[10pt,a4paper, reqno, preprint,bj]{imsart}

\usepackage[utf8]{inputenc}
\usepackage[T1]{fontenc}
\usepackage{a4wide}
\usepackage[UKenglish]{babel}

\usepackage{amssymb,amsmath,amsfonts, amsthm}

\usepackage{textcomp}
\usepackage{hyphenat}
\usepackage{enumerate}
\usepackage{mathrsfs}
\usepackage{txfonts}
\usepackage{cancel}
\usepackage{multirow}
\usepackage{mathdots}
\usepackage{graphicx}
\usepackage{placeins}

\usepackage[sort]{natbib}

\usepackage[usenames,dvipsnames]{color}
\usepackage[breaklinks=true,
	    bookmarksnumbered=false,
	    pdfdisplaydoctitle=true,
	    pdfstartview=FitH]{hyperref}
\hypersetup{
    colorlinks,%
    citecolor=DarkOrchid,%
    filecolor=black,%
    linkcolor=MidnightBlue,%
    urlcolor=black
}
\hypersetup{
pdftitle ={Quasi maximum likelihood estimation for strongly mixing state space models and multivariate L\'evy\hyp{}driven CARMA processes},
pdfsubject ={Mathematical Statistics},
pdfauthor = {Eckhard Schlemm},
pdfcreator = {ps2pdf},
pdfproducer = {LaTeX}
}

\newcommand{\Y}{\boldsymbol{Y}}
\newcommand{\y}{\boldsymbol{y}}
\newcommand{\X}{\boldsymbol{X}}
\newcommand{\Lb}{\boldsymbol{L}}
\newcommand{\ZZ}{\boldsymbol{Z}}
\newcommand{\W}{\boldsymbol{W}}
\newcommand{\NN}{\boldsymbol{N}}

\newcommand{\E}{\mathbb{E}}
\newcommand{\Pb}{\mathbb{P}}
\newcommand{\N}{\mathbb{N}}
\newcommand{\Z}{\mathbb{Z}}
\newcommand{\R}{\mathbb{R}}
\newcommand{\C}{\mathbb{C}}
\renewcommand{\SS}{\mathbb{S}}
\newcommand{\beps}{\boldsymbol{\varepsilon}}
\newcommand{\dd}{\mathrm{d}}
\newcommand{\ee}{\mathrm{e}}
\newcommand{\ii}{\mathrm{i}}
\newcommand{\DD}{\mathrm{D}}

\newcommand{\imag}{\operatorname{Im}}
\newcommand{\im}{\operatorname{im}}
\newcommand{\I}{{\bf 1}}
\newcommand{\bu}{\boldsymbol{u}}

\newcommand{\bx}{\boldsymbol{x}}
\newcommand{\bb}{\boldsymbol{b}}
\newcommand{\bc}{\boldsymbol{c}}

\newcommand{\bzero}{\boldsymbol{0}}
\newcommand{\Cov}{\operatorname{\mathbb{C}ov}}
\newcommand{\Var}{\operatorname{\mathbb{V}ar}}
\newcommand{\bth}{\boldsymbol{\vartheta}}
\newcommand{\bnu}{\boldsymbol{\nu}}
\newcommand{\BSO}{\operatorname{B}}
\newcommand{\argmin}{\operatorname{argmin}}
\newcommand{\ET}{Ergodic Theorem}
\newcommand{\CLT}{Central Limit Theorem}

\newcommand{\SigGauss}{\Sigma^{\mathcal{G}}}
\newcommand{\SigL}{\Sigma^{\Lb}}

\renewcommand{\vec}{\operatorname{vec}}
\newcommand{\vech}{\operatorname{vech}}
\newcommand{\tr}{\operatorname{tr}}

\newcommand{\convd}{\xrightarrow{d}}
\newcommand{\bgamma}{\boldsymbol{\gamma}}
\newcommand{\bgammaL}{\bgamma^{\Lb}}
\newcommand{\nuL}{\nu^{\Lb}}

\newcommand{\formulaplural}{formulae}

\renewcommand{\leq}{\leqslant}
\renewcommand{\geq}{\geqslant}

\numberwithin{equation}{section}

\usepackage[capitalise]{cleveref}

\theoremstyle{plain}

\newtheorem{theorem}{Theorem}
\newtheorem{lemma}[theorem]{Lemma}
\newtheorem{proposition}[theorem]{Proposition}
\newtheorem{corollary}[theorem]{Corollary}

\numberwithin{theorem}{section}

\theoremstyle{definition}

\newtheorem{definition}{Definition}

\numberwithin{definition}{section}

\newcounter{assumptionlevy}
\newtheorem{assumptionlevy}[assumptionlevy]{Assumption}
\newcounter{assumptioneigen}
\newtheorem{assumptioneigen}[assumptioneigen]{Assumption}
\numberwithin{assumptioneigen}{section}

\newcounter{assumption}

\numberwithin{assumption}{section}

\newcounter{assumptionparaD}[section]
\newtheorem{assumptionparaD}[assumptionparaD]{Assumption}

\newcounter{assumptionparaC}[section]
\newtheorem{assumptionparaC}[assumptionparaC]{Assumption}

\crefname{equation}{Eq.}{Eqs.}
\Crefname{equation}{Equation}{Equations}
\crefname{assumption}{Assumption}{Assumptions}
\crefname{assumptionparaD}{Assumption}{Assumptions}
\crefname{assumptionparaC}{Assumption}{Assumptions}
\crefname{assumptioneigen}{Assumption}{Assumptions}
\crefname{assumptionlevy}{Assumption}{Assumptions}
\creflabelformat{enumi}{#2#1#3)}

\crefname{pluralequation}{Eqs.}{Eqs.}
\creflabelformat{pluralequation}{(#2#1#3)}

\begin{document}

\begin{frontmatter}

\title{Quasi maximum likelihood estimation for strongly mixing state space models and multivariate L\'evy\hyp{}driven CARMA processes}
\runtitle{QML estimation for strongly mixing state space models and MCARMA processes}

\begin{aug}
  \author{Eckhard Schlemm\corref{}\ead[label=e1]{es555@cam.ac.uk}}
  \and
  \author{Robert Stelzer\ead[label=e2]{robert.stelzer@uni-ulm.de}}

  \address{Wolfson College, University of Cambridge,\\ 
           \printead{e1}}

  \ead[label=u2,url]{http://www.uni-ulm.de/mawi/finmath}
  \address{Institute of Mathematical Finance, Ulm University,\\
          \printead{e2,u2}}

%   \thankstext{t2}{Corresponding author}

  \runauthor{E. Schlemm and R. Stelzer}
\end{aug}

\begin{abstract}
We consider quasi maximum likelihood (QML) estimation for general non\hyp{}Gaussian discrete\hyp{}time linear state space models and equidistantly observed multivariate L\'evy\hyp{}driven continuous\hyp{}time autoregressive moving average (MCARMA) processes. In the discrete\hyp{}time setting, we prove strong consistency and asymptotic normality of the QML estimator under standard moment assumptions and a strong-mixing condition on the output process of the state space model. In the second part of the paper, we investigate probabilistic and analytical properties of equidistantly sampled continuous\hyp{}time state space models and apply our results from the discrete\hyp{}time setting to derive the asymptotic properties of the QML estimator of discretely recorded MCARMA processes. Under natural identifiability conditions, the estimators are again consistent and asymptotically normally distributed for any sampling frequency. We also demonstrate the practical applicability of our method through a simulation study and a data example from econometrics.
\end{abstract}

% \subjclass[2010]{Primary: 62F10, %  	Point estimation
% 			  62F12, % 	Asymptotic properties of estimators
% 			  62M09; % 	Non-Markovian processes: estimation
% 		Secondary: 60G51, %   	Processes with independent increments; Lévy processes
% 			  60G10  % 	Stationary processes
% }



\end{frontmatter}

\section{Introduction}
Linear state space models have been used in time series analysis and stochastic modelling for many decades because of their wide applicability and analytical tractability \citep[see, e.\,g.,][for a detailed account]{brockwell1991tst,hamilton1994tsa}. In discrete time they are defined by the equations
\begin{equation}
\label{eq-ssmintro}
 \X_n =F\X_{n-1}+\ZZ_{n-1},\quad\Y_n = H\X_n + \W_n,\quad n\in\Z,
\end{equation}
where $\X=\left(\X_n\right)_{n\in\Z}$ is a latent state process, $F$, $H$ are coefficient matrices and, $\ZZ=\left(\ZZ_n\right)_{n\in\Z}$, $\W=\left(\W_n\right)_{n\in\Z}$ are sequences of random variables, see \cref{def-ssmgen} for a precise formulation of this model. In this paper we investigate the problem of estimating the coefficient matrices $F,H$ as well as the second  moments of $\ZZ$ and $\W$ from a sample of observed values of the output process $\Y=\left(\Y_n\right)_{n\in\Z}$, using a quasi maximum likelihood (QML) or generalized least squares approach. Given the importance of this problem in practice, it is surprising that a proper mathematical analysis of the QML estimation for the model \labelcref{eq-ssmintro} has only been performed in cases where the model is in the so-called innovations form
\begin{equation}
\label{eq-ssminnoform}
\X_n =F\X_{n-1}+K\beps_{n-1},\quad\Y_n = H\X_n + \beps_n,\quad n\in\Z,
\end{equation}
where the innovations $\beps$ 
% form a martingale difference sequence 
have constant conditional variance and satisfy some higher order moment conditions \citep[Chapter 4]{hannan1987stl}. This includes state space models in which the noise sequences $\ZZ,\W$ are Gaussian, because then the innovations, which are uncorrelated by definition, form an i.\,i.\,d.\ sequence. Restriction to these special cases excludes, however, the state space representations of aggregated linear processes, as well as of equidistantly observed continuous\hyp{}time linear state space models. 

In the first part of the present paper we shall prove consistency (Theorem \ref{theorem-consistency}) and asymptotic normality (Theorem \ref{theorem-hatbthLCLT}) of the QML estimator for the general linear state space model \labelcref{eq-ssmintro} under the assumptions that the noise sequences $\ZZ,\W$ are ergodic, and that the output process $\Y$ satisfies a strong\hyp{}mixing condition in the sense of \citet{rosenblatt1956central}. This assumption is not very restrictive, and is, in particular, satisfied if the noise sequence $\ZZ$ is i.\,i.\,d.\ with an absolutely continuous component, and $\W$ is strongly mixing. Our results are a multivariate generalization of \citet{francq1998}, who considered the QML estimation for univariate strongly mixing ARMA processes. The very recent paper \citet{mainassara2011estimating}, which deals with the structural estimation of weak vector ARMA processes, instead makes a mixing assumption about the innovations sequence $\beps$ of the process under consideration, which is very difficult to verify for state space models; their results can therefore not be used for the estimation of general discretely\hyp{}observed linear continuous\hyp{}time state space models.
% More importantly, their proof appears to be incomplete, because a crucial step in the proof of their Lemma 4 is claimed by the authors to be analogous to the corresponding step in the proof of \citet[Lemma 3]{francq1998}. It is, however, not clear how the argument given there can be modified in order to be compatible with the assumption of strongly mixing innovations, which does not imply that the output process is strongly mixing.

As alluded to above, one advantage of relaxing the assumption of i.\,i.\,d.\ innovations in a discrete\hyp{}time state space model is the inclusion of sampled continuous\hyp{}time state space models. These were introduced in the form of continuous\hyp{}time ARMA (CARMA) models in \citet{doob1944egp} as stochastic processes satisfying the formal analogue of the familiar autoregressive moving average equations of discrete\hyp{}time ARMA processes, namely
\begin{equation}
\label{eq-carmadoob}
a(\DD)Y(t) = b(\DD)\DD W(t),\quad\DD=\dd/\dd t,
\end{equation}
where $a$ and $b$ are suitable polynomials, and $W$ denotes a Brownian motion. In the recent past, a considerable body of research has been devoted to these processes.
% \citep[see, e.\,g.,][and references therein]{brockwell2001cta}
One particularly important extension of the model \labelcref{eq-carmadoob} was introduced in \citet{brockwell2001levy}, where the driving Brownian motion was replaced by a L\'evy process with finite logarithmic moments. This allowed for a wide range of possibly heavy-tailed marginal distribution of the process $Y$ as well as the occurrence of jumps in the sample paths, both characteristic features of many observed time series, e.\,g.\ in finance \citep{cont2001empirical}. Recently, \citet{marquardt2007multivariate} further generalized \cref{eq-carmadoob} to the multivariate setting, which gave researchers the possibility to model several dependent time series jointly by one linear continuous\hyp{}time process. This extension is important, because many time series, exhibit strong dependencies and can therefore not be modelled adequately on an individual basis. In that paper, the multivariate non\hyp{}Gaussian equivalent of \cref{eq-carmadoob}, namely $P(\DD)\Y(t) = Q(\DD)\DD \Lb(t)$, for matrix\hyp{}valued polynomials $P$ and $Q$ and a L\'evy process $\Lb$, was interpreted by spectral techniques as a continuous\hyp{}time state space model of the form
\begin{equation}
\label{eq-ssmintroCT}
\dd \boldsymbol{G}(t)=\mathcal{A}\boldsymbol{G}(t)\dd t+\mathcal{B} \dd \Lb(t),\quad \Y(t)= \mathcal{C}\boldsymbol{G}(t);
\end{equation}
see \cref{eq-MCARMAcoeffABC-QML} for an expression of the matrices $\mathcal{A}$, $\mathcal{B}$ and $\mathcal{C}$. The structural similarity between \cref{eq-ssmintro} and \cref{eq-ssmintroCT} is apparent, and it is essential for many of our arguments. Taking a different route, multivariate CARMA processes can be defined as the continuous\hyp{}time analogue of discrete\hyp{}time vector ARMA models, described in detail in \citet{hannan1987stl}.
% lutkepohl2005nim 
As continuous\hyp{}time processes, CARMA processes are suited particularly well to model irregularly spaced and high-frequency data, which makes them a flexible and efficient tool for building stochastic models of time series arising in the natural sciences, engineering and finance \citep[e.\,g.][]{benth2009dynamic,todorov2006simulation}.
% fan1998cta, na2002experimental
In the univariate Gaussian setting, several different approaches to the estimation problem of CARMA processes have been investigated \citep[see, e.\,g.,][and references therein]{larsson2006oip}.
% nielsen2000pes
Maximum likelihood estimation based on a continuous record was considered in \citet{feigin1976mle,pham1977estimation,brown1975alt}. Due to the fact that processes are typically not observed continuously and the limitations of digital computer processing, inference based on discrete observations has become more important in recent years; these approaches include variants of the Yule--Walker algorithm for time\hyp{}continuous autoregressive processes \citep{hyndman1993ywe}, maximum likelihood methods \citep{brockwell2010estimation},
% duncan1999nsa, larsson2002ict, lahalle2004cas, masry1978psa
and randomized sampling \citep{rivoira2002rtc} to overcome the aliasing problem. Alternative methods include discretization of the differential operator \citep{soderstrom1997sau}, and spectral estimation \citep{gillberg2009fdi,lii1995srs}. For the special case of Ornstein--Uhlenbeck processes, least squares and moment estimators have also been investigated without the assumptions of Gaussianity \citep{hu2009leastsquares,spiliopoulos2008mme}.

In the second part of this paper we consider the estimation of general multivariate CARMA (MCARMA) processes with finite second moments based on equally spaced discrete observations exploiting the results about the QML estimation of general linear discrete\hyp{}time state space models. Under natural identifiability assumptions we obtain in the main Theorem \ref{theorem-CLTmcarma} strongly consistent and asymptotically normal estimators for the coefficient matrices of a second\hyp{}order MCARMA process and the covariance matrix of the driving L\'evy process, which determine the second\hyp{}order structure of the process. It is a natural restriction of the QML method that distributional properties of the driving L\'evy process which are not determined by its covariance matrix cannot be estimated. However, once the autoregressive and moving average coefficients of a CARMA process are (approximately) known, and if high-frequency observations are available, a parametric model for the driving L\'evy process can be estimated by the methods described in \citet{brockwell2011parametric}. Thus it should be noted that the paper \citet{brockwell2011parametric} considers the same model, but whereas the present paper considers the estimation of the autoregressive and moving average parameters from equidistant observations letting the number of observations go to infinity, \citet{brockwell2011parametric} assume that the autoregressive and moving average parameters are known and show how to estimate the driving L\'evy process and its parameters when both the observation frequency and the time horizon go to infinity. A further related paper is 
\cite{schlemmmixing2010} whose result on the equivalence of MCARMA processes and state space models provides the foundations for the estimation procedure considered here. That paper also aimed at using the results of \cite{mainassara2011estimating} directly to estimate the autoregressive and moving average parameters of an MCARMA process and therefore provided conditions for the noise of the induced discrete time state space model to be strongly mixing. However, when we investigated this route further it turned out that the approach we take in the present paper is more general and far more convenient, since any stationary discretely sampled MCARMA process with finite second moments is strongly mixing, whereas  assumptions ensuring a non-trivial absolutely continuous component    of the noise are needed to be able to use the results of \cite{mainassara2011estimating}. Hence, the approach taken in the present paper appears rather natural for MCARMA processes. Finally, we note that the estimation of the spectral density of univariate CARMA processes and the estimation in the case of an infinite variance has recently been considered in \cite{FasenFuchs2012a,FasenFuchs2012b}, and that \cite{Fasen2012} looks at the behaviour of the sample autocovariance function of discretely observed MCARMA processes in a high frequency limit.

\paragraph*{\bf Outline of the paper}
The organization of the paper is as follows. In \cref{section-QMLDTSSM} we develop a QML estimation theory for general non\hyp{}Gaussian discrete\hyp{}time linear stochastic state space models with finite second moments.  In \cref{section-notationQMLDTSSM} we precisely define the class of linear stochastic state space models as well as the QML estimator. The main results, that  under a set of technical conditions this estimator is strongly consistent and asymptotically normally distributed as the number of observations tends to infinity, are given as \cref{theorem-consistency,theorem-hatbthLCLT} in Section \ref{section-resultsDTSSM}. The following two Sections \labelcref{section-QMLDTSSMconsistency,section-QMLDTSSMnormality} present the proofs.

In \cref{section-QMLMCARMA} we use the results from \cref{section-QMLDTSSM} to establish asymptotic properties of a QML estimator for multivariate CARMA processes which are observed on a fixed equidistant time grid. As a first step, we review in \cref{section-MCARMA-QML} their definition as well as their relation to the class of continuous\hyp{}time state space models. This is followed by an investigation of the probabilistic properties of a sampled MCARMA process in \cref{section-sampling} and an analysis of the important issue of identifiability in \cref{section-identifiability}. Finally, we are able to state and prove our main result, \cref{theorem-CLTmcarma}, about the strong consistency and asymptotic normality of the QML estimator for equidistantly sampled multivariate CARMA processes in \cref{section-asymptoticmcarma}.

In the final \cref{section-practicalapplicability}, we present canonical parametrizations, and we demonstrate the applicability of the QML estimation for continuous\hyp{}time state space models with a simulation study.
% and a data example from economics.
\paragraph*{\bf Notation}
We use the following notation: The space of $m\times n$ matrices with entries in the ring $\mathbb{K}$ is denoted by $M_{m,n}(\mathbb{K})$ or $M_{m}(\mathbb{K})$ if $m=n$. The set of symmetric matrices is denoted by $\SS_m(\mathbb{K})$, and the symbols $\SS_m^+(\R)$ ($\SS_m^{++}(\R)$) stand for the subsets of positive semidefinite (positive definite) matrices, respectively. $A^T$ denotes the transpose of the matrix A, $\im A$ its image, $\ker A$ its kernel, $\sigma(A)$ its spectrum, and $\I_m\in M_m(\mathbb{K})$ is the identity matrix. The vector space $\R^m$ is identified with $M_{m,1}(\R)$ so that $\bu=(u^1,\ldots,u^m)^T\in\R^m$ is a column vector. $\left\|\cdot\right\|$ represents the Euclidean norm, $\langle\cdot,\cdot\rangle$ the Euclidean inner product, and $\bzero_m\in\R^m$ the zero vector. $\mathbb{K}[X]$ ($\mathbb{K}\{X\}$) denotes the ring of polynomial (rational) expressions in X over $\mathbb{K}$, $I_B(\cdot)$ the indicator function of the set $B$, and $\delta_{n,m}$ the Kronecker symbol. The symbols $\E$, $\Var$, and $\Cov$ stand for the expectation, variance and covariance operators, respectively. Finally, we write $\partial_m$ for the partial derivative operator with respect to the $m$th coordinate and $\nabla=\left(\begin{array}{ccc}\partial_1 & \cdots & \partial_r\end{array}\right)$ for the gradient operator. When there is no ambiguity, we use  $\partial_m f(\bth_0)$ and $\nabla_{\bth}f(\bth_0)$ as shorthands for  $\partial_m f(\bth)|_{\bth=\bth_0}$ and $\nabla_{\bth}f(\bth)|_{\bth=\bth_0}$, respectively. A generic constant, the value of which may change from line to line, is denoted by $C$.

\section{Quasi maximum likelihood estimation for  state space models}
\label{section-QMLDTSSM}
In this section we investigate QML estimation for general linear state space models in discrete time, and prove consistency and asymptotic normality. On the one hand, due to the wide applicability of state space systems in stochastic modelling and control, these results are interesting and useful in their own right. In the present paper they will be applied in \cref{section-QMLMCARMA} to prove asymptotic properties of the QML estimator for discretely observed multivariate continuous\hyp{}time ARMA processes.

Our theory extends existing results from the literature, in particular concerning the QML estimation of Gaussian state space models, of state space models with independent innovations \citep{hannan1975eam},
% hannan1969emm, reinsel1997elements
and of weak univariate ARMA processes which satisfy a strong mixing condition \citep{francq1998}. The techniques used in this section are similar to \citet{mainassara2011estimating}.
% who consider the estimation of structural discrete\hyp{}time vector ARMA models with strongly mixing innovations. 
% Neither of these theories is applicable to the estimation of multivariate CARMA processes because ... ??
\subsection{Preliminaries and definition of the QML estimator}
\label{section-notationQMLDTSSM}

The general linear stochastic state space model is defined as follows.
\begin{definition}
% [State space model]
\label{def-ssmgen}
An $\R^d$\hyp{}valued discrete\hyp{}time linear stochastic state space mo\-del $(F,H,\ZZ,\W)$ of dimension $N$ is characterized by a strictly stationary $\R^{N+d}$\hyp{}valued sequence $\left(\begin{array}{cc}\ZZ^T & \W^T\end{array}\right)^T$ with mean zero and finite covariance matrix
\begin{equation}
\label{eq-def-ssmgencov}
\E\left[\left(\begin{array}{c}\ZZ_n\\\W_n\end{array}\right)\left(\begin{array}{cc}\ZZ_m^T & \W_m^T\end{array}\right)\right]=\delta_{m,n}\left(\begin{array}{cc}Q & R \\R^T & S\end{array}\right),\quad n,m\in\Z,
\end{equation}
for some matrices $Q\in \SS^+_N(\R)$, $S\in \SS^+_d(\R)$, and $R\in M_{N,d}(\R)$; a state transition matrix $F\in M_N(\R)$; and an observation matrix $H\in M_{d,N}(\R)$. It consists of a state equation
\begin{subequations}
\label[pluralequation]{eq-ssmDT}
\begin{equation}
\label{eq-eq-stateeq-QMLDT}
 \X_n =F\X_{n-1}+\ZZ_{n-1},\quad n\in\Z,
 \end{equation}
and an observation equation
\begin{equation}
\label{eq-obseqDT}
\Y_n = H\X_n+\W_n,\quad n\in\Z.
\end{equation}
\end{subequations}
The $\R^N$\hyp{}valued autoregressive process $\X=(\X_n)_{n\in\Z}$ is called the {\it state vector process}, and $\Y=(\Y_n)_{n\in\Z}$ is called the {\it output process}.
\end{definition}
The assumption that the processes $\ZZ$ and $\W$ are centred is not essential for our results, but simplifies the notation considerably. Basic properties of the output process $\Y$ are described in \citet[\S 12.1]{brockwell1991tst}; in particular, if the eigenvalues of $F$ are less than unity in absolute value, then $\Y$ has the moving average representation
\begin{equation}
\label{eq-MArepY}
\Y_n = \W_n + H \sum_{\nu=1}^\infty{F^{\nu-1}\ZZ_{n-\nu}},\quad n\in\Z.
\end{equation}

Before we turn our attention to the estimation problem for this class of state space models, we review the necessary aspects of the theory of Kalman filtering, see \citet{kalman1960new} for the original control-theoretic account and \citet[\S 12.2]{brockwell1991tst} for a treatment in the context of time series analysis. The linear innovations of the output process $\Y$ are of particular importance for the QML estimation of state space models.
\begin{definition}
% [Linear innovations]
\label{definition-innovations}
Let $\Y=(\Y_n)_{n\in\Z}$ be an $\R^d$\hyp{}valued stationary stochastic process with finite second moments. The {\it linear innovations} $\beps=(\beps_n)_{n\in\Z}$ of $\Y$ are then defined by
\begin{equation}
\label{DefInno-QML}
\beps_n=\Y_n-P_{n-1}\Y_n,\quad P_n=\text{orthogonal projection onto } \overline{\operatorname{span}}\left\{\Y_\nu:-\infty<\nu\leq n\right\},
\end{equation}
where the closure is taken in the Hilbert space of square\hyp{}integrable random variables with inner product $(X,Y)\mapsto \E \langle X,Y\rangle$.
\end{definition}
This definition immediately implies that the innovations $\beps$ of a stationary stochastic process $\Y$ are stationary and uncorrelated. The following proposition is a combination of \citet[Proposition 12.2.3]{brockwell1991tst} and \citet[Proposition 13.2]{hamilton1994tsa}.
\begin{proposition}
\label{prop-Kalmanfilter}
Assume that $\Y$ is the output process of the state space model \labelcref{eq-ssmDT}, that at least one of the matrices $Q$ and $S$ is positive definite, and that the absolute values of the eigenvalues of $F$ are less than unity. Then the following hold.
\begin{enumerate}[i)]
 \item\label{prop-KalmanfilterDARE} The discrete\hyp{}time algebraic Riccati equation
\begin{equation}
\label{eq-DefOmega}
\Omega = F\Omega F^T+Q-\left[F\Omega H^T+R\right]\left[H\Omega H^T+S\right]^{-1}\left[F\Omega H^T+R\right]^T
\end{equation}
has a unique positive semidefinite solution $\Omega\in\SS^+_N(\R)$.
 \item\label{prop-KalmanfilterGain} The absolute values of the eigenvalues of the matrix $F-KH\in M_N(\R)$ are less than one, where
\begin{equation}
\label{eq-DefK-QML}
K =\left[F\Omega H^T+R\right]\left[H\Omega H^T+S\right]^{-1}\in M_{N,d}(\R)
\end{equation}
is the steady-state Kalman gain matrix.
 \item\label{prop-KalmanfilterMAbeps}  The linear innovations $\beps$ of $\Y$ are the unique stationary solution to
\begin{subequations}
\label[pluralequation]{eq-innoSSMcombined}
\begin{equation}
\label{eq-innoSSM}
\hat\X_n = \left(F-KH\right)\hat\X_{n-1}+K\Y_{n-1},\quad \beps_n = \Y_n - H\hat\X_n,\quad n\in\Z.
\end{equation}
Using the backshift operator $\BSO$, which is defined by $\BSO\Y_n=\Y_{n-1}$, this can be written equivalently as
\begin{align}
\label{eq-innoSSMpolynomial}
\beps_n =& \left\{\I_d - H\left[\I_N - (F-KH)\BSO\right]^{-1}K\BSO\right\}\Y_n = \Y_n - H\sum_{\nu=1}^\infty{(F-KH)^{\nu-1}K\Y_{n-\nu}}.
\end{align}
\end{subequations}
The covariance matrix $V=\E \beps_n \beps_n^T\in\SS^+_d(\R)$ of the innovations $\beps$ is given by
\begin{equation}
\label{eq-DefVcovmatrix}
V = \E \beps_n \beps_n^T = H\Omega H^T + S.
\end{equation}
\item \label{prop-KalmanfilterMAY} The process $\Y$ has the innovations representation
\begin{subequations}
\begin{equation}
\label{eq-innorepSSM}
\hat\X_n = F\X_{n-1} + K\beps_{n-1},\quad \Y_n = H\X_n + \beps_n,\quad n\in\Z,
\end{equation}
which, similar to \cref{eq-innoSSMcombined}, allows for the moving average representation
\begin{align}
\label{eq-MAssmY}
\Y_n =& \left\{\I_d - H\left[\I_N - F\BSO\right]^{-1}K\BSO\right\}\Y_n = \beps_n + H\sum_{\nu=1}^\infty{F^{\nu-1}K\beps_{n-\nu}},\quad n\in\Z.
\end{align}
\end{subequations}
\end{enumerate}
\end{proposition}
% We now consider parametric families of state space models
For some parameter space $\Theta\subset\R^r$, $r\in\N$, the mappings
\begin{subequations}
\label[pluralequation]{eq-paramSSM}
\begin{equation}
\label{eq-paramSSMFH}
F_{(\cdot)}:\Theta\to M_{N}(\R),\qquad H_{(\cdot)}:\Theta\to M_{d,N},
\end{equation}
together with a collection of strictly stationary stochastic processes $\ZZ_{\bth}$, $\W_{\bth}$, $\bth\in\Theta$, with finite second moments determine a parametric family $\left(F_{\bth},H_{\bth},\ZZ_{\bth},\W_{\bth}\right)_{\bth\in\Theta}$ of linear state space models according to \cref{def-ssmgen}. For the variance and covariance matrices of the noise sequences $\ZZ,\W$ we use the notation (cf. \cref{eq-def-ssmgencov}) $Q_{\bth} = \E\ZZ_{\bth,n}\ZZ_{\bth,n}^T$, $S_{\bth} = \E\W_{\bth,n}\W_{\bth,n}^T$, and $R_{\bth} = \E\ZZ_{\bth,n}\W_{\bth,n}^T$, which defines the functions
\begin{equation}
\label{eq-paramSSMQSR}
Q_{(\cdot)}:\Theta\to \SS^+_{N}(\R),\qquad S_{(\cdot)}:\Theta\to \SS^+_d,\qquad R_{(\cdot)}:\Theta\to M_{N,d}(\R).
\end{equation}
\end{subequations}
It is well known \citep[Eq. (11.5.4)]{brockwell1991tst} that for this model, minus twice the logarithm of the Gaussian likelihood of $\bth$ based on a sample $\y^L=(\Y_1,\ldots,\Y_L)$ of observations can be written as
\begin{equation}
 \label{eq-likelihood}
\mathscr{L}(\bth,\y^L) = \sum_{n=1}^L l_{\bth,n}=\sum_{n=1}^L\left[d\log{2\pi}+\log{\det V_{\bth}}+\beps_{\bth,n}^T V_{\bth}^{-1}\beps_{\bth,n}\right],
\end{equation}
where $\beps_{\bth,n}$ and $V_{\bth}$ are given by analogues of \cref{eq-innoSSM,eq-DefVcovmatrix}, namely
\begin{align}
\label{eq-defpseudoinno}
\beps_{\bth,n} =& \left\{\I_d - H_{\bth}\left[\I_N - (F_{\bth}-K_{\bth}H_{\bth})\BSO\right]^{-1}K_{\bth}\BSO\right\}\Y_n,\quad n\in\Z,\qquad V_{\bth} = H_{\bth}\Omega_{\bth} H_{\bth}^T + S_{\bth},
\end{align}
and $K_{\bth}, \Omega_{\bth}$ are defined in the same way as $K$, $\Omega$ in \cref{eq-DefK-QML,eq-DefOmega}. In the following we always assume that $\y^L=(\Y_{\bth_0,1},\ldots,\Y_{\bth_0,L})$ is a sample from the output process of the state space model $\left(F_{\bth_0},H_{\bth_0},\ZZ_{\bth_0},\W_{\bth_0}\right)$ corresponding to the parameter value $\bth_0$. We therefore call $\bth_0$ the {\it true parameter value}. It is important to note that $\beps_{\bth_0}$ are the true innovations of $\Y_{\bth_0}$, and that therefore $\E\beps_{\bth_0,n}\beps_{\bth_0,n}^T=V_{\bth_0}$, but that this relation fails to hold for other values of $\bth$. This is due to the fact that $\beps_{\bth}$ is not the true innovations sequence of the state space model corresponding to the parameter value $\bth$. We therefore call the sequence $\beps_{\bth}$ {\it pseudo\hyp{}innovations}. 

The goal of this section is to investigate how the value $\bth_0$ can be estimated from $\y^L$ by maximizing \cref{eq-likelihood}. The first difficulty one is confronted with is that the pseudo\hyp{}innovations $\beps_{\bth}$ are defined in terms of the full history of the process $\Y=\Y_{\bth_0}$, which is not observed. It is therefore necessary to use an approximation to these innovations which can be computed from the finite sample $\y^L$. One such approximation is obtained if, instead of using the steady-state Kalman filter described in \cref{prop-Kalmanfilter}, one initializes the filter at $n=1$ with some prescribed values. More precisely, we define the approximate pseudo\hyp{}innovations $\hat\beps_{\bth}$ via the recursion
\begin{equation}
\label{eq-innohatSSM}
\hat\X_{\bth,n} = \left(F_{\bth}-K_{\bth}H_{\bth}\right)\hat\X_{\bth,n-1}+K_{\bth}\Y_{n-1},\quad \hat\beps_{\bth,n} = \Y_n - H_{\bth}\hat\X_{\bth,n},\quad n\in\N,
\end{equation}
and the prescription $\hat\X_{\bth,1} = \hat\X_{\bth,\text{initial}}$. The initial values $\hat\X_{\bth,\text{initial}}$ are usually either sampled from the stationary distribution of $\X_{\bth}$, if that is possible, or set to some deterministic value. Alternatively, one can additionally define a positive semidefinite matrix $\Omega_{\bth,\text{initial}}$ and compute Kalman gain matrices $K_{\bth,n}$ recursively via \citet[Eq. (12.2.6)]{brockwell1991tst}. While this procedure might be advantageous for small sample sizes, the computational burden is significantly smaller when the steady-state Kalman gain is used. The asymptotic properties which we are dealing with in this paper are expected to be the same for both choices because the Kalman gain matrices $K_{\bth,n}$ converge to their steady state values as $n$ tends to infinity \citep[Proposition 13.2]{hamilton1994tsa}.

The QML estimator $\hat\bth^L$ for the parameter $\bth$ based on the sample $\y^L$ is defined as
\begin{equation}
\label{eq-DefhatbthL}
\hat\bth^L = \argmin_{\bth\in\Theta}\widehat{\mathscr{L}}(\bth,\y^L),
\end{equation}
where $\widehat{\mathscr{L}}(\bth,\y^L)$ is obtained from $\mathscr{L}(\bth,\y^L)$ by substituting $\hat\beps_{\bth,n}$ from \cref{eq-innohatSSM} for $\beps_{\bth,n}$, i.\,e.\
\begin{align}
 \label{eq-hatlikelihood}
\widehat{\mathscr{L}}(\bth,\y^L) =& \sum_{n=1}^L \hat l_{\bth,n} =\sum_{n=1}^L\left[d\log{2\pi}+\log{\det V_{\bth}}+\hat\beps_{\bth,n}^T V_{\bth}^{-1}\hat\beps_{\bth,n}\right].
\end{align}

\subsection{Technical assumptions and main results}
\label{section-resultsDTSSM}

Our main results about the QML estimation for discrete\hyp{}time state space models are \cref{theorem-consistency}, stating that the estimator $\hat\bth^L$ given by \cref{eq-DefhatbthL} is strongly consistent, which means that $\hat\bth^L$ converges to $\bth_0$ almost surely, and \cref{theorem-hatbthLCLT}, which asserts the asymptotic normality of $\hat\bth^L$ with the usual $L^{1/2}$ scaling. In order to prove these results, we need to impose the following conditions.
\begin{assumptionparaD}
\label{assum-compact}
The parameter space $\Theta$ is a compact subset of $\R^r$. 
\end{assumptionparaD}

\begin{assumptionparaD}
\label{assum-smoothparam}
The mappings $F_{(\cdot)}$, $H_{(\cdot)}$, $Q_{(\cdot)}$, $S_{(\cdot)}$, and $R_{(\cdot)}$ in \cref{eq-paramSSM} are continuous.
\end{assumptionparaD}
The next condition guarantees that the models under consideration describe stationary processes.
\begin{assumptionparaD}
\label{assum-stability}
For every $\bth\in\Theta$, the following hold:
\begin{enumerate}[i)]
 \item \label{assum-stabilityF} the eigenvalues of $F_{\bth}$ have absolute values less than unity,
 \item \label{assum-stabilityQS}at least one of the two matrices $Q_{\bth}$ and $S_{\bth}$ is positive definite,
 \item \label{assum-stabilityV} the matrix $V_{\bth}$ is non\hyp{}singular.
\end{enumerate}
\end{assumptionparaD}
The next lemma shows that the assertions of \cref{assum-stability} hold in fact uniformly in $\bth$.
\begin{lemma}
\label{lemma-assum123consequences}
Suppose that \cref{assum-compact,assum-smoothparam,assum-stability} are satisfied. Then the following hold.
\begin{enumerate}[i)]
 \item\label{lemma-assum123consequences-F} There exists a positive number $\rho<1$ such that, for all $\bth\in\Theta$, it holds that
\begin{subequations}
\begin{equation}
\max\left\{|\lambda|:\lambda\in\sigma\left(F_{\bth}\right)\right\}\leq\rho.
\end{equation}
 \item\label{lemma-assum123consequences-FminusKH}There exists a positive number $\rho<1$ such that, for all $\bth\in\Theta$, it holds that
\begin{equation}
\max\left\{|\lambda|:\lambda\in\sigma\left(F_{\bth}-K_{\bth}H_{\bth}\right)\right\}\leq\rho,
\end{equation}
\end{subequations}
where $K_{\bth}$ is defined by \cref{eq-DefOmega,eq-DefK-QML}.
 \item\label{lemma-assum123consequences-V} There exists a positive number $C$ such that $\left\|V_{\bth}^{-1}\right\|\leq C$ for all $\bth$.
\end{enumerate}
\end{lemma}
\begin{proof}
Assertion \labelcref{lemma-assum123consequences-F} is a direct consequence of \cref{assum-stability}, \labelcref{assum-stabilityF}, the assumed smoothness of $\bth\mapsto F_{\bth}$ (\cref{assum-smoothparam}), the compactness of $\Theta$ (\cref{assum-compact}), and the fact \citep[Fact 10.11.2]{bernstein2005matrix} that the eigenvalues of a matrix are continuous functions of its entries. Claim \labelcref{lemma-assum123consequences-FminusKH} follows with the same argument from \cref{prop-Kalmanfilter}, \labelcref{prop-KalmanfilterGain} and the fact that the solution of a discrete\hyp{}time algebraic Riccati equation is a continuous function of the coefficient matrices
% \citep[Chapter 14]{lancaster1995algebraic},
\citep{sun1998sensitivity}. Moreover, by \cref{eq-DefVcovmatrix}, the function $\bth\mapsto V_{\bth}$ is continuous, which shows that \cref{assum-stability}, \labelcref{assum-stabilityV} holds uniformly in $\bth$ as well, and so \labelcref{lemma-assum123consequences-V} is proved.
\end{proof}
For the following assumption about the noise sequences $\ZZ$ and $\W$ we use the usual notion of ergodicity \citep[see, e.\,g.,][Chapter 6]{durrett2010probability}. 
\begin{assumptionparaD}
\label{assum-2moments}
The process $\left(\begin{array}{cc} \W_{\bth_0}^T & \ZZ_{\bth_0}^T \end{array}\right)^T$ is ergodic.
\end{assumptionparaD}
The assumption that the processes $\ZZ_{\bth_0}$ and $\W_{\bth_0}$ are ergodic implies via the moving average representation \labelcref{eq-MArepY} and \citet[Theorem 4.3]{krengel1985ergodic} that the output process $\Y=\Y_{\bth_0}$ is ergodic. As a consequence, the pseudo\hyp{}innovations $\beps_{\bth}$ defined in \cref{eq-defpseudoinno} are ergodic for every $\bth\in\Theta$.

Our first identifiability assumption precludes redundancies in the parametrization of the state space models under consideration and is therefore necessary for the true parameter value $\bth_0$ to be estimated consistently. It will be used in \cref{lemma-uniqueminimumQ} to show that the quasi likelihood function given by \cref{eq-hatlikelihood} asymptotically has a unique global minimum at $\bth_0$.
\begin{assumptionparaD}
\label{assum-identifiability1}
For all $\bth_0\neq\bth\in\Theta$, there exists a $z\in\C$ such that
% \begin{subequations}
\begin{equation}
\label{eq-identifiability1filter}
H_{\bth}\left[\I_N-\left(F_{\bth}-K_{\bth}H_{\bth}\right)z\right]^{-1}K_{\bth}\neq H_{\bth_0}\left[\I_N-\left(F_{\bth_0}-K_{\bth_0}H_{\bth_0}\right)z\right]^{-1}K_{\bth_0},\quad\text{or}\quad V_{\bth}\neq V_{\bth_0}.
\end{equation}

\end{assumptionparaD}
\Cref{assum-identifiability1} can be rephrased in terms of the spectral densities $f_{\Y_{\bth}}$ of the output processes $\Y_{\bth}$ of the state space models $\left(F_{\bth},H_{\bth},\ZZ_{\bth},\W_{\bth}\right)$. This characterization will be very useful when we apply the estimation theory developed in this section to state space models that arise from sampling a continuous\hyp{}time ARMA process.
\begin{lemma}
\label{lemma-identifiabilityspectralDT}
If, for all $\bth_0\neq\bth\in\Theta$, there exists an $\omega\in[-\pi,\pi]$ such that $f_{\Y_{\bth}}(\omega)\neq f_{\Y_{\bth_0}}(\omega)$, then \cref{assum-identifiability1} holds.
\end{lemma}
\begin{proof}
We recall from \citet[Eq. (10.4.43)]{hamilton1994tsa} that the spectral density $f_{\Y_{\bth}}$ of the output process $\Y_{\bth}$ of the state space model $\left(F_{\bth},H_{\bth},\ZZ_{\bth},\W_{\bth}\right)$ is given by $f_{\Y_{\bth}}(\omega) =(2\pi)^{-1}\mathscr{H}_{\bth}\left(\ee^{\ii\omega}\right)V_{\bth}\mathscr{H}_{\bth}\left(\ee^{-\ii\omega}\right)^T$, $\omega\in[-\pi,\pi]$, where $\mathscr{H}_{\bth}(z)\coloneqq H_{\bth}\left[\I_N-\left(F_{\bth}-K_{\bth}H_{\bth}\right)z\right]^{-1}K_{\bth}+z$. If \cref{assum-identifiability1} does not hold, we have that both $\mathscr{H}_{\bth}(z)=\mathscr{H}_{\bth_0}(z)$ for all $z\in\C$, and $V_{\bth}=V_{\bth_0}$, and, consequently, that $f_{\Y_{\bth}}(\omega)=f_{\Y_{\bth_0}}(\omega)$, for all $\omega\in[-\pi,\pi]$, contradicting the assumption of the lemma.
\end{proof}

Under the assumptions described so far we obtain the following consistency result.

\begin{theorem}[Consistency of $\hat\bth^L$]
\label{theorem-consistency}
Assume that $\left(F_{\bth},H_{\bth},\ZZ_{\bth},\W_{\bth}\right)_{\bth\in\Theta}$ is a parametric family of state space models according to \cref{def-ssmgen}, and let $\y^L=(\Y_{\bth_0,1},\ldots,\Y_{\bth_0,L})$ be a sample of length $L$ from the output process of the model corresponding to $\bth_0$. If \cref{assum-compact,assum-smoothparam,assum-stability,assum-2moments,assum-identifiability1} hold, then the QML estimator $\hat\bth^L=\argmin_{\bth\in\Theta}\widehat{\mathscr{L}}(\bth,\y^L)$ is strongly consistent, i.\,e.\ $\hat\bth^L\to\bth_0$ almost surely, as $L\to\infty$.
\end{theorem}

We now describe the conditions which we need to impose in addition to \cref{assum-compact,assum-smoothparam,assum-stability,assum-2moments,assum-identifiability1} for the asymptotic normality of the QML estimator to hold. The first one excludes the case that the true parameter value $\bth_0$ lies on the boundary of the domain $\Theta$.
\begin{assumptionparaD}
\label{assum-interior}
The true parameter value $\bth_0$ is an element of the interior of $\Theta$. 
\end{assumptionparaD}
Next we need to impose a higher degree of smoothness than stated in \cref{assum-smoothparam} and a stronger moment condition than \cref{assum-2moments}.

\begin{assumptionparaD}
\label{assum-smoothparam3}
The mappings $F_{(\cdot)}$, $H_{(\cdot)}$, $Q_{(\cdot)}$, $S_{(\cdot)}$, and $R_{(\cdot)}$ in \cref{eq-paramSSM} are three times continuously differentiable.
\end{assumptionparaD}
By the results of the sensitivity analysis of the discrete\hyp{}time algebraic Riccati equation in \citet{sun1998sensitivity}, the same degree of smoothness, namely $C^3$, also carries over to the mapping $\bth\mapsto V_{\bth}$.
\begin{assumptionparaD}
\label{assum-4moments}
The process $\left(\begin{array}{cc} \W_{\bth_0}^T & \ZZ_{\bth_0}^T \end{array}\right)^T$ has finite $(4+\delta)$th moments for some $\delta>0$.
\end{assumptionparaD}
\Cref{assum-4moments} implies that the process $\Y$ has finite $(4+\delta)$th moments. In the definition of the general linear stochastic state space model and in \cref{assum-2moments}, it was only assumed that the sequences $\ZZ$ and $\W$ are stationary and ergodic. This structure alone does not entail a sufficient amount of asymptotic independence for results like \cref{theorem-hatbthLCLT} to be established. We assume that the process $\Y$ is strongly mixing in the sense of \citet{rosenblatt1956central}, and we impose a summability condition on the strong mixing coefficients, which is known to be sufficient for a \CLT\ for $\Y$ to hold \citep{ibragimov1962some,bradley2007introduction}.
\begin{assumptionparaD}
\label{assum-mixing}
Denote by $\alpha_{\Y}$ the strong mixing coefficients of the process $\Y=\Y_{\bth_0}$. There exists a constant $\delta>0$ such that $\sum_{m=0}^\infty\left[\alpha_{\Y}(m)\right]^{\frac{\delta}{2+\delta}}<\infty$.
\end{assumptionparaD}
In the case of exponential strong mixing, \cref{assum-mixing} is always satisfied, and it is no restriction to assume that the $\delta$ appearing in \cref{assum-4moments,assum-mixing} are the same. It has been shown in \citet{mokkadem1988mixing,schlemmmixing2010} that, because of the autoregressive structure of the state equation \labelcref{eq-eq-stateeq-QMLDT}, exponential strong mixing of the output process $\Y_{\bth_0}$ can be assured by imposing the condition that the process $\ZZ_{\bth_0}$ is an i.\,i.\,d.\ sequence whose marginal distributions possess a non\hyp{}trivial absolutely continuous component in the sense of Lebesgue's decomposition theorem.
% , see e.\,g., \citet[\S 31, Theorem C]{halmos1950measure}.
% or \citet{lebesgue1904une}.

Finally, we require another identifiability assumption, that will be used to ensure that the Fisher information matrix of the QML estimator is non\hyp{}singular. This is necessary because the asymptotic covariance matrix in the asymptotic normality result for $\hat\bth^L$ is directly related to the inverse of that matrix. \Cref{assum-identifiabilityFisher} is formulated in terms of the first derivative of the parametrization of the model, which makes it relatively easy to check in practice; the Fisher information matrix, in contrast, is related to the second derivative of the logarithmic Gaussian likelihood. For $j\in\N$ and $\bth\in\Theta$, the vector $\psi_{\bth,j}\in \R^{(j+2)d^2}$ is defined as
\begin{equation}
\label{eq-DefPsij}
\psi_{\bth,j} = \left(\begin{array}{c}
	  \left[\I_{j+1}\otimes K_{\bth}^T\otimes H_{\bth}\right]\left[\begin{array}{cccc}\left(\vec \I_N\right)^T& \left(\vec F_{\bth}\right)^T & \cdots & \left(\vec F_{\bth}^{j}\right)^T\end{array}\right]^T\\
	  \vec V_{\bth}
	  \end{array}\right),
\end{equation}
where $\otimes$ denotes the Kronecker product of two matrices, and $\vec$ is the linear operator that transforms a matrix into a vector by stacking its columns on top of each other.
\begin{assumptionparaD}
\label{assum-identifiabilityFisher}
%Denote by $\mathscr{M}_{\bth,j}=H_{\bth}F_{\bth}^jK_{\bth}$ the Markov matrices of the state space model \cref{eq-innorepSSM}. 
There exists an integer $j_0\in\N$ such that the $[(j_0+2) d^2]\times r$ matrix $\nabla_{\bth}\psi_{\bth_0,j_0}$ has rank $r$.
\end{assumptionparaD}

Our main result about the asymptotic distribution of the QML estimator for discrete\hyp{}time state space models is the following theorem. \Cref{eq-DefIJDT} shows in particular that this asymptotic distribution is independent of the choice of the initial values $\hat\X_{\bth,\text{initial}}$.

\begin{theorem}[Asymptotic normality of $\hat\bth^L$]
\label{theorem-hatbthLCLT}
Assume that $\left(F_{\bth},H_{\bth},\ZZ_{\bth},\W_{\bth}\right)_{\bth\in\Theta}$ is a parametric family of state space models according to \cref{def-ssmgen}, and let $\y^L=(\Y_{\bth_0,1},\ldots,\Y_{\bth_0,L})$ be a sample of length $L$ from the output process of the model corresponding to $\bth_0$. If \cref{assum-compact,assum-smoothparam,assum-smoothparam3,assum-stability,assum-2moments,assum-4moments,assum-interior,assum-mixing,assum-identifiability1,assum-identifiabilityFisher} hold, then the maximum likelihood estimator $\hat\bth^L=\argmin_{\bth\in\Theta}\widehat{\mathscr{L}}(\bth,\y^L)$ is asymptotically normally distributed with covariance matrix $\Xi=J^{-1}IJ^{-1}$, i.\,e.\
\begin{equation}
\label{eq-hatbthLCLT}
\sqrt{L}\left(\hat\bth^L-\bth_0\right)\xrightarrow[L\to\infty]{d}\mathscr{N}(\bzero,\Xi),
\end{equation}
where
\begin{equation}
\label{eq-DefIJDT}
I = \lim_{L\to\infty}L^{-1}\Var\left(\nabla_{\bth}\mathscr{L}\left(\bth_0,\y^L\right)\right),\quad J = \lim_{L\to\infty}L^{-1}\nabla^2_{\bth}\mathscr{L}\left(\bth_0,\y^L\right).
\end{equation}

\end{theorem}

Note that $I$ and $J$, which give the asymptotic covariance matrix $\Xi$ of the estimators, are deterministic and only depend on the true parameter value $\bth_0$. The matrix $J$ actually is the Fisher information and an alternative expression for $J$ can be found in \cref{lemma-convJ}. Despite being deterministic, the asymptotic variance $\Xi$ is  not immediate to obtain and needs to be estimated, as usually in connection with QML estimators. This is a non-trivial task and a detailed analysis of this is beyond the scope of the present paper, but worthy of consideration in more  detail in future work. However, it should be noted that when $\widehat\Xi^L$ is a consistent estimator for $\Xi$, then \cref{theorem-hatbthLCLT} implies that $\sqrt{L}(\widehat\Xi^L)^{-1/2}\left(\hat\bth^L-\bth_0\right)\xrightarrow[L\to\infty]{d}\mathscr{N}(\bzero,\I_r)$. Observe that no stable convergence in law (in the sense originally introduced by \cite{Renyi1963}) is needed to obtain the latter result for our QML estimator, as this stronger convergence concept is needed only when the limiting variance in a ``mixed normal limit theorem'' is random. 

In practice,  estimating the asymptotic covariance matrix $\Xi$ is important in order to construct confidence regions for the estimated parameters or in performing statistical tests. The problem of estimating it has also been considered in the framework of estimating weak VARMA processes in \citet{mainassara2011estimating} where the following procedure has been suggested, which is also applicable in our set\hyp{}up. First, $J(\bth_0)$ is estimated consistently by $\hat J^L=L^{-1}\nabla^2\widehat{\mathscr{L}}_{\bth}\left(\hat\bth^L,\y^L\right)$. For the computation of $\hat J^L$ we rely on the fact that the Kalman filter cannot only be used to evaluate the Gaussian log-likelihood of a state space model but also its gradient and Hessian. The most straightforward
% , but computationally burdensome
way of achieving this is by direct differentiation of the Kalman filter equations, which results in increasing the number of passes through the filter to $r+1$ and $r(r+3)/2$ for the gradient and the Hessian, respectively.
% More sophisticated algorithms, including the Kalman smoother and/or the backward filter have been devised and can be found in \citet{kulikova2006score,segal1989new}.
The construction of a consistent estimator of $I=I(\bth_0)$ is based on the observation that $I=\sum_{\Delta\in\Z}{\Cov(\ell_{\bth_0,n},\ell_{\bth_0,n+\Delta})}$, where $\ell_{\bth_0,n} = \nabla_{\bth}\left[\log\det V_{\bth_0}+\beps_{\bth_0,n}^T V_{\bth_0}^{-1}\beps_{\bth_0,n}\right]$. Assuming that  $(\ell_{\bth_0,n})_{n\in\N^+}$ admits an infinite\hyp{}order AR representation $\Phi(\BSO)\ell_{\bth_0,n}=\boldsymbol{U}_n$, where $\Phi(z)=\I_r+\sum_{i=1}^\infty{\Phi_iz^i}$ and $(\boldsymbol{U}_n)_{n\in\N^+}$ is a weak white noise with covariance matrix $\Sigma_{\boldsymbol{U}}$, it follows from the interpretation of $I/(2\pi)$ as the value of the spectral density of $(\ell_{\bth_0,n})_{n\in\N^+}$ at frequency zero that $I$ can also be written as $I=\Phi^{-1}(1)\Sigma_{\boldsymbol{U}}\Phi(1)^{-1}$. The idea is to fit a long autoregression to $(\ell_{\hat\bth^L,n})_{n=1,\ldots L}$, the empirical counterparts of $(\ell_{\bth_0,n})_{n\in\N^+}$ which are defined by replacing $\bth_0$ with the estimate $\hat\bth^L$ in the definition of $\ell_{\bth_0,n}$. This is done by choosing an integer $s>0$, and performing a least-squares regression of $\ell_{\hat\bth^L,n}$ on $\ell_{\hat\bth^L,n-1},\ldots,\ell_{\hat\bth^L,n-s}$, $s+1\leq n\leq L$. Denoting by $\hat\Phi_s^L(z)=\I_r+\sum_{i=1}^s{\hat\Phi_{i,s}^L z^i}$ the obtained empirical autoregressive polynomial and by $\hat\Sigma_s^L$ the empirical covariance matrix of the residuals of the regression, it was claimed in \citet[Theorem 4]{mainassara2011estimating} that under the additional assumption $\E\left[\left\|\beps_n\right\|^{8+\delta}\right]<\infty$ the spectral estimator $\hat I_s^L = \left(\hat\Phi_s^L(1)\right)^{-1}\hat\Sigma_s^L\left(\hat\Phi_s^L(1)\right)^{T,-1}$ converges to $I$ in probability as $L,s\to\infty$ if $s^3/L\to 0$. The covariance matrix of $\hat\bth^L$ is then estimated consistently as
\begin{equation}
\label{asympCov}
\widehat\Xi_s^L = \frac{1}{L}\left(\hat J^L\right)^{-1}\hat I_s^L\left(\hat J^L\right)^{-1}.
\end{equation}

In the simulation study performed in \cref{section-Simulation}, we estimate the covariance  matrix $\Xi$ of the estimators in the way just describe. From a comparison with the standard deviations of the estimators obtained from the simulations it can be seen that the approach  performs convincingly.

A possible alternative approach to estimate the asymptotic covariance matrix $\Xi$ may also be the use of bootstrap techniques. However, it seems that to this end the existing bootstrapping techniques need to be extended considerably (cf. \citet{BrockwellKreissNiebuhr2012}).

\subsection[Proof of strong consistency]{Proof of Theorem \ref{theorem-consistency} -- Strong consistency}
\label{section-QMLDTSSMconsistency}

In this section we prove the strong consistency of the QML estimator $\hat\bth^L$. 

The standard idea why the QML (or sometimes also \emph{Gaussian} maximum likelihood) estimators work in a linear time series/state space model setting is that the QML approach basically is very close to estimating the parameters using the spectral density which is in turn in a one-to-one relation with the second moment structure (see e.g. \citet[Chapter 10]{brockwell1991tst}). The reason is, of course, that a Gaussian process is completely characterized by the mean and autocovariance function.  So as soon as one knows that the parameters to be estimated are identifiable from the autocovariance function (and the mean) and the process is known to be ergodic, the QML estimators should be strongly consistent. Despite this simple standard idea, the upcoming actual proof of the strong consistency is lengthy as well as technical and consists of the following steps:

\begin{enumerate}
 \item When we use the the Kalman filter with fixed parameters $\bth$ on the finite sample $\y^L$, the obtained pseudo-innovations $\hat\beps_{\bth}$ approximate the true pseudo-innovations $\beps_{\bth}$ (obtainable from the steady state Kalman filter in theory) well; see \cref{lemma-propinnovations}.
\item The quasi likelihood (QL) function $\widehat{\mathscr{L}}$ obtained from the finite sample $\y^L$ (via $\hat\beps_{\bth}$) converges for the sample size $L\to\infty$ uniformly in the parameter space to the true QL function $\mathscr{L}$ (obtained from the  pseudo-innovations $\beps_{\bth}$); see \cref{lemma-likelihoodequiv}.
\item As the number $L$ of observation grows, the QL function  $\widehat{\mathscr{L}}$ divided by $L$ converges to the expected QL function $\mathscr{Q}$ uniformly in the parameter space; see \cref{lemma-uniformconvergence}.
\item The expected QL function $\mathscr{Q}$ has a unique minimum at the true parameter $\bth_0$; see \cref{lemma-identifiability1,lemma-uniqueminimumQ}.
\item The QL function  $\widehat{\mathscr{L}}$ divided by the number of observations evaluated at its minimum in the parameter space (i.e., at the QML estimator) converges almost surely to the expected QL function $\mathscr{Q}$ evaluated at the true parameter $\bth_0$ (\emph{its} minimum).
\item Finally, one can show that also the argumentof the minimum of the QL function $\widehat{\mathscr{L}}$ (i.e. the QML estimators) converges for $L\to\infty$ to $\bth_0$, which proves the strong consistency.
\end{enumerate}

As a first step we show that the stationary pseudo\hyp{}innovations processes defined by the steady-state Kalman filter are uniformly approximated by their counterparts based on the finite sample $\y^L$.
\begin{lemma}
\label{lemma-propinnovations}
Under \cref{assum-compact,assum-smoothparam,assum-stability}, the pseudo\hyp{}innovations sequences $\beps_{\bth}$ and $\hat\beps_{\bth}$ defined by the Kalman filter equations \labelcref{eq-innoSSM,eq-innohatSSM} have the following properties.
\begin{enumerate}[i)]
 \item\label{lemma-propinnovationsexpconv}If the initial values $\hat\X_{\bth,\text{initial}}$ are such that $\sup_{\bth\in\Theta}\left\|\hat\X_{\bth,\text{initial}}\right\|$ is almost surely finite, then, with probability one, there exist a positive number $C$ and a positive number $\rho<1$, such that $\sup_{\bth\in\Theta}\left\|\beps_{\bth,n}-\hat\beps_{\bth,n}\right\|\leq C\rho^n$, $n\in\N$.
In particular, $\hat\beps_{\bth_0,n}$ converges to the true innovations $\beps_n=\beps_{\bth_0,n}$ at an exponential rate.
\item\label{lemma-propinnovations-expbounded} The sequences $\beps_{\bth}$ are linear functions of $\Y$, i.\,e.\ there exist matrix sequences $\left(c_{\bth,\nu}\right)_{\nu\geq 1}$, such that $\beps_{\bth,n} = \Y_n + \sum_{\nu=1}^\infty{c_{\bth,\nu}}\Y_{n-\nu}$. The matrices $c_{\bth,\nu}$ are uniformly exponentially bounded, i.\,e.\ there exist a positive constant $C$ and a positive constant $\rho<1$, such that $\sup_{\bth\in\Theta}\left\|c_{\bth,\nu}\right\|\leq C \rho^\nu$, $\nu\in\N$.
\end{enumerate}
\end{lemma}
\begin{proof}
We first prove part \labelcref{lemma-propinnovationsexpconv} about the uniform exponential approximation of $\beps$ by $\hat\beps$. Iterating the Kalman equations \labelcref{eq-innoSSM,eq-innohatSSM}, we find that, for $n\in\N$,
\begin{align*}
\beps_{\bth,n} =& \Y_n-H_{\bth}\left(F_{\bth}-K_{\bth}H_{\bth}\right)^{n-1}\hat\X_{\bth,1}-\sum_{\nu=1}^{n-1}{H_{\bth} \left(F_{\bth}-K_{\bth}H_{\bth}\right)^{\nu-1}K_{\bth}\Y_{n-\nu}},\quad\text{and}\\
\hat\beps_{\bth,n} =& \Y_n-H_{\bth}\left(F_{\bth}-K_{\bth}H_{\bth}\right)^{n-1}\hat\X_{\bth,\text{initial}}-\sum_{\nu=1}^{n-1}{H_{\bth} \left(F_{\bth}-K_{\bth}H_{\bth}\right)^{\nu-1}K_{\bth}\Y_{n-\nu}}.
\end{align*}
Thus, using the fact that, by \cref{lemma-assum123consequences}, the spectral radii of $F_{\bth}-K_{\bth}H_{\bth}$ are bounded by $\rho<1$, it follows that
\begin{align*}
\sup_{\bth\in\Theta}\left\|\beps_{\bth,n}-\hat\beps_{\bth,n}\right\| =& \sup_{\bth\in\Theta}\left\|H_{\bth}\left(F_{\bth}-K_{\bth}H_{\bth}\right)^{n-1}(\X_{\bth,0}-\X_{\bth,\text{initial}})\right\| \leq \left\|H\right\|_{L^\infty(\Theta)}\rho^{n-1}\sup_{\bth\in\Theta}\left\|\X_{\bth,0}-\X_{\bth,\text{initial}}\right\|,
\end{align*}
where $\left\|H\right\|_{L^\infty(\Theta)}\coloneqq \sup_{\bth\in\Theta}\left\|H_{\bth}\right\|$ denotes the supremum norm of $H_{(\cdot)}$, which is finite by the Extreme Value Theorem. Since the last factor is almost surely finite by assumption, the claim follows. For part \labelcref{lemma-propinnovations-expbounded}, we observe that \cref{eq-innoSSM} and \cref{lemma-assum123consequences}, \labelcref{lemma-assum123consequences-FminusKH} imply that $\beps_{\bth}$ has the infinite\hyp{}order moving average representation $\beps_{\bth,n} = \Y_n-{H_{\bth}\sum_{\nu=1}^{\infty}\left(F_{\bth}-K_{\bth}H_{\bth}\right)^{\nu-1}K_{\bth}\Y_{n-\nu}}$, whose coefficients $c_{\bth,\nu}\coloneqq - H_{\bth} \left(F_{\bth}-K_{\bth}H_{\bth}\right)^{\nu-1}K_{\bth}$ are uniformly exponentially bounded. Explicitly, $\left\|c_{\bth.\nu}\right\|\leq \left\|H\right\|_{L^\infty(\Theta)}\left\|K\right\|_{L^\infty(\Theta)}\rho^{n-1}$. This completes the proof.
\end{proof}

\begin{lemma}
\label{lemma-likelihoodequiv}
Let $\mathscr{L}$ and $\widehat{\mathscr{L}}$ be given by \cref{eq-likelihood,eq-hatlikelihood}. If \cref{assum-compact,assum-smoothparam,assum-stability} are satisfied, then the sequence $L^{-1}\sup_{\bth\in\Theta}{\left|\widehat{\mathscr{L}}(\bth,\y^L)-\mathscr{L}(\bth,\y^L)\right|}$ converges to zero almost surely, as $L\to\infty$.
\end{lemma}
\begin{proof}
We first observe that
\begin{equation*}
\left|\widehat{\mathscr{L}}(\bth,\y^L)-\mathscr{L}(\bth,\y^L)\right| = \sum_{n=1}^L{\left[\left(\hat\beps_{\bth,n}-\beps_{\bth,n}\right)^T V_{\bth}^{-1}\hat\beps_{\bth,n}+\beps_{\bth,n}^T V_{\bth}^{-1}\left(\hat\beps_{\bth,n}-\beps_{\bth,n}\right)\right]}.
\end{equation*}
The fact that, by \cref{lemma-assum123consequences}, \labelcref{lemma-assum123consequences-V}, there exists a constant $C$ such that $\left\|V_{\bth}^{-1}\right\|\leq C$ implies that
\begin{align}
\label{eq-supLestimate}
\frac{1}{L}\sup_{\bth\in\Theta}{\left|\widehat{\mathscr{L}}(\bth,\y^L)-\mathscr{L}(\bth,\y^L)\right|}\leq& \frac{C}{L}\sum_{n=1}^L{\rho^n\left[\sup_{\bth\in\Theta}{\left\|\hat\beps_{\bth,n}\right\|}+\sup_{\bth\in\Theta}{\left\|\beps_{\bth,n}\right\|}\right]}.
%   \leq & \frac{2C}{L}\sum_{n=1}^L{\rho^n\sup_{\bth\in\Theta}{\left\|\beps_{\bth,n}\right\|}}+o(L).
\end{align}
\Cref{lemma-propinnovations}, \labelcref{lemma-propinnovations-expbounded} and the assumption that $\Y$ has finite second moments imply that $\E\sup_{\bth\in\Theta}\left\|\beps_{\bth,n}\right\|$ is finite. Applying Markov's inequality, one sees that, for every positive $\epsilon$,
\begin{equation*}
\sum_{n=1}^\infty{\Pb\left(\rho^n\sup_{\bth\in\Theta}\left\|\beps_{\bth,n}\right\|\geq\epsilon\right)}\leq\E\sup_{\bth\in\Theta}\left\|\beps_{\bth,1}\right\|\sum_{n=1}^\infty{\frac{\rho^n}{\epsilon}}<\infty,
\end{equation*}
because $\rho<1$. The Borel--Cantelli Lemma shows that $\rho^n\sup_{\bth\in\Theta}\left\|\beps_{\bth,n}\right\|$ converges to zero almost surely, as $n\to\infty$. In an analogous way one can show that $\rho^n\sup_{\bth\in\Theta}\left\|\hat\beps_{\bth,n}\right\|$ converges to zero almost surely, and, consequently, so does the Ces\`aro mean in \cref{eq-supLestimate}. The claim thus follows.
\end{proof}

\begin{lemma}
\label{lemma-uniformconvergence}
If \cref{assum-compact,assum-smoothparam,assum-stability,assum-2moments} hold, then, with probability one, the sequence of random functions $\bth\mapsto L^{-1}\widehat{\mathscr{L}}(\bth,\y^L)$ converges, as $L$ tends to infinity, uniformly in $\bth$ to the limiting function $\mathscr{Q}:\Theta\to\R$ defined by
\begin{equation}
\label{eq-DefQ}
\mathscr{Q}(\bth) = d\log(2\pi) + \log\det V_{\bth} + \E \beps_{\bth,1}^T V_{\bth}^{-1}\beps_{\bth,1}.
\end{equation} 
\end{lemma}
\begin{proof}
In view of the approximation results in \cref{lemma-likelihoodequiv}, it is enough to show that the sequence of random functions $\bth\mapsto L^{-1}\mathscr{L}(\bth,\y^L)$ converges uniformly to $\mathscr{Q}$. The proof of this assertion is based on the observation following \cref{assum-2moments} that for each $\bth\in\Theta$ the sequence $\beps_{\bth}$ is ergodic and its consequence that, by Birkhoff's \ET\ \citep[Theorem 6.2.1]{durrett2010probability}, the sequence $L^{-1}\mathscr{L}(\bth,\y^L)$ converges to $\mathscr{Q}(\bth)$ point\hyp{}wise. The stronger statement of uniform convergence follows from \cref{assum-compact} that $\Theta$ is compact by an argument analogous to the proof of \citet[Theorem 16]{ferguson1996acourse}.
\end{proof}
\begin{lemma}
\label{lemma-identifiability1} 
Assume that \cref{assum-stability,assum-2moments} as well as the first alternative of \cref{assum-identifiability1} hold. If $\beps_{\bth,1}=\beps_{\bth_0,1}$ almost surely, then $\bth=\bth_0$.
\end{lemma}
\begin{proof}
Assume, for the sake of contradiction, that $\bth\neq\bth_0$. By \cref{assum-identifiability1}, there exist matrices $C_j\in M_d(\R)$, $j\in\N_0$, such that, for $|z|\leq 1$,
\begin{equation}
\label{eq-difftransferfunctions}
H_{\bth}\left[\I_N-(F_{\bth}-K_{\bth}H_{\bth})z\right]^{-1}K_{\bth} - H_{\bth_0}\left[\I_N-(F_{\bth_0}-K_{\bth_0}H_{\bth_0}z\right]^{-1}K_{\bth_0} = \sum_{j=j_0}^\infty C_j z^j,
\end{equation}
where $C_{j_0}\neq 0$, for some $j_0\geq0$. Using \cref{eq-innoSSMpolynomial} and the assumed equality of $\beps_{\bth,1}$ and $\beps_{\bth_0,1}$, this implies that $\bzero_d=\sum_{j=j_0}^\infty{C_j \Y_{j_0-j}}$ almost surely; in particular, the random variable $C_{j_0}\Y_0$ is equal to a linear combination of the components of $\Y_n$, $n < 0$. It thus follows from the interpretation of the innovations sequence $\beps_{\bth_0}$ as linear prediction errors for the process $\Y$ that $C_{j_0}\beps_{\bth_0,0}$ is equal to zero, which implies that $\E C_{j_0}\beps_{\bth_0,0}\beps_{\bth_0,0}^TC_{j_0}^T = C_{j_0}V_{\bth_0}C_{j_0}^T = 0_d$. Since $V_{\bth_0}$ is assumed to be non\hyp{}singular, this implies that the matrix $C_{j_0}$ is the null matrix, a contradiction to \cref{eq-difftransferfunctions}.
\end{proof}

\begin{lemma}
\label{lemma-uniqueminimumQ}
Under \cref{assum-compact,assum-smoothparam,assum-stability,assum-identifiability1}, the function $\mathscr{Q}:\Theta\to\R$, as defined in \cref{eq-DefQ}, has a unique global minimum at $\bth_0$.
\end{lemma}
\begin{proof}
We first observe that the difference $\beps_{\bth,1} - \beps_{\bth_0,1}$ is an element of the Hilbert space spanned by the random variables $\{\Y_n,n\leq 0\}$, and that $\beps_{\bth_0,1}$ is, by definition, orthogonal to this space. Thus, the expectation $\E \left(\beps_{\bth,1} - \beps_{\bth_0,1}\right)^T V_{\bth}^{-1}\beps_{\bth_0,1}$ is equal to zero and, consequently, $\mathscr{Q}(\bth)$ can be written as
\begin{equation*}
\mathscr{Q}(\bth) = d\log(2\pi) + \E \beps_{\bth_0,1}^T V_{\bth}^{-1} \beps_{\bth_0,1} + \E\left(\beps_{\bth,1} - \beps_{\bth_0,1}\right)^T V_{\bth}^{-1} \left(\beps_{\bth,1} - \beps_{\bth_0,1}\right) + \log\det V_{\bth}.
\end{equation*}
In particular, since $\E \beps_{\bth_0,1}^T V_{\bth_0}^{-1} \beps_{\bth_0,1}=\tr\left[V_{\bth_0}^{-1}\E\beps_{\bth_0,1}\beps_{\bth_0,1}^T\right]=d$, it follows that $\mathscr{Q}(\bth_0) = \log\det V_{\bth_0}+d(1+\log(2\pi))$. The elementary inequality $x-\log x\geq 1$, for $x>0$, implies that $\tr M-\log\det M\geq d$ for all symmetric positive definite $d\times d$ matrices $M\in\SS^{++}_d(\R)$ with equality if and only if $M=\I_d$. Using this inequality for $M=V_{\bth_0}^{-1}V_{\bth}$, we thus obtain that, for all $\bth\in\Theta$,
\begin{align*}
\mathscr{Q}(\bth) - \mathscr{Q}(\bth_0) = &d + \tr\left[V_{\bth}^{-1}\E\beps_{\bth_0,1}\beps_{\bth_0,1}^T\right]  - \log\det \left(V_{\bth_0}^{-1}V_{\bth}\right) \\*
  & +\E\left(\beps_{\bth,1} - \beps_{\bth_0,1}\right)^T V_{\bth}^{-1} \left(\beps_{\bth,1} - \beps_{\bth_0,1}\right) - \E \beps_{\bth_0,1}^T V_{\bth_0}^{-1} \beps_{\bth_0,1} \\
  \geq & \E\left(\beps_{\bth,1} - \beps_{\bth_0,1}\right)^T V_{\bth}^{-1} \left(\beps_{\bth,1} - \beps_{\bth_0,1}\right) \geq  0.
\end{align*}
It remains to argue that this chain of inequalities is in fact a strict inequality if $\bth\neq\bth_0$. If $V_{\bth}\neq V_{\bth_0}$, the first inequality is strict, and we are done. If $V_{\bth} = V_{\bth_0}$, the first alternative of \cref{assum-identifiability1} is satisfied. The second inequality is an equality if and only if $\beps_{\bth,1}=\beps_{\bth_0,1}$ almost surely, which, by \cref{lemma-identifiability1}, implies that $\bth=\bth_0$. Thus, the function $\mathscr{Q}$ has a unique global minimum at $\bth_0$.
\end{proof}

\begin{proof}[Proof of \cref{theorem-consistency}]
We shall first show that the sequence $L^{-1}\widehat{\mathscr{L}}(\hat\bth^L,\y^L)$, $L\in\N$, converges almost surely to the deterministic number $\mathscr{Q}(\bth_0)$ as the sample size $L$ tends to infinity. Assume that, for some positive number $\epsilon$, it holds that $\sup_{\bth\in\Theta}\left|L^{-1}\widehat{\mathscr{L}}(\bth,\y^L)-\mathscr{Q}(\bth)\right|\leq\epsilon$. It then follows that 
\begin{align*}
L^{-1}\widehat{\mathscr{L}}(\hat\bth^L,\y^L) \leq L^{-1}\widehat{\mathscr{L}}(\bth_0,\y^L)\leq \mathscr{Q}(\bth_0)+\epsilon\quad\text{and}\quad L^{-1}\widehat{\mathscr{L}}(\hat\bth^L,\y^L)\geq \mathscr{Q}(\hat\bth^L)-\epsilon \geq \mathscr{Q}(\bth_0) - \epsilon,
\end{align*}
where it was used that $\hat\bth^L$ is defined to minimize $\widehat{\mathscr{L}}(\cdot,\y^L)$ and that, by \cref{lemma-uniqueminimumQ}, $\bth_0$ minimizes $\mathscr{Q}(\cdot)$. In particular, it follows that $\left|L^{-1}\widehat{\mathscr{L}}(\hat\bth^L,\y^L)-\mathscr{Q}(\bth_0)\right|\leq\epsilon$. This observation and \cref{lemma-uniformconvergence} immediately imply that
\begin{equation}
\label{eq-convLhatbthObth0}
\Pb\left(\frac{1}{L}\widehat{\mathscr{L}}(\hat\bth^L,\y^L)\xrightarrow[L\to\infty]{}\mathscr{Q}(\bth_0)\right) \geq \Pb\left(\sup_{\bth\in\Theta}\left|\frac{1}{L}\widehat{\mathscr{L}}(\bth,\y^L)-\mathscr{Q}(\bth)\right|\xrightarrow[L\to\infty]{}0\right) = 1.
\end{equation}
To complete the proof of the theorem, it suffices to show that, for every neighbourhood $U$ of $\bth_0$, with probability one, $\hat\bth^L$ will eventually lie in $U$.  For every such neighbourhood $U$ of $\bth_0$, we define the real number $\delta(U)\coloneqq\inf_{\bth\in\Theta\backslash U}\mathscr{Q}(\bth)-\mathscr{Q}(\bth_0)$, which is strictly positive by \cref{lemma-uniqueminimumQ}. Then the following sequence of inequalities holds:
\begin{align*}
&\Pb\left(\hat\bth^L\xrightarrow[L\to\infty]{}\bth_0\right) = \Pb\left(\forall U\,\exists L_0:\hat\bth^L\in U\quad\forall L>L_0\right)\\
  \geq& \Pb\left(\forall U\,\exists L_0:\mathscr{Q}(\hat\bth^L)-\mathscr{Q}(\bth_0)<\delta(U)\quad\forall L>L_0\right)\\
  \geq& \Pb\left(\forall U\,\exists L_0:\left|L^{-1}\widehat{\mathscr{L}}(\hat\bth^L,\y^L)-\mathscr{Q}(\bth_0)\right|<\delta(U)/2 \quad\text{and}\quad \left|L^{-1}\widehat{\mathscr{L}}(\hat\bth^L,\y^L)-\mathscr{Q}(\hat\bth^L)\right|<\delta(U)/2 \quad\forall L>L_0\right)
%   \geq & \Pb\left(\forall U\,\exists L_0:\left|L^{-1}\widehat{\mathscr{L}}(\hat\bth^L,\y^L)-\mathscr{Q}(\bth_0)\right|<\delta(U)/2\quad\forall L>L_0\right) + \Pb\left(\forall U\,\exists L_0:\sup_{\bth\in\Theta}\left|L^{-1}\widehat{\mathscr{L}}(\bth,\y^L)-\mathscr{Q}(\bth)\right|<\delta(U)/2 \quad\forall L>L_0\right) - 1.
\end{align*}
The last probability is equal to one by \cref{eq-convLhatbthObth0,lemma-uniformconvergence}.
\end{proof}

\subsection[Proof of asymptotic normality]{Proof of Theorem \ref{theorem-hatbthLCLT} -- Asymptotic normality}
\label{section-QMLDTSSMnormality}

In this section we prove the assertion of \cref{theorem-hatbthLCLT}, that the distribution of $L^{1/2}\left(\hat\bth^L-\bth_0\right)$ converges to a normal random variable with mean zero and covariance matrix $\Xi=J^{-1}IJ^{-1}$, an expression for which is given in \cref{eq-DefIJDT}. 

The idea behind the proof of the asymptotic normality essentially is that the strong mixing property implies various central limit theorems. As already said, the QML estimators are intuitively close to moment based estimators. So the main task is to show that the central limit results translate into asymptotic normality of the estimators. The individual steps in the following again lengthy and technical proof are:
\begin{enumerate}
 \item First we extend the result that the  pseudo-innovations $\hat\beps_{\bth}$ obtained via the Kalman filter from the finite sample $\y^L$ approximate the true pseudo-innovations $\beps_{\bth}$ (obtainable from the steady state Kalman filter in theory) well to their first and second derivatives; see \cref{lemma-propinnovationspartial}.
\item The first derivatives of the QL function $\mathscr{L}$ obtained from the pseudo-innovations $\beps_{\bth}$ have a finite variance for every possible parameter $\bth$; see \cref{lemma-partialLfinitevariance}.
\item Certain fourth moments (viz. covariances of scalar products of the vectors of values of the process  at different times) of a strongly mixing process with  $4+\delta$ finite moments can be uniformly bounded using the strong mixing coefficients; see \cref{lemma-davydov}.
\item The covariance matrix of the gradients of the QL function $\mathscr{L}$ divided by the number of observations converges for every possible parameter $\bth$; see \cref{lemma-convIbth}.
\item The result that the quasi likelihood (QL) function $\widehat{\mathscr{L}}$ obtained from the finite sample $\y^L$ (via $\hat\beps_{\bth}$) converges for the sample size $L\to\infty$ uniformly in the parameter space to the true QL function $\mathscr{L}$ (obtained from the  pseudo-innovations $\beps_{\bth}$) is extended to the first and second derivatives; see \cref{lemma-likelihoodequivpartial}.
\item The previous steps allow to show that the QL function $\widehat{\mathscr{L}}$ at the true parameter $\bth_0$  divided by the number of observations is asymptotically normal with limiting variance determined in step 4; see \cref{lemma-asymnormaldeltaLmixingY}.
\item The limit of the rescaled second derivative of the QL function $\widehat{\mathscr{L}}$ at the true parameter exists, equals the Fisher information and is invertible; see \cref{lemma-convJ}.
\item A zeroth order Taylor expansion of the gradient of the QL function $\widehat{\mathscr{L}}$ divided by the number of observations at the true parameter $\bth_0$  is combined with the asymptotic normality result of step 4 and the already established strong consistency of the QML estimator. Using the third derivatives of $\widehat{\mathscr{L}}$, the error of the Taylor approximation expressed in terms of second derivatives of $\widehat{\mathscr{L}}$ is controlled and using the result of step 7 the asymptotic normality of the QML estimator is deduced.
\end{enumerate}

First, we collect basic properties of $\partial_m\beps_{\bth,n}$ and $\partial_m\hat\beps_{\bth,n}$, where $\partial_m=\partial/\partial\vartheta^m$ denotes the partial derivative with respect to the $m$th component of $\bth$; the following lemma mirrors \cref{lemma-propinnovations}.

\begin{lemma}
\label{lemma-propinnovationspartial}
If \cref{assum-compact,assum-smoothparam,assum-smoothparam3,assum-stability} hold, the pseudo\hyp{}innovations sequences $\beps_{\bth}$ and $\hat\beps_{\bth}$ defined by the Kalman filter equations \labelcref{eq-innoSSM,eq-innohatSSM} have the following properties. 
\begin{enumerate}[i)]
 \item\label{lemma-propinnovationspartialexpconv}If, for an integer $k\in\{1,\ldots,r\}$, the initial values $\hat\X_{\bth,\text{initial}}$ are such that both $\sup_{\bth\in\Theta}\left\|\hat\X_{\bth,\text{initial}}\right\|$ and $\sup_{\bth\in\Theta}\left\|\partial_k\hat\X_{\bth,\text{initial}}\right\|$ are almost surely finite, then, with probability one, there exist positive numbers $C$ and $\rho<1$, such that $\sup_{\bth\in\Theta}\left\|\partial_k\beps_{\bth,n}-\partial_k\hat\beps_{\bth,n}\right\|\leq C\rho^n$, $n\in\N$.
 \item\label{lemma-propinnovationspartialexpbounded} For each $k\in\{1,\ldots,r\}$, the random sequences $\partial_k\beps_{\bth}$ are linear functions of $\Y$, i.\,e.\ there exist matrix sequences  $\left(c^{(k)}_{\bth,\nu}\right)_{\nu\geq 1}$, such that $\partial_k\beps_{\bth,n} = \sum_{\nu=1}^\infty c^{(k)}_{\bth,\nu}\Y_{n-\nu}$. The matrices $c^{(k)}_{\bth,\nu}$ are uniformly exponentially bounded, i.\,e.\ there exist positive numbers $C$ and $\rho<1$, such that $\sup_{\bth\in\Theta}\left\|c^{(k)}_{\bth,\nu}\right\|\leq C \rho^\nu$,$\nu\in\N$.
\item\label{lemma-propinnovationspartial2expconv}If, for integers $k,l\in\{1,\ldots,r\}$, the initial values $\hat\X_{\bth,\text{initial}}$ are such that $\sup_{\bth\in\Theta}\left\|\hat\X_{\bth,\text{initial}}\right\|$, as well as $\sup_{\bth\in\Theta}\left\|\partial_i\hat\X_{\bth,\text{initial}}\right\|$, $i\in\{k,l\}$, and $\sup_{\bth\in\Theta}\left\|\partial^2_{k,l}\hat\X_{\bth,\text{initial}}\right\|$ are almost surely finite, then, with probability one, there exist positive numbers $C$ and $\rho<1$, such that $\sup_{\bth\in\Theta}\left\|\partial^2_{k,l}\beps_{\bth,n}-\partial^2_{k,l}\hat\beps_{\bth,n}\right\|\leq C\rho^n$, $n\in\N$.
 \item\label{lemma-propinnovationspartial2expbounded} For each $k,l\in\{1,\ldots,r\}$, the random sequences $\partial^2_{k,l}\beps_{\bth}$ are linear functions of $\Y$, i.\,e.\ there exist matrix sequences  $\left(c^{(k,l)}_{\bth,\nu}\right)_{\nu\geq 1}$, such that $\partial^2_{k,l}\beps_{\bth,n} = \sum_{\nu=1}^\infty c^{(k,l)}_{\bth,\nu}\Y_{n-\nu}$. The matrices $c^{(k,l)}_{\bth,\nu}$ are uniformly exponentially bounded, i.\,e.\ there exist positive numbers $C$ and $\rho<1$, such that $\sup_{\bth\in\Theta}\left\|c^{(k,l)}_{\bth,\nu}\right\|\leq C \rho^\nu$, $\nu\in\N$.
\end{enumerate}
\end{lemma}
\begin{proof}
Analogous to the proof of \cref{lemma-propinnovations}, repeatedly interchanging differentiation and summation, and using the fact that, as a consequence of \cref{assum-compact,assum-smoothparam,assum-stability,assum-smoothparam3}, both $\partial_k\left[H_{\bth} \left(F_{\bth}-K_{\bth}H_{\bth}\right)^{\nu-1}K_{\bth}\right]$ and $\partial^2_{k,l}\left[H_{\bth} \left(F_{\bth}-K_{\bth}H_{\bth}\right)^{\nu-1}K_{\bth}\right]$ are uniformly exponentially bounded.
\end{proof}

\begin{lemma}
\label{lemma-partialLfinitevariance}
For each $\bth\in\Theta$ and every $m=1,\ldots,r$, the random variable $\partial_m\mathscr{L}(\bth,\y^L)$ has finite variance.
\end{lemma}
\begin{proof}
The claim follows from \cref{assum-4moments},  the exponential decay of the coefficient matrices $c_{\bth,\nu}$ and $c^{(m)}_{\bth,\nu}$ proved in \cref{lemma-propinnovations}, \labelcref{lemma-propinnovations-expbounded} and \cref{lemma-propinnovationspartial}, and the Cauchy--Schwarz inequality.
\end{proof}

We need the following covariance inequality which is a consequence of Davydov's inequality and the multidimensional generalization of an inequality used in the proof of \citet[Lemma 3]{francq1998}. For a positive real number $\alpha$, we denote by $\lfloor\alpha\rfloor$ the greatest integer smaller than or equal to $\alpha$.
\begin{lemma}
\label{lemma-davydov}
Let $\X$ be a strictly stationary, strongly mixing $d$\hyp{}dimensional stochastic process with finite $(4+\delta)$th moments for some $\delta>0$. Then there exists a constant $\kappa$, such that for all $d\times d$ matrices $A$, $B$, every $n\in\Z$, $\Delta\in\N$, and time indices $\nu, \nu'\in\N_0$, $\mu, \mu'=0,1\ldots,\lfloor \Delta/2\rfloor$, it holds that
\begin{equation}
\Cov\left(\X_{n-\nu}^TA\X_{n-\nu'};\X_{n+\Delta-\mu}^TB\X_{n+\Delta-\mu'}\right) \leq \kappa \left\|A\right\|\left\|B\right\| \left[\alpha_{\X}\left(\left\lfloor\frac{\Delta}{2}\right\rfloor\right)\right]^{\delta/(\delta+2)},
\end{equation}
where $\alpha_{\X}$ denote the strong mixing coefficients of the process $\X$.
\end{lemma}
\begin{proof}
We first note that the bilinearity of $\Cov(\cdot;\cdot)$ and the elementary inequality $M_{ij}\leq\left\|M\right\|$, $M\in M_d(\R)$, imply that
\begin{align*}
\Cov\left(\X_{n-\nu}^TA\X_{n-\nu'};\X_{n+\Delta-\mu}^TB\X_{n+\Delta-\mu'}\right) 
% =& \sum_{i,j,s,t=1}^d A_{ij}B_{st}\Cov\left(X_{n-\nu}^iX_{n-\nu'}^j;X_{n+\Delta-\mu}^s X_{n+\Delta-\mu'}^t\right)\\
\leq &d^4 \left\|A\right\| \left\|B\right\| \max_{i,j,s,t=1,\ldots,d}\Cov\left(X_{n-\nu}^iX_{n-\nu'}^j;X_{n+\Delta-\mu}^s X_{n+\Delta-\mu'}^t\right).
\end{align*}
Since the projection which maps a vector to one of its components is measurable, it follows that $X_{n-\nu}^iX_{n-\nu'}^j$ is measurable with respect to $\mathscr{F}_{-\infty}^{n-\min\{\nu,\nu'\}}$, the $\sigma$\hyp{}algebra generated by $\left\{\X_k:-\infty<k\leq n-\min\{\nu,\nu'\}\right\}$. Similarly, the random variable $X_{n+\Delta-\mu}^s X_{n+\Delta-\mu'}^t$ is measurable with respect to $\mathscr{F}_{n+\Delta-\max\left\{\mu,\mu'\right\}}^\infty$. Davydov's inequality \citep[Lemma 2.1]{davydov1968convergence} implies that there exists a universal constant $K$ such that
\begin{align*}
\Cov\left(X_{n-\nu}^iX_{n-\nu'}^j;X_{n+\Delta-\mu}^s X_{n+\Delta-\mu'}^t\right) \leq & K \left(\E\left| X^i_{n-\nu}  X^j_{n-\nu'}\right|^{2+\delta}\right)^{1/(2+\delta)} \left(\E\left|X^s_{n+\Delta-\mu}  X^t_{n+\Delta-\mu'}\right|^{2+\delta}\right)^{1/(2+\delta)}\\*
      &\qquad\qquad\times \left[\alpha_{\X}\left(\Delta-\max\left\{\mu,\mu'\right\}+\min\left\{\nu,\nu'\right\}\right)\right]^{\delta/(2+\delta)}\\
 \leq & \kappa \left[\alpha_{\X}\left(\left\lfloor\frac{\Delta}{2}\right\rfloor\right)\right]^{\delta/(2+\delta)},
\end{align*}
where it was used that $\Delta-\max\left\{\mu,\mu'\right\}+\min\left\{\nu,\nu'\right\}\geq \lfloor \Delta/2\rfloor$, and that strong mixing coefficients are non\hyp{}increasing. By the Cauchy--Schwarz inequality the constant $\kappa$ satisfies
\begin{equation*}
\kappa = K \left(\E\left| X^i_{n-\nu}  X^j_{n-\nu'}\right|^{2+\delta}\right)^{1/(2+\delta)} \left(\E\left|X^s_{n+\Delta-\mu}  X^t_{n+\Delta-\mu'}\right|^{2+\delta}\right)^{1/(2+\delta)}\leq K\left(\E\left\|\X_1\right\|^{4+2\delta}\right)^{\frac{2}{2+\delta}},
\end{equation*}
and thus does not depend on $n,\nu,\nu',\mu,\mu',\Delta$, nor on $i,j,s,t$.
\end{proof}

The next lemma is a multivariate generalization of \citet[Lemma 3]{francq1998}. In the proof of \citet[Lemma 4]{mainassara2011estimating} this generalization is used without providing details and, more importantly without imposing \cref{assum-mixing} about the strong mixing of $\Y$. In view of the derivative terms $\partial_m\beps_{\bth,n}$ in \cref{eq-Defellmbthn} it is not immediately clear how the result of the lemma  can be proved under the mere assumption of strong mixing of the innovations sequence $\beps_{\bth_0}$. We therefore think that a detailed account, properly generalizing the arguments in the original paper \citep{francq1998} to the multidimensional setting, is justified.
\begin{lemma}
\label{lemma-convIbth}
Suppose that \cref{assum-compact,assum-smoothparam,assum-stability,assum-4moments,assum-mixing} hold. Then, for every $\bth\in\Theta$, the sequence $L^{-1}\Var \nabla_{\bth}\mathscr{L}(\bth,\y^L)$ of deterministic matrices converges to a limit $I(\bth)$ as $L\to\infty$.
\end{lemma}
\begin{proof}
It is enough to show that, for each $\bth\in\Theta$, and all $k,l=1,\ldots,r$, the sequence of real\hyp{}valued random variables $I^{(k,l)}_{\bth,L}$, defined by
\begin{equation}
\label{eq-DefIlimn}
I^{(k,l)}_{\bth,L} = \frac{1}{L}\sum_{n=1}^L\sum_{t=1}^L\Cov\left(\ell^{(k)}_{\bth,n},\ell^{(l)}_{\bth,t}\right),
\end{equation}
converges to a limit as $L$ tends to infinity, where $\ell^{(m)}_{\bth,n} = \partial_m l_{\bth,n}$ is the partial derivative of the $n$th term in expression \labelcref{eq-likelihood} for $\mathscr{L}(\bth,\y^L)$. It follows from well\hyp{}known differentiation rules for matrix functions \citep[see, e.\,g.][Sections 6.5 and 6.6]{horn1994topics} that
\begin{equation}
\label{eq-Defellmbthn}
\ell^{(m)}_{\bth,n} = \tr\left[V_{\bth}^{-1}\left(\I_d - \beps_{\bth,n}\beps_{\bth,n}^T V_{\bth}^{-1}\right)\left(\partial_m V_{\bth}\right)\right] + 2\left(\partial_m\beps_{\bth,n}^T\right)V_{\bth}^{-1}\beps_{\bth,n}.
\end{equation}
By the assumed stationarity of the processes $\beps_{\bth}$, the covariances in the sum \labelcref{eq-DefIlimn} depend only on the difference $n-t$. For the proof of the lemma it suffices to show that the sequence $\mathfrak{c}^{(k,l)}_{\bth,\Delta}=\Cov\left(\ell^{(k)}_{\bth,n},\ell^{(l)}_{n+\Delta,\bth}\right)$, $\Delta\in\Z$, is absolutely summable for all $k,l=1,\ldots,r$, because then
\begin{equation}
\label{eq-convIklbthN}
I^{(k,l)}_{\bth,L}=\frac{1}{L}\sum_{\Delta=-L}^L\left(L-|\Delta|\right)\mathfrak{c}^{(k,l)}_{\bth,\Delta}\xrightarrow[L\to\infty]{}\sum_{\Delta\in\Z}\mathfrak{c}^{(k,l)}_{\bth,\Delta}<\infty.
\end{equation}
In view of the of the symmetry $\mathfrak{c}^{(k,l)}_{\bth,\Delta}=\mathfrak{c}^{(k,l)}_{\bth,-\Delta}$, it is no restriction to assume that $\Delta\in\N$. In order to show that $\sum_\Delta\left|\mathfrak{c}^{(k,l)}_{\bth,\Delta}\right|$ is finite, we first use the bilinearity of $\Cov(\cdot;\cdot)$ to estimate
\begin{align*}
\left|\mathfrak{c}^{(k,l)}_{\bth,\Delta}\right|\leq & 4\left|\Cov\left(\left(\partial_k\beps_{\bth,n}^T\right)V_{\bth}^{-1}\beps_{\bth,n} ;\left(\partial_l\beps_{\bth,n+\Delta}^T\right)V_{\bth}^{-1}\beps_{\bth,n+\Delta}  \right)\right|\\*
  &+ \left|\Cov\left(\tr\left[V_{\bth}^{-1}\beps_{\bth,n}\beps_{\bth,n}^TV_{\bth}^{-1}\partial_kV_{\bth}\right];\tr\left[V_{\bth}^{-1}\beps_{\bth,n+\Delta}\beps_{\bth,n+\Delta}^TV_{\bth}^{-1}\partial_lV_{\bth}\right]\right)\right| +\\*
  &+ 2\left|\Cov\left(\tr\left[V_{\bth}^{-1}\beps_{\bth,n}\beps_{\bth,n}^TV_{\bth}^{-1}\partial_kV_{\bth}\right];   \left(\partial_l\beps_{\bth,n+\Delta}^T\right)V_{\bth}^{-1}\beps_{\bth,n+\Delta}   \right)\right| +\\*
  &+  2\left|\Cov\left(\left(\partial_k\beps_{\bth,n}^T\right)V_{\bth}^{-1}\beps_{\bth,n} ;\tr\left[V_{\bth}^{-1}\beps_{\bth,n+\Delta}\beps_{\bth,n+\Delta}^TV_{\bth}^{-1}\partial_lV_{\bth}\right]  \right)\right|.
\end{align*}
Each of these four terms can be analysed separately. We give details only for the first one, the arguments for the other three terms being similar. Using the moving average representations for $\beps_{\bth}$, $\partial_k\beps_{\bth}$ and $\partial_l\beps_{\bth}$, it follows that
\begin{align*}
&\left|\Cov\left(\left(\partial_k\beps_{\bth,n}^T\right)V_{\bth}^{-1}\beps_{\bth,n} ;\left(\partial_l\beps_{\bth,n+\Delta}^T\right)V_{\bth}^{-1}\beps_{\bth,n+\Delta}  \right)\right|\\
=&\sum_{\nu,\nu',\mu,\mu'=0}^\infty\left|\Cov\left(\Y_{n-\nu}^Tc^{(k),T}_{\bth,\nu}V_{\bth}^{-1}c_{\bth,\nu'}\Y_{n-\nu'},\Y_{n+\Delta-\mu}^Tc^{(l),T}_{\bth,\mu}V_{\bth}^{-1}c_{\bth,\mu'}\Y_{n+\Delta-\mu'}\right)\right|.
%\leq& \left\|V_{\bth}^{-1}\right\|^2\sum_{\nu,\nu',\mu,\mu'=0}^\infty\left\|c^{(k)}_{\bth,\nu}\right\|\left\|c_{\bth,\nu'}\right\|\left\|c^{(l)}_{\bth,\mu'}\right\|\left\|c_{\bth,%\mu'}\right\|\left[\Cov\left(\left\|\Y_{n-\nu}\right\|\left\|\Y_{n-\nu'}\right\|,\left\|\Y_{n+r-\mu}\right\|\left\|\Y_{n+r-\mu'}\right\|\right)\right.\\
%    &\left.\qquad\qquad\qquad\qquad\qquad\qquad\qquad\qquad\qquad\qquad+2\E\left\|\Y_{n-\nu}\right\|\left\|\Y_{n-\nu'}\right\|\E\left\|\Y_{n+r-\mu}\right\|\left\|\Y_{n+r-\mu'}\right\|%\right].
\end{align*}
This sum can be split into one part $I^+$ in which at least one of the summation indices $\nu$, $\nu'$, $\mu$ and $\mu'$ exceeds $\Delta/2$, and one part $I^-$ in which all summation indices are less than or equal to $\Delta/2$. Using the fact that, by the Cauchy--Schwarz inequality,
\begin{align*}
&\left|\Cov\left(\Y_{n-\nu}^Tc^{(k),T}_{\bth,\nu}V_{\bth}^{-1}c_{\bth,\nu'}\Y_{n-\nu'};\Y_{n+\Delta-\mu}^Tc^{(l),T}_{\bth,\mu}V_{\bth}^{-1}c_{\bth,\mu'}\Y_{n+\Delta-\mu'}\right)\right| \leq  \left\|V_{\bth}^{-1}\right\|^2\left\|c^{(k)}_{\bth,\nu}\right\|\left\|c_{\bth,\nu'}\right\|\left\|c^{(l)}_{\bth,\mu'}\right\|\left\|c_{\bth,\mu'}\right\|\E\left\|\Y_n\right\|^4,
\end{align*}
it follows from \cref{assum-4moments} and the uniform exponential decay of $\left\|c_{\bth,\nu}\right\|$ and $\left\|c^{(m)}_{\bth,\nu}\right\|$ proved in \cref{lemma-propinnovations}, \labelcref{lemma-propinnovations-expbounded} and \cref{lemma-propinnovationspartial}, \labelcref{lemma-propinnovationspartialexpbounded} that there exist constants $C$ and $\rho<1$ such that
\begin{align}
\label{eq-boundcklbthr1}
I^+=&\sum_{\substack{\nu,\nu',\mu,\mu'=0\\\max\{\nu,\nu',\mu,\mu'\}>\Delta/2}}^\infty\left|\Cov\left(\Y_{n-\nu}^Tc^{(k),T}_{\bth,\nu}V_{\bth}^{-1}c_{\bth,\nu'}\Y_{n-\nu'},\Y_{n+\Delta-\mu}^Tc^{(l),T}_{\bth,\mu}V_{\bth}^{-1}c_{\bth,\mu'}\Y_{n+\Delta-\mu'}\right)\right|\leq  C \rho^{\Delta/2}.
\end{align}
For the contribution from all indices smaller than or equal to $\Delta/2$, \cref{lemma-davydov} implies that there exists a constant $C$ such that
\begin{align}
\label{eq-boundcklbthr2}
I^-=&\sum_{\nu,\nu',\mu,\mu'=0}^{\lfloor \Delta/2\rfloor}\left|\Cov\left(\Y_{n-\nu}^Tc^{(k),T}_{\bth,\nu}V_{\bth}^{-1}c_{\bth,\nu'}\Y_{n-\nu'},\Y_{n+\Delta-\mu}^Tc^{(l),T}_{\bth,\mu}V_{\bth}^{-1}c_{\bth,\mu'}\Y_{n+\Delta-\mu'}\right)\right| \leq C \left[\alpha_{\Y}\left(\left\lfloor \frac{\Delta}{2} \right\rfloor\right)\right]^{\delta/(2+\delta)}.
\end{align}
It thus follows from \cref{assum-mixing} that the sequences $\left|\mathfrak{c}_{\bth,\Delta}^{(k,l)}\right|$, $\Delta\in\N$, are summable, and \cref{eq-convIklbthN} completes the proof of the lemma.
\end{proof}

\begin{lemma}
\label{lemma-likelihoodequivpartial}
Let $\mathscr{L}$ and $\widehat{\mathscr{L}}$ be given by \cref{eq-likelihood,eq-hatlikelihood}. Assume that \cref{assum-compact,assum-smoothparam,assum-smoothparam3,assum-stability} are satisfied. Then the following hold.
\begin{enumerate}[i)]
 \item \label{lemma-likelihoodequivpartial1} For each $m=1,\ldots,r$, the sequence $L^{-1/2}\sup_{\bth\in\Theta}{\left|\partial_m\widehat{\mathscr{L}}(\bth,\y^L)-\partial_m\mathscr{L}(\bth,\y^L)\right|}$ converges to zero in probability, as $L\to\infty$.
 \item \label{lemma-likelihoodequivpartial2} For all $k,l=1,\ldots,r$, the sequence $L^{-1}\sup_{\bth\in\Theta}{\left|\partial^2_{k,l}\widehat{\mathscr{L}}(\bth,\y^L)-\partial^2_{k,l}\mathscr{L}(\bth,\y^L)\right|}$ converges to zero almost surely, as $L\to\infty$.
\end{enumerate}
\end{lemma}
\begin{proof}
Similar to the proof of \cref{lemma-likelihoodequiv}.
\end{proof}

\begin{lemma}
\label{lemma-asymnormaldeltaLmixingY}
Under \cref{assum-compact,assum-smoothparam3,assum-stability,assum-4moments,assum-mixing}, the random variable $L^{-1/2}\nabla_{\bth}\widehat{\mathscr{L}}(\bth_0,\y^L)$ is asymptotically normally distributed with mean zero and covariance matrix $I(\bth_0)$.
\end{lemma}
\begin{proof}
Because of \cref{lemma-likelihoodequivpartial}, \labelcref{lemma-likelihoodequivpartial1} it is enough to show that $L^{-1/2}\nabla_{\bth}\mathscr{L}\left(\bth_0,\y^L\right)$ is asymptotically normally distributed with mean zero and covariance matrix $I(\bth_0)$. First, we note that
\begin{equation}
\label{eq-partialiL}
\partial_i\mathscr{L}(\bth,\y^L) =\sum_{n=1}^L\left\{ \tr\left[V_{\bth}^{-1}\left(\I_d - \beps_{\bth,n}\beps_{\bth,n}^T V_{\bth}^{-1}\right)\partial_i V_{\bth}\right] + 2\left(\partial_i\beps_{\bth,n}^T\right)V_{\bth}^{-1}\beps_{\bth,n} \right\},
\end{equation}
which holds for every component $i=1,\ldots,r$. The facts that $\E\beps_{\bth_0,n}\beps_{\bth_0,n}^T$ equals $V_{\bth_0}$, and that $\beps_{\bth_0,n}$ is orthogonal to the Hilbert space generated by $\{\Y_t,t<n\}$, of which $\partial_i\beps_{\bth,n}^T$ is an element, show that $\E\partial_i\mathscr{L}\left(\bth_0,\y^L\right)=0$. Using \cref{lemma-propinnovations}, \labelcref{lemma-propinnovations-expbounded}, expression \labelcref{eq-partialiL} can be rewritten as
\begin{align*}
\partial_i\mathscr{L}\left(\bth_0,\y^L\right) =& \sum_{n=1}^L\left[Y^{(i)}_{m,n}-\E Y^{(i)}_{m,n}\right]+\sum_{n=1}^L\left[Z^{(i)}_{m,n}-\E Z^{(i)}_{m,n}\right],
\end{align*}
where, for every $m\in\N$, the processes $Y^{(i)}_m$ and $Z^{(i)}_m$ are defined by
\begin{subequations}
\label{eq-DefYZimn}
\begin{align}
\label{eq-DefYimn}Y^{(i)}_{m,n} =&\tr\left[V_{\bth_0}^{-1}(\partial_iV_{\bth_0})\right] + \sum_{\nu,\nu'=0}^m\left\{-\tr\left[V_{\bth_0}^{-1} c_{\bth_0,\nu}\Y_{n-\nu}\Y_{n-\nu'}^Tc_{\bth,\nu'}^T V_{\bth_0}^{-1}(\partial_iV_{\bth_0})\right]+ 2\Y_{n-\nu}^Tc^{(i),T}_{\bth_0,\nu}V_{\bth_0}^{-1}c_{\bth_0,\nu'}\Y_{n-\nu'} \right\},\\
\label{eq-DefZimn}Z^{(i)}_{m,n} =& U^{(i)}_{m,n} + V^{(i)}_{m,n},
\end{align}
\end{subequations}
and
\begin{align*}
U^{(i)}_{m,n} =& \sum_{\nu=0}^\infty\sum_{\nu'=m+1}^\infty\left\{-\tr\left[V_{\bth_0}^{-1}c_{\bth_0,\nu}\Y_{n-\nu}\Y_{n-\nu'}^Tc_{\bth,\nu'}^T V_{\bth_0}^{-1}(\partial_iV_{\bth_0})\right] + 2\Y_{n-\nu}^Tc^{(i),T}_{\bth_0,\nu}V_{\bth_0}^{-1}c_{\bth_0,\nu'}\Y_{n-\nu'} \right\},\\
V^{(i)}_{m,n} =& \sum_{\nu=m+1}^\infty\sum_{\nu'=0}^m\left\{-\tr\left[V_{\bth_0}^{-1}c_{\bth_0,\nu}\Y_{n-\nu}\Y_{n-\nu'}^Tc_{\bth,\nu'}^T V_{\bth_0}^{-1}(\partial_iV_{\bth_0})\right] + 2\Y_{n-\nu}^Tc^{(i),T}_{\bth_0,\nu}V_{\bth_0}^{-1}c_{\bth_0,\nu'}\Y_{n-\nu'} \right\}.
\end{align*}
It is convenient to also introduce the notations
\begin{equation}
\label{eq-DefYmnZmn}
\mathcal{Y}_{m,n} = \left(\begin{array}{ccc}Y^{(1)}_{m,n} & \cdots & Y^{(r)}_{m,n}\end{array}\right)^T\quad\text{and}\quad \mathcal{Z}_{m,n} = \left(\begin{array}{ccc}Z^{(1)}_{m,n} & \cdots & Z^{(r)}_{m,n}\end{array}\right)^T.
\end{equation}
The rest of the proof proceeds in three steps: in the first we show that, for each natural number $m$, the sequence  $L^{-1/2}\sum_n \left[\mathcal{Y}_{m,n}-\E\mathcal{Y}_{m,n}\right]$ is asymptotically normally distributed with asymptotic covariance matrix $I_m$, and that $I_m$ converges to $I(\bth_0)$ as $m$ tends to infinity. We then prove that $L^{-1/2}\sum_n \left[\mathcal{Z}_{m,n}-\E\mathcal{Z}_{m,n}\right]$ goes to zero uniformly in $L$, as $m\to\infty$, and the last step is devoted to combining the first two steps to prove the asymptotic normality of $L^{-1/2}\nabla_{\bth}\mathscr{L}\left(\bth_0,\y^L\right)$.

\paragraph*{\bf Step 1}
Since $\Y$ is stationary, it is clear that $\mathcal{Y}_m$ is a stationary process. Moreover, the strong mixing coefficients $\alpha_{\mathcal{Y}_m}(k)$ of $\mathcal{Y}_m$ satisfy $\alpha_{\mathcal{Y}_m}(k) \leq \alpha_{\Y}(\max\{0,k-m\})$ because $\mathcal{Y}_{m,n}$ depends only on the finitely many values $\Y_{n-m},\ldots,\Y_{n}$ of $\Y$ \citep[see][Remark 1.8 b)]{bradley2007introduction}. In particular, by \cref{assum-mixing}, the strong mixing coefficients of the processes $\mathcal{Y}_m$ satisfy the summability condition $\sum_{k}[\alpha_{\mathcal{Y}_m}(k)]^{\delta/(2+\delta)}<\infty$. Since, by the Cram\'er--Wold device, weak convergence of the sequence $L^{-1/2}\sum_{n=1}^L{\left[\mathcal{Y}_{m,n}-\E\mathcal{Y}_{m,n}\right]}$ to a multivariate normal distribution with mean zero and covariance matrix $\Sigma$ is equivalent to the condition that, for every vector $\bu\in\R^r$, the sequence $L^{-1/2}\bu^T\sum_{n=1}^L{\left[\mathcal{Y}_{m,n}-\E\mathcal{Y}_{m,n}\right]}$ converges to a one-dimensional normal distribution with mean zero and variance $\bu^T\Sigma\bu$, we can apply the \CLT\ for univariate strongly mixing processes 
% \citep{herrndorf1984},
\citep[Theorem 1.7]{ibragimov1962some} to obtain that
\begin{equation}
\label{eq-CLTsumYmn}
\frac{1}{\sqrt{L}}\sum_{n=1}^L{\left[\mathcal{Y}_{m,n}-\E\mathcal{Y}_{m,n}\right]}\xrightarrow[L\to\infty]{d}\mathscr{N}(\bzero_r,I_m),\quad\text{where}\quad I_m = \sum_{\Delta\in\Z}\Cov\left(\mathcal{Y}_{m,n};\mathcal{Y}_{m,n+\Delta}\right).
\end{equation}
The claim that $I_m$ converges to $I(\bth_0)$ will follow if we can show that
\begin{equation}
\Cov\left(Y^{(k)}_{m,n};Y^{(l)}_{m,n+\Delta}\right)\xrightarrow[m\to\infty]{} \Cov\left(\ell^{(k)}_{\bth_0,n};\ell^{(l)}_{\bth_0,n+\Delta}\right),\quad\forall \Delta\in\Z,
\end{equation}
and that $\left|\Cov\left(Y^{(k)}_{m,n};Y^{(l)}_{m,n+\Delta}\right)\right|$ is dominated by an absolutely summable sequence. For the first condition, we note that the bilinearity of $\Cov(\cdot;\cdot)$ implies that
\begin{align*}
\Cov\left(Y^{(k)}_{m,n};Y^{(l)}_{m,n+\Delta}\right) - \Cov\left(\ell^{(k)}_{\bth_0,n};\ell^{(l)}_{\bth_0,n+\Delta}\right) =& \Cov\left(Y^{(k)}_{m,n};Y^{(l)}_{m,n+\Delta}-\ell^{(l)}_{\bth_0,n+\Delta}\right) +  \Cov\left(Y^{(k)}_{m,n} - \ell^{(k)}_{\bth_0,n};\ell^{(l)}_{\bth_0,n+\Delta}\right).
\end{align*}
These two terms can be treated in a similar manner so we restrict our attention to the second one. The definitions of $Y^{(i)}_{m,n}$ (\cref{eq-DefYimn}) and $\ell^{(i)}_{\bth,n}$ (\cref{eq-DefIlimn}) allow us to compute
\begin{align*}
Y^{(k)}_{m,n} - \ell^{(k)}_{\bth_0,n}=&\sum_{\substack{\nu,\nu'\\\max\{\nu,\nu'\}>m}}\left[\tr\left[V_{\bth_0}^{-1}c_{\bth_0,\nu}\Y_{n-\nu}\Y_{n-\nu'}^Tc_{\bth,\nu'}^T V_{\bth_0}^{-1}\partial_iV_{\bth_0}\right] - 2\Y_{n-\nu}^Tc^{(i),T}_{\bth_0,\nu}V_{\bth_0}^{-1}c_{\bth_0,\nu'}\Y_{n-\nu'} \right].
\end{align*}
As a consequence of the Cauchy--Schwarz inequality, \cref{assum-4moments} and the exponential bounds in \cref{lemma-propinnovations}, \labelcref{lemma-propinnovationsexpconv}, we therefore obtain that $\Var\left(Y^{(k)}_{m,n} - \ell^{(k)}_{\bth_0,n}\right)\leq C\rho^m$ independent of $n$. The $L^2$\hyp{}continuity of $\Cov(\cdot;\cdot)$ thus implies that the sequence $\Cov\left(Y^{(k)}_{m,n} - \ell^{(k)}_{\bth_0,n};\ell^{(l)}_{\bth_0,n+\Delta}\right)$ converges to zero as $m$ tends to infinity at an exponential rate uniformly in $\Delta$. The existence of a summable sequence dominating $\left|\Cov\left(Y^{(k)}_{m,n};Y^{(l)}_{m,n+\Delta}\right)\right|$ is ensured by the arguments given in the proof of \cref{lemma-convIbth}, reasoning as in the derivation of \cref{eq-boundcklbthr1,eq-boundcklbthr2}.
\paragraph*{\bf Step 2}
We shall show that there exist positive constants $C$ and $\rho<1$, independent of $L$, such that
\begin{equation}
\label{eq-varsumZ}
\tr\Var\left(\frac{1}{\sqrt{L}}\sum_{n=1}^L\mathcal{Z}_{m,n}\right)\leq C \rho^m,\quad \text{$\mathcal{Z}_{m,n}$ given in \cref{eq-DefYmnZmn}}.
\end{equation}
Since
\begin{equation}
\label{eq-varsumZboundUV}
\tr\Var\left(\frac{1}{\sqrt{L}}\sum_{n=1}^L\mathcal{Z}_{m,n}\right)\leq 2\left[\tr\Var\left(\frac{1}{\sqrt{L}}\sum_{n=1}^L\mathcal{U}_{m,n}\right)+\tr\Var\left(\frac{1}{\sqrt{L}}\sum_{n=1}^L\mathcal{V}_{m,n}\right)\right],
\end{equation}
it suffices to consider the latter two terms. We first observe that
\begin{align}
\label{eq-VarsumU}
\tr\Var\left(\frac{1}{\sqrt{L}}\sum_{n=1}^L\mathcal{U}_{m,n}\right) = &\frac{1}{L}\tr\sum_{n,n'=1}^L\Cov\left(\mathcal{U}_{m,n};\mathcal{U}_{m,n'}\right) = \frac{1}{L}\sum_{k,l=1}^r\sum_{\Delta=-L+1}^{L-1}\left(L-|\Delta|\right)\mathfrak{u}^{(k,l)}_{m,\Delta}\leq\sum_{k,l=1}^r\sum_{\Delta\in\Z}\left|\mathfrak{u}^{(k,l)}_{m,\Delta}\right|,
\end{align}
where
\begin{align*}
 \mathfrak{u}^{(k,l)}_{m,\Delta} =& \Cov\left(U^{(k)}_{m,n};U^{(l)}_{m,n+\Delta}\right)\\
  =& \sum_{\substack{\nu,\mu=0\\\nu',\mu'=m+1}}^m \Cov\left( -\tr\left[V_{\bth_0}^{-1}c_{\bth_0,\nu}\Y_{n-\nu}\Y_{n-\nu'}^Tc_{\bth,\nu'}^T V_{\bth_0}^{-1}\partial_kV_{\bth_0}\right] + \Y_{n-\nu}^Tc^{(k),T}_{\bth_0,\nu}V_{\bth_0}^{-1}c_{\bth_0,\nu'}\Y_{n-\nu'} \right.;\\*
  &\qquad\left. -\tr\left[V_{\bth_0}^{-1}c_{\bth_0,\mu}\Y_{n+\Delta-\mu}\Y_{n+\Delta-\mu'}^Tc_{\bth,\mu'}^T V_{\bth_0}^{-1}\partial_lV_{\bth_0}\right] + \Y_{n+\Delta-\mu}^Tc^{(l),T}_{\bth_0,\mu}V_{\bth_0}^{-1}c_{\bth_0,\mu'}\Y_{n+\Delta-\mu'} \right).
\end{align*}
As before, under \cref{assum-4moments}, the Cauchy--Schwarz inequality and the exponential bounds for $\left\|c_{\bth_0,\nu}\right\|$ and $\left\|c^{(k)}_{\bth_0,\nu}\right\|$ imply that $\left|\mathfrak{u}^{(k,l)}_{m,\Delta}\right|< C\rho^m$. By arguments similar to the ones used in the proof of \cref{lemma-davydov} Davydov's inequality implies that, for $m<\lfloor \Delta/2\rfloor$,
\begin{align*}
\left|\mathfrak{u}^{(k,l)}_{m,\Delta}\right| \leq & C\sum_{\nu=0}^\infty\sum_{\nu'=m+1}^\infty\sum_{\mu,\mu'=0}^{\lfloor \Delta/2\rfloor}\rho^{\nu+\nu'+\mu+\mu'}\left[\alpha_{\Y}\left(\left\lfloor\frac{\Delta}{2}\right\rfloor\right)\right]^{\delta/(2+\delta)}  +  C\sum_{\nu,\nu'=0}^\infty\sum_{\substack{\mu,\mu'\\\max\{\mu,\mu'\}>\lfloor \Delta/2\rfloor}}\rho^{\nu+\nu'+\mu+\mu'} \\
  \leq& C \rho^ m \left\{ \left[\alpha_{\Y}\left(\left\lfloor\frac{\Delta}{2}\right\rfloor\right)\right]^{\delta/(2+\delta)} + \rho^{\Delta/2} \right\}.
\end{align*}
It thus follows that, independent of the value of $k$ and $l$,
\begin{equation*}
\sum_{\Delta=0}^\infty\left|\mathfrak{u}^{(k,l)}_{m,\Delta}\right|=\sum_{\Delta=0}^{2m}\left|\mathfrak{u}^{(k,l)}_{m,\Delta}\right|+\sum_{\Delta=2m+1}^\infty\left|\mathfrak{u}^{(k,l)}_{m,\Delta}\right|\leq C\rho^m \left\{ m  +  \sum_{\Delta=0}^\infty \left[\alpha_{\Y}\left(\Delta\right)\right]^{\delta/(2+\delta)}\right\},
\end{equation*}
and therefore, by \cref{eq-VarsumU}, that $\tr\Var\left(L^{-1/2}\sum_{n=1}^L\mathcal{U}_{m,n}\right) \leq C \rho^m$. In an analogous way one also can show that $\tr\Var\left(L^{-1/2}\sum_{n=1}^L\mathcal{V}_{m,n}\right) \leq C \rho^m$, and thus the claim \labelcref{eq-varsumZ} follows with \cref{eq-varsumZboundUV}.
\paragraph*{\bf Step 3}
In step 1 it has been shown that $L^{-1/2}\sum_n \left[\mathcal{Y}_{m,n}-\E\mathcal{Y}_{m,n}\right]\xrightarrow[L\to\infty]{d}\mathscr{N}(\bzero_r,I_m)$, and that $I_m$ converges to $I(\bth_0)$, as $m\to\infty$. In particular, the limiting normal random variables with covariances $I_m$ converge weakly to a normal random variable with covariance matrix $I(\bth_0)$. Step 2 together with the multivariate Chebyshev inequality implies that, for every $\epsilon>0$,
\begin{align*}
&\lim_{m\to\infty}\limsup_{L\to\infty}\Pb\left(\left\|\frac{1}{\sqrt L}\nabla_{\bth}\mathscr{L}\left(\bth_0,\y^L\right)-\frac{1}{\sqrt L}\sum_{n=1}^L\left[\mathcal{Y}_{m,n}-\E\mathcal{Y}_{m,n}\right]\right\|>\epsilon\right)\\
%   =&\lim_{m\to\infty}\limsup_{L\to\infty}\Pb\left(\left\|\frac{1}{\sqrt L}\sum_{m=1}^L\left[\mathcal{Z}_{m,n}-\E\mathcal{Z}_{m,n}\right]\right\|>\epsilon\right) \\
  \leq & \lim_{m\to\infty}\limsup_{L\to\infty}\frac{r}{\epsilon^2}\tr\Var\left(\frac{1}{\sqrt{L}}\sum_{n=1}^L\mathcal{Z}_{m,n}\right)\leq\lim_{m\to\infty} \frac{C r}{\epsilon^2}\rho^m =0.
\end{align*}
Proposition 6.3.9 of \citet{brockwell1991tst} thus completes the proof.
\end{proof}

A very important step in the proof of asymptotic normality of QML estimators is to establish that the Fisher information matrix $J$, evaluated at the true parameter value, is non\hyp{}singular. We shall now show that \cref{assum-identifiabilityFisher} is sufficient to ensure that $J^{-1}$ exists for linear state space models. For vector ARMA processes, \formulaplural\ similar to \cref{eq-DefJ1J2} below have been derived in the literature \citep[see, e.\,g.,][]{klein2008asymptotic,klein2000direct}; in fact, the resultant property of the Fisher information matrix of a vector ARMA process implies that $J$ in this case is non\hyp{}singular if and only if its autoregressive and moving average polynomials have no common eigenvalues \citep{klein2005resultant}. In conjunction with the equivalence of linear state space and vector ARMA models this provides an alternative way of checking that $J$ in non\hyp{}singular. We continue to work with \cref{assum-identifiabilityFisher}, however, because it avoids the transformation of the state space model \labelcref{eq-innohatSSM} into an equivalent ARMA form.

% but have not been used to derive criteria for $J$ being non\hyp{}singular. 

\begin{lemma}
\label{lemma-convJ}
Assume that \cref{assum-compact,assum-smoothparam,assum-2moments,assum-smoothparam3,assum-stability,assum-identifiabilityFisher} hold. With probability one, the matrix $J=\lim_{L\to\infty}L^{-1}\nabla_{\bth}^2\widehat{\mathscr{L}}(\bth_0,\y^L)$ exists and is non\hyp{}singular.
\end{lemma}
\begin{proof}
It can be shown as in the proof of \citet[Lemma 4]{mainassara2011estimating} that $J$ exists and is equal to $J=J_1+J_2$, where
\begin{equation}
\label[pluralequation]{eq-DefJ1J2}
J_1 = 2\E \left[\left(\nabla_{\bth}\beps_{\bth_0,1}\right)^TV_{\bth_0}^{-1}\left(\nabla_{\bth}\beps_{\bth_0,1}\right)\right]\quad\text{and}\quad J_2 = \left(\tr\left[V_{\bth_0}^{-1/2} \left(\partial_i V_{\bth_0}\right)V_{\bth_0}^{-1}\left(\partial_j V_{\bth_0}\right)V_{\bth_0}^{-1/2}\right]\right)_{ij}.
\end{equation}
$J_2$ is positive semidefinite because it can be written as $J_2 = \left(\begin{array}{ccc}\bb_1 & \hdots & \bb_r\end{array}\right)^T\left(\begin{array}{ccc}\bb_1 & \hdots & \bb_r\end{array}\right)$, where $\bb_m = \left(V_{\bth_0}^{-1/2}\otimes V_{\bth_0}^{-1/2}\right)\vec\left(\partial_m V_{\bth_0}\right)$. Since $J_1$ is positive semidefinite as well, proving that $J$ is non\hyp{}singular is equivalent to proving that for any non\hyp{}zero vector $\bc\in\R^r$, the numbers $\bc^T J_i\bc$, $i=1,2$, are not both zero. Assume, for the sake of contradiction, that there exists such a vector $\bc=(c_1,\ldots,c_r)^T$. The condition $\bc^T J_1\bc$ implies that, almost surely, $\sum_{k=1}^r c_k\partial_k\beps_{\bth_0,n}=\bzero_d$, for all $n\in\Z$. It thus follows that $\sum_{\nu=1}^\infty{\sum_{k=1}^r c_k\left(\partial_k\mathscr{M}_{\bth_0,\nu}\right)}\beps_{\bth_0,-\nu}=\bzero_d$, where the Markov parameters $\mathscr{M}_{\bth,\nu}$ are given by $\mathscr{M}_{\bth,\nu} = -H_{\bth} F_{\bth}^{\nu-1}K_{\bth}$, $\nu\geq 1$. Since the sequence $\beps_{\bth_0}$ is uncorrelated with positive definite covariance matrix, it follows that $\sum_{k=1}^rc_k\left(\partial_k\mathscr{M}_{\bth_0,\nu}\right)=\bzero_d$, for every $\nu\in\N$. Using the relation $\vec(ABC) = \left(C^T\otimes A\right)\vec B$ \citep[Proposition 7.1.9]{bernstein2005matrix}, we see that the last display is equivalent to $\nabla_{\bth}\left(\left[K_{\bth_0}^T\otimes H_{\bth_0}\right]\vec F_{\bth_0}^{\nu-1}\right)\bc=\bzero_{d^2}$ for every $\nu\in\N$. The condition $\bc^T J_2\bc=0$ implies that $\left(\nabla_{\bth} \vec V_{\bth_0}\right)\bc = \bzero_{d^2}$. By the definition of $\psi_{\bth,j}$ in \cref{eq-DefPsij} it thus follows that $\nabla_{\bth}\psi_{\bth_0,j}\bc=\bzero_{(j+2)d^2}$, for every $j\in \N$, which, by \cref{assum-identifiabilityFisher}, is equivalent to the contradiction that $\bc=\bzero_r$.
\end{proof}

\begin{proof}[Proof of \cref{theorem-hatbthLCLT}]
Since the estimate $\hat\bth^L$ converges almost surely to $\bth_0$ by the consistency result proved in \cref{theorem-consistency}, and $\bth_0$ is an element of the interior of $\Theta$ by \cref{assum-interior}, the estimate $\hat\bth^L$ is an element of the interior of $\Theta$ eventually almost surely. The assumed smoothness of the parametrization (\cref{assum-smoothparam3}) implies that the extremal property of $\hat\bth^L$ can be expressed as the first order condition $\nabla_{\bth}\widehat{\mathscr{L}}(\hat\bth^L,\y^L)=\bzero_r$. A Taylor expansion of $\bth\mapsto\nabla_{\bth}\widehat{\mathscr{L}}(\bth,\y^L)$ around the point $\bth_0$ shows that there exist parameter vectors $\bth_i\in\Theta$ of the form $\bth_i=\bth_0+c_i(\hat\bth^L-\bth_0)$, $0\leq c_i\leq 1$, such that
\begin{equation}
\label{eq-TaylorLbth}
\bzero_r = L^{-1/2}\nabla_{\bth}\widehat{\mathscr{L}}(\bth_0,\y^L) + \frac{1}{L}\nabla^2_{\bth}\widehat{\mathscr{L}}(\underline\bth^L,\y^L)L^{1/2}\left(\hat\bth^L-\bth_0\right),
\end{equation}
where $\nabla^2_{\bth}\widehat{\mathscr{L}}(\underline\bth^L,\y^L)$ denotes the matrix whose $i$th row, $i=1,\ldots,r$, is equal to the $i$th row of $\nabla^2_{\bth}\widehat{\mathscr{L}}(\bth_i,\y^L)$. By \cref{lemma-asymnormaldeltaLmixingY} the first term on the right hand side converges weakly to a multivariate normal random variable with mean zero and covariance matrix $I=I(\bth_0)$. As in \cref{lemma-uniformconvergence} one can show that the sequence $\bth\mapsto L^{-1}\nabla^3_{\bth}\widehat{\mathscr{L}}(\bth,\y^L)$, $L\in\N$, of random functions converges almost surely uniformly to the continuous function $\bth\mapsto\nabla_{\bth}^3\mathscr{Q}(\bth)$ taking values in the space $\R^{r\times r\times r}$. Since on the compact space $\Theta$ this function is bounded in the operator norm obtained from identifying $\R^{r\times r\times r}$ with the space of linear functions from $\R^r$ to $M_r(\R)$, that sequence is almost surely uniformly bounded, and we obtain that
\begin{align*}
\left\|\frac{1}{L}\nabla^2_{\bth}\widehat{\mathscr{L}}(\underline\bth^L,\y^L) - \frac{1}{L}\nabla^2_{\bth}\widehat{\mathscr{L}}(\bth_0,\y^L)\right\|\leq \sup_{\bth\in\Theta}\left\|\frac{1}{L}\nabla^3_{\bth}\widehat{\mathscr{L}}(\bth,\y^L)\right\|\left\|\underline\bth^L-\bth_0\right\|\xrightarrow[L\to\infty]{\text{a.\,s.}}0,
\end{align*}
because, by \cref{theorem-consistency}, the second factor almost surely converges to zero as $L$ tends to infinity. It follows from \cref{lemma-convJ} that $L^{-1}\nabla^2_{\bth}\widehat{\mathscr{L}}(\underline\bth^L,\y^L)$ converges to the matrix $J$ almost surely, and thus from \cref{eq-TaylorLbth} that $L^{1/2}\left(\hat\bth^L-\bth_0\right)\convd \mathscr{N}\left(\bzero_r,J^{-1}IJ^{-1}\right)$, as $L\to\infty$. This shows \cref{eq-hatbthLCLT} and completes the proof.
\end{proof}

\section[QML estimation for multivariate CARMA processes]{Quasi maximum likelihood estimation for multivariate continuous\hyp{}time ARMA processes}
\label{section-QMLMCARMA}

In this section we pursue the second main topic of the present paper, a detailed investigation of the asymptotic properties of the QML estimator of discretely observed multivariate continuous\hyp{}time autoregressive moving average processes. We will make use of the equivalence between MCARMA and continuous\hyp{}time linear state space models, as well as of the important observation that the state space structure of a continuous\hyp{}time process is preserved under equidistant sampling, which allows for the results of the previous section to be applied. The conditions we need to impose on the parametrization of the models under consideration are therefore closely related to the assumptions made in the discrete\hyp{}time case, except that the mixing and ergodicity assumptions \labelcref{assum-2moments,assum-mixing} are automatically satisfied \citep[][Proposition 3.34]{marquardt2007multivariate}.

We start the section with a short recapitulation of the definition and basic properties of L\'evy\hyp{}driven continuous\hyp{}time ARMA processes and their equivalence to state space models (based mainly on \cite{marquardt2007multivariate,schlemmmixing2010}). Thereafter we work towards being able to apply our results on QML estimation for discrete time state models to QML estimators for MCARMA processes culminating in our main result \cref{theorem-CLTmcarma}. To this end we first recall the second order structure of continuous time state space models and provide auxiliary results on the transfer function  in \cref{sec:secord}. This is followed in \cref{section-sampling} by recalling that equidistant observations of an MCARMA processes follow a state space model in discrete time, as well as discussions of the minimality of a state space model and of how to make the relation between the continuous and discrete time  state space models unique.  The following \cref{section-identifiability} looks at the second\hyp{}order properties of a discretely observed MCARMA process and the aliasing effect. Together the results of \cref{sec:secord,section-sampling,section-identifiability} allow to give accessible identifiability conditions needed to apply the QML estimation theory developed in \cref{section-QMLDTSSM}. Finally, \Cref{section-asymptoticmcarma} introduces further technical assumptions needed to employ the theory for strongly mixing state space models and then derives our main result about the consistency and asymptotic normality of the QML estimator for equidistantly sampled MCARMA processes in \cref{theorem-CLTmcarma}.

\subsection{L\'evy\hyp{}driven multivariate CARMA processes and continuous\hyp{}time state space models}
\label{section-MCARMA-QML}

A natural source of randomness in the specification of continuous\hyp{}time stochastic processes are L\'evy processes. For a thorough discussion of these processes we refer the reader to the monographs \citet{applebaum2004lpa,sato1991lpa}.
% bertoin1996levy
\begin{definition}
% [L\'evy process]
\label{def-levyprocess}
A two\hyp{}sided $\R^m$\hyp{}valued {\em L\'evy process} $\left(\Lb(t)\right)_{t\in\mathbb{R}}$ is a stochastic process, defined on a probability space $(\Omega,\mathscr{F},\Pb)$, with stationary, independent increments, continuous in probability, and satisfying $\Lb(0)=\bzero_m$ almost surely.
\end{definition}

The characteristic function of a L\'evy process $\Lb$ has the L\'evy-Khintchine-form $\E \ee^{\ii\langle\bu,\Lb(t)\rangle}=\exp\{t\psi^{\Lb}(\bu)\}$, $\bu\in\R^m$, $t\in\R^+$, where the characteristic exponent $\psi^{\Lb}$ is given by
\begin{equation}
\label{eq-levykhintchine2}
\psi^{\Lb}(\bu)=\ii\langle \bgammaL,\bu\rangle-\frac{1}{2}\langle \bu,\SigGauss\bu\rangle+\int_{\R^m}{\left[\ee^{\ii\langle\bu,\bx\rangle}-1-\ii\langle\bu,\bx\rangle I_{\{\left\|x\right\|\leq 1\}}\right]\nuL (\dd\bx)}.
\end{equation}
The vector $\bgammaL\in\R^m$ is called the {\it drift}, $\SigGauss$ is a non\hyp{}negative definite, symmetric $m\times m$ matrix called the {\it Gaussian covariance matrix}, and the {\it L\'evy measure} $\nuL $ satisfies the two conditions $\nuL (\{\bzero_m\})=0$ and $\int_{\R^m}\min(\left\|\bx\right\|^2,1)\nuL (\dd\bx)<\infty$. For the present purpose it is enough to know that a L\'evy process $\Lb$ has finite $k$th absolute moments, $k>0$, that is $\E\left\|\Lb(t)\right\|^k<\infty$, if and only if $\int_{\left\|\bx\right\|\geq 1}\left\|\bx\right\|^k\nuL(\dd\bx)<\infty$ \citep[Corollary 25.8]{sato1991lpa}, and that the covariance matrix $\SigL$ of $\Lb(1)$, if it exists, is given by $\SigGauss+\int_{\left\|\bx\right\|\geq 1}{\bx\bx^T\nuL(\dd\bx)}$ \citet[Example 25.11]{sato1991lpa}.
\begin{assumptionlevy}
\label{assum-levy2}
The L\'evy process $\Lb$ has mean zero and finite second moments, i.\,e.\ $\bgammaL+\int_{\left\|\bx\right\|\geq 1}{\bx\nuL(\dd\bx)}$ is zero, and the integral $\int_{\left\|\bx\right\|\geq 1}{\left\|\bx\right\|^2\nuL(\dd\bx)}$ is finite.
\end{assumptionlevy}

Just like i.\,i.\,d.\ sequences are used in time series analysis to define ARMA processes, L\'evy processes can be used to construct (multivariate) continuous\hyp{}time autoregressive moving average processes, called (M)CARMA processes. If $\Lb$ is a two\hyp{}sided  L\'evy process with values in $\R^m$ and $p>q$ are integers, the $d$\hyp{}dimensional $\Lb$\hyp{}driven $\text{MCARMA}(p,q)$ process with autoregressive polynomial
\begin{subequations}
\label{ARMApoly-QML}
\begin{equation}
\label{ARpoly-QML}
z\mapsto P(z)\coloneqq \I_dz^p+A_1z^{p-1}+\ldots+A_p\in M_d(\R[z])                                                                                                                                                                                                                                                                                                                                                                                                                                                                                   \end{equation}
and moving average polynomial
\begin{equation}
\label{poly-QML}
z\mapsto Q(z)\coloneqq B_0z^q+B_1z^{q-1}+\ldots+B_q\in M_{d,m}(\R[z]) 
\end{equation}
\end{subequations}
is defined as the solution to the formal differential equation $P(\DD)\Y(t)=Q(\DD)\DD\Lb(t)$, $\DD\equiv(\dd/\dd t)$. It is often useful to allow for the dimensions of the driving L\'evy process $\Lb$ and the $\Lb$\hyp{}driven MCARMA process to be different, which is a slight extension of the original definition of \citet{marquardt2007multivariate}. The results obtained in that paper remain true if our definition is used. In general, the paths of a L\'evy process are not differentiable, so we interpret the defining differential equation as being equivalent to the {\it state space representation}
\begin{equation}
\label{eq-MCARMAssm-QML}
\dd \boldsymbol{G}(t)=\mathcal{A}\boldsymbol{G}(t)\dd t+\mathcal{B} \dd \Lb(t),\quad \Y(t)= \mathcal{C}\boldsymbol{G}(t),\quad t\in\R,
\end{equation}
where $\mathcal{A}$ ,$\mathcal{B}$, and $\mathcal{C}$ are given by
\begin{subequations}
\label{eq-MCARMAcoeffABC-QML}
\begin{align}
\label{eq-MCARMAcoeffA-QML} \mathcal{A} =& \left(\begin{array}{ccccc}
       0 & \I_d & 0 & \ldots & 0 \\
	0 & 0 & \I_d & \ddots & \vdots \\
	\vdots && \ddots & \ddots & 0\\
	0 & \ldots & \ldots & 0 & \I_d\\
	-A_p & -A_{p-1} & \ldots & \ldots & -A_1
      \end{array}\right)\in M_{pd}(\R),\\
\label{MCARMAcoeff-BQML}\mathcal{B}=&\left(\begin{array}{ccc}\beta_1^T & \cdots & \beta_p^T\end{array}\right)^T\in M_{pd,m}(\R),\quad\beta_{p-j} = -I_{\{0,\ldots,q\}}(j)\left[\sum_{i=1}^{p-j-1}{A_i\beta_{p-j-i}-B_{q-j}}\right],\\
\label{MCARMAcoeffCQML}\mathcal{C}=&\left(\I_d,0,\ldots,0\right)\in M_{d,pd}(\R).
\end{align}
\end{subequations}
It follows from representation \labelcref{eq-MCARMAssm-QML} that MCARMA processes are special cases of linear multivariate continuous\hyp{}time state space models, and in fact, the class of linear state space models is equivalent to the class of MCARMA models \citep[Corollary 3.4]{schlemmmixing2010}. By considering the class of linear state space models, one can define representations of MCARMA processes which are different from \cref{eq-MCARMAssm-QML} and better suited for the purpose of estimation. 
\begin{definition}
% [State space model]
A continuous\hyp{}time linear state space model $(A,B,C,\Lb)$ of dimension $N$ with values in $\R^d$ is characterized by an $\R^m$\hyp{}valued driving L\'evy process $\Lb$, a state transition matrix $A\in M_N(\R)$, an input matrix $B\in M_{N,m}(\R)$, and an observation matrix $C\in M_{d,N}(\R)$. It consists of a state equation of Ornstein--Uhlenbeck type
\begin{subequations}
\label{eq-ssm-QML}
\begin{equation}
\label{eq-stateeq-QML}
 \dd \X(t) =A\X(t) \dd t+B \dd \Lb(t),\quad t\in\R,
 \end{equation}
and an observation equation
\begin{equation}
\label{obseq}
\Y(t) = C\X(t),\quad t\in\R.
\end{equation}
\end{subequations}
The $\R^N$\hyp{}valued process $\X=\left(\X(t)\right)_{t\in\R}$ is the {\it state vector process}, and $\Y=\left(\Y(t)\right)_{t\in\R}$ the {\it output process}.
\end{definition}
A solution $\Y$ to \cref{eq-ssm-QML} is called {\it causal} if, for all $t$, $\Y(t)$ is independent of the $\sigma$\hyp{}algebra generated by $\{\Lb(s):s>t\}$. Every solution to \cref{eq-stateeq-QML} satisfies
\begin{equation}
\label{eq-markovstate-QML}
\X(t) = \ee^{A(t-s)}\X(s)+\int_s^t{\ee^{A(t-u)}B\dd \Lb(u)},\quad\forall s,t\in\R,\quad s<t.
\end{equation}
The following can be seen as the multivariate extension of \citet[Proposition 1]{brockwell2010estimation} and recalls conditions for the existence of a stationary causal solution of the state equation \labelcref{eq-stateeq-QML} for easy reference. We always work under the following assumption.
\begin{assumptioneigen}
\label{assumEigen1-QML}
The eigenvalues of the matrix $A$ have strictly negative real parts.
\end{assumptioneigen}
\subsection{Second order structure and the transfer function}\label{sec:secord}
\begin{proposition}[{{\citet[Theorem 5.1]{sato1983stationary}}}]
\label{prop-2ndorderOU}
If Assumptions \labelcref{assumEigen1-QML} and \labelcref{assum-levy2} hold, then \cref{eq-stateeq-QML} has a unique strictly stationary, causal solution $\X$ given by $\X(t)=\int_{-\infty}^t{\ee^{A(t-u)}B\dd \Lb(u)}$. Moreover, $\X(t)$ has mean zero and second\hyp{}order structure
\begin{subequations}
\begin{align}
\label{eq-Varssm-QML}\Var(\X(t))\eqqcolon& \Gamma_0=\int_0^\infty{\ee^{Au}B\SigL B^T\ee^{A^Tu}}\dd u,\\
\label{eq-Covssm-QML}\Cov{(\X(t+h),\X(t)})\eqqcolon& \gamma_{\Y}(h) =\ee^{Ah}\Gamma_0,\quad h\geq 0,
\end{align}
\end{subequations}
where the variance $\Gamma_0$ satisfies $A\Gamma_0+\Gamma_0A^T = -B\SigL B^T$.
\end{proposition}
It is an immediate consequence that the output process $\Y$ has mean zero and autocovariance function $\R\ni h\mapsto\gamma_{\Y}(h)$ given by $\gamma_{\Y}(h)=C\ee^{Ah}\Gamma_0C^T$, $h\geq 0$, and that $\Y$ itself can be written succinctly as a moving average of the driving L\'evy process as $\Y(t)=\int_{-\infty}^\infty{g(t-u)\dd \Lb(u)}$, where $g(t)=C\ee^{At}BI_{[0,\infty)}(t)$. This representation shows that the behaviour of the process $\Y$ depends on the values of the individual matrices $A$, $B$, and $C$ only through the products $C\ee^{At}B$, $t\in\R$. The following lemma relates this analytical statement to an algebraic one about rational matrices, allowing us to draw a connection to the identifiability theory of discrete\hyp{}time state space models.
\begin{lemma}
Two matrix triplets $(A,B,C)$, $(\tilde A,\tilde B,\tilde C)$ of appropriate dimensions satisfy $C\ee^{At}B=\tilde C\ee^{\tilde At}\tilde B$ for all $t\in\R$ if and only if $C(z\I-A)^{-1}B=\tilde C(z\I-\tilde A)^{-1}\tilde B$ for all $z\in\C$.
\end{lemma}
\begin{proof}
If we start at the first equality and replace the matrix exponentials by their spectral representations \citep[see][Theorem 17.5]{lax2002functional}, we obtain $\int_{\gamma}{\ee^{zt}C(z\I-A)^{-1}B\dd z}=\int_{\tilde\gamma}{\ee^{zt}\tilde C(z\I-\tilde A)^{-1}\tilde B\dd z}$, where $\gamma$ is a closed contour in $\C$ winding around each eigenvalue of $A$ exactly once, and likewise for $\tilde\gamma$. Since we can always assume that $\gamma=\tilde\gamma$ by taking $\gamma$ to be $R$ times the unit circle, $R>\max\{|\lambda|:\lambda\in\sigma_A\cup\sigma_{\tilde A}\}$,it follows that, for each $t\in\R$, $\int_{\gamma}{\ee^{zt}\left[C(z\I-A)^{-1}B-\tilde C(z\I-\tilde A)^{-1}\tilde B\right]\dd z}=0$. Since the rational matrix function $ \Delta(z)=C(z\I-A)^{-1}B-\tilde C(z\I-\tilde A)^{-1}\tilde B$ has only poles with modulus less than $R$, it has an expansion around infinity, $\Delta(z) = \sum_{n=0}^\infty{A_n z^{-n}}$, $A_n\in M_d(\C)$, which converges in a region $\{z\in\C:|z|>r\}$ containing $\gamma$. Using the fact that this series converges uniformly on the compact set $\gamma$ and applying the Residue Theorem from complex analysis,
% \citep[9.16.1]{dieudonne2006foundations}
which implies $\int_\gamma{\ee^{zt}z^{-n}\dd z}=t^n/n!$, one sees that $\sum_{n=0}^\infty{\frac{t^n}{n!}A_{n+1}}\equiv 0_N$. Consequently, by the Identity Theorem,
% \citep[Theorem 9.4.3]{dieudonne2006foundations}
$A_n$ is the zero matrix for all $n>1$, and since $\Delta(z)\to 0$ as $z\to\infty$, it follows that $\Delta(z)\equiv 0_{d,m}$.
\end{proof}
The rational matrix function $H:z\mapsto C(z\I_N-A)^{-1}B$ is called the {\it transfer function} of the state space model \labelcref{eq-ssm-QML} and is closely related to the spectral density $f_{\Y}$ of the output process $\Y$, which is defined as $f_{\Y}(\omega)=\int_{\R}{\ee^{-\ii \omega h}\gamma_{\Y}(h)\dd h}$ -- the Fourier transform of $\gamma_{\Y}$. Before we make this relation explicit, we prove the following lemma.
\begin{lemma}
\label{lemma-inteABSigmaBeA}
For any real number $v$, and matrices $A,B,\SigL ,\Gamma_0$ as in \cref{eq-Varssm-QML}, it holds that
\begin{equation}
\label{matrixequation}
 \int_{-v}^\infty{\ee^{Au}B\SigL B^T\ee^{A^Tu}\dd u} = \ee^{-Av}\Gamma_0\ee^{-A^Tv}.
\end{equation}
\end{lemma}
\begin{proof}
We define functions $l,r:\R\to M_N(\R)$ by $l(v)=\int_{-v}^\infty{\ee^{Au}B\SigL B^T\ee^{A^Tu}\dd u}$ and $r(v)=\ee^{-Av}\Gamma_0\ee^{-A^Tv}$. Both $l:v\mapsto l(v)$ and $r:v\mapsto r(v)$ are differentiable functions of $v$, satisfying
\begin{align*}
\frac{\dd }{\dd v}l(v)	=& 	\ee^{-Av}B\SigL B^T\ee^{-A^Tv}\quad\text{and}\quad \frac{\dd }{\dd v}r(v)	=	-A\ee^{-Av}\Gamma_0\ee^{-A^Tv} - \ee^{-Av}\Gamma_0A^T\ee^{-A^Tv}.
\end{align*}
Using \cref{prop-2ndorderOU} one sees immediately that $(\dd/\dd v) l(v)=(\dd/\dd v) r(v)$, for all $v\in\R$. Hence, $l$ and $r$ differ only by an additive constant. Since $l(0)$ equals $r(0)$ by the definition of $\Gamma_0$, the constant is zero, and $l(v)=r(v)$ for all real numbers $v$.
\end{proof}

\begin{proposition}
\label{SSMspecfac}
Let $\Y$ be the output process of the state space model \labelcref{eq-ssm-QML}, and denote by $H:z\mapsto C(z\I_N-A)^{-1}B$ its transfer function. Then the relation $f_{\Y}(\omega) = (2\pi)^{-1}H(\ii\omega)\SigL H(-\ii\omega)^T$ holds for all real $\omega$; in particular, $\omega\mapsto f_{\Y}(\omega)$ is a rational matrix function.
\end{proposition}
\begin{proof}
First, we recall \citep[Proposition 11.2.2]{bernstein2005matrix} that the Laplace transform of any matrix $A$ is given by its resolvent, that is, 
% \begin{equation}
% \label{matrixexpint}
$(z I-A)^{-1}=\int_0^\infty{\ee^{-zu}\ee^{Au}\dd u}$,
% \end{equation}
for any complex number $z$. We are now ready to compute
\begin{align*}
\frac{1}{2\pi}H(\ii\omega)\SigL H(-\ii\omega)^T=&\frac{1}{2\pi}C\left[\int_0^\infty{\ee^{-\ii\omega u}\ee^{Au}\dd u}B\SigL B^T\int_0^\infty{\ee^{\ii \omega v}\ee^{A^Tv}\dd v}\right]\dd hC^T.
\end{align*}
Introducing the new variable $h=u-v$, and using \cref{lemma-inteABSigmaBeA}, this becomes
\begin{align*}
% &\frac{1}{2\pi}C\left[\int_0^\infty{\int_{-v}^\infty{\ee^{-\ii\omega h}\ee^{Ah}\ee^{Av}B\SigL B^T\ee^{A^Tv}\dd h\dd v}}\right]C^T\\
&\frac{1}{2\pi}C\left[\int_0^\infty{\int_{0}^\infty{\ee^{-\ii\omega h}\ee^{Ah}\ee^{Av}B\SigL B^T\ee^{A^Tv}\dd h\dd v}}+\int_0^\infty{\int_{-v}^0{\ee^{-\ii\omega h}\ee^{Ah}\ee^{Av}B\SigL B^T\ee^{A^Tv}\dd h\dd v}}\right]C^T\\
=&\frac{1}{2\pi}C\left[\int_{0}^\infty{\ee^{-\ii\omega h}\ee^{Ah}\Gamma_0\dd h}+\int_{-\infty}^0{\ee^{-\ii\omega h}\Gamma_0\ee^{-A^Th}\dd h}\right]C^T.
\end{align*}
By \cref{eq-Covssm-QML} and the fact that the spectral density and the autocovariance function of a stochastic process are Fourier duals of each other, the last expression is equal to $(2\pi)^{-1}\int_{-\infty}^\infty{\ee^{-\ii\omega h}\gamma_{\Y}(h)\dd h}=f_{\Y}(\omega)$, which completes the proof.
\end{proof}

A converse of \cref{SSMspecfac}, which will be useful in our later discussion of identifiability, is the Spectral Factorization Theorem. Its proof can be found in \citet[Theorem 1.10.1]{rozanov1967srp}.
% and also in \citet[Theorem 4.1.4]{caines1988linear}.
\begin{theorem}
\label{specfac}
Every positive definite rational matrix function $f\in \SS_d^+\left(\C\{\omega\}\right)$ of full rank can be factorized as $f(\omega)=(2\pi)^{-1}W(\ii\omega)W(-\ii\omega)^T$, where the rational matrix function $z\mapsto W(z)\in M_{d,N}\left(\R\{z\}\right)$ has full rank and is, for fixed $N$, uniquely determined up to an orthogonal transformation $W(z)\mapsto W(z)O$, for some orthogonal $N\times N$ matrix $O$.
\end{theorem}

\subsection{Equidistant observations}
\label{section-sampling}
We now turn to properties of the sampled process $\Y^{(h)}=(\Y^{(h)}_n)_{n\in\mathbb{Z}}$ which is defined by $\Y^{(h)}_n=\Y(nh)$ and represents observations of the process $\Y$ at equally spaced points in time. A very fundamental observation is that the linear state space structure of the continuous\hyp{}time process is preserved under sampling, as detailed in the following proposition. Of particular importance is the explicit formula \labelcref{discspecdens} for the spectral density of the sampled process $\Y^{(h)}$.
\begin{proposition}[{partly \citet[Lemma 5.1]{schlemmmixing2010}}]
\label{prop-sampledSSM}
Assume that $\Y$ is the output process of the state space model \labelcref{eq-ssm-QML}. Then the sampled process $\Y^{(h)}$ has the state space representation
\begin{equation}
\label[pluralequation]{eq-sampledSSM}
\X_n = \ee^{Ah}\X_{n-1}+\NN^{(h)}_n,\quad \NN^{(h)}_n= \int_{(n-1)h}^{nh}{\ee^{A(nh-u)}B\dd \Lb(u)},\quad \Y^{(h)}_n = C\X^{(h)}_n.
\end{equation}
The sequence $\left(\NN^{(h)}_n\right)_{n\in\Z}$ is i.\,i.\,d.\ with mean zero and covariance matrix $\cancel\Sigma^{(h)}=\int_0^h{\ee^{Au}B\SigL  B^T\ee^{A^Tu}\dd u}$. Moreover, the spectral density of $\Y^{(h)}$, denoted by $f_{\Y}^{(h)}$, is given by
\begin{equation}
\label{discspecdens}
f_{\Y}^{(h)}(\omega)=C \left(\ee^{\ii \omega}\I_N-\ee^{A h}\right)^{-1}\cancel\Sigma^{(h)}\left(\ee^{-\ii\omega}\I_N-\ee^{A^T h}\right)^{-1}C^T;
\end{equation}
in particular, $ f_{\Y}^{(h)}:[-\pi,\pi]\to \SS^+_d\left(\R\left\{\ee^{\ii \omega}\right\}\right)$ is a rational matrix function.
\end{proposition}
\begin{proof}
The first part is \citet[Lemma 5.1]{schlemmmixing2010} and Expression \labelcref{discspecdens} follows from \citet[Eq. (10.4.43)]{hamilton1994tsa}.
\end{proof}
In the following we derive conditions for the sampled state space model \labelcref{eq-sampledSSM} to be minimal in the sense that the process $\Y^{(h)}$ is not the output process of any state space model of dimension less than $N$, and for the noise covariance matrix $\cancel\Sigma^{(h)}$ to be non\hyp{}singular. We begin by recalling some well\hyp{}known notions from discrete\hyp{}time realization and control theory. For a detailed account we refer to \citet{aastrom1970isc,sontag1998mathematical}, which also explain the origin of the terminology.
% caines1988linear
\begin{definition}
% [Algebraic realization]
Let $H\in M_{d,m}(\R\{z\})$ be a rational matrix function. A matrix triple $(A,B,C)$ is called an {\it algebraic realization of $H$ of dimension $N$} if $H(z) = C(z \I_N-A)^{-1}B$, where $A\in M_N(\R)$, $B\in M_{N,m}(\R)$, and $C\in M_{d,N}(\R)$.
\end{definition}
Every rational matrix function has many algebraic realizations of various dimensions. A particularly convenient class are the ones of minimal dimension, which have a number of useful properties.
\begin{definition}
% [Minimality]
\label{def-minreal}
Let $H\in M_{d,m}(\R\{z\})$ be a rational matrix function. A {\it minimal realization} of $H$ is an algebraic realization of $H$ of dimension smaller than or equal to the dimension of every other algebraic realization of $H$. The dimension of a minimal realization of $H$ is the {\it McMillan degree} of $H$.
\end{definition}
Two other important properties of algebraic realizations, which are related to the notion of minimality and play a key role in the study of identifiability, are introduced in the following definitions.
\begin{definition}
% [Controllability]
\label{DefControl}
An algebraic realization $(A,B,C)$ of dimension $N$ is {\it controllable} if the {\it controllability} matrix $\mathscr{C}=\left[\begin{array}{cccc}B & AB & \cdots & A^{n-1}B\end{array}\right]\in M_{m,mN}(\R)$ has full rank.
\end{definition}
\begin{definition}
% [Observability]
\label{DefObserv}
An algebraic realization $(A,B,C)$ of dimension $N$ is {\it observable} if the {\it observability} matrix $\mathscr{O}=\left[\begin{array}{cccc}C^T & (CA)^T & \cdots & (CA^{n-1})^T\end{array}\right]^T\in M_{dN,N}(\R)$ has full rank.
% The matrix $\mathscr{O}$ is called the {\it observability matrix}.
\end{definition}
We will often say that a state space system \labelcref{eq-ssm-QML} is minimal, controllable or observable if the corresponding transfer function has this property. In the context of ARMA processes these concepts have been used to investigate the non\hyp{}singularity of the Fisher information matrix \citep{klein2006bezoutian}. The next theorem characterizes minimality in terms of controllability and observability.
\begin{theorem}[{{\citet[Theorem 2.3.3]{hannan1987stl}}}]
\label{theominimality}
A realization $(A,B,C)$ is minimal if and only if it is both controllable and observable.
\end{theorem}
\begin{lemma}
\label{subspaces}
For all matrices $A\in M_N(\R)$, $B\in M_{N,m}(\R)$, $\Sigma\in \SS^{++}_m(\R)$, and every real number $t>0$, the linear subspaces $\im \left[B, AB, \ldots, A^{N-1}B\right]$ and $\im \int_0^t{\ee^{Au}B\Sigma B^T \ee^{A^Tu}\dd u}$ are equal.
%If, in addition, $A$ satisfies \cref{eigvalcon}, the limit $\lim_{t\to\infty}\int_0^t{\ee^{Au}B\Sigma B^T \ee^{A^Tu}\dd u}$ exists and the subspace $\im \int_0^\infty{\ee^{Au}B\Sigma B^T \ee^{A^Tu}\dd u}$ is equal to {\it i)} and {\it ii)}.
\end{lemma}
\begin{proof}
The assertion is a straightforward generalization of \citet[Lemma 12.6.2]{bernstein2005matrix}. 
\end{proof}

\begin{corollary}
\label{coro-fullrankcov}
If the triple $(A,B,C)$ is minimal of dimension $N$, and $\Sigma$ is positive definite, then the $N\times N$ matrix $\cancel{\Sigma}=\int_0^h{\ee^{Au}B\Sigma B^T \ee^{A^Tu}\dd u}$ has full rank $N$.
\end{corollary}
\begin{proof}
By \cref{theominimality}, minimality of $(A,B,C)$ implies controllability, and by \cref{subspaces}, this is equivalent to $\cancel{\Sigma}$ having full rank.
\end{proof}
\begin{proposition}
\label{kalBerThm}
 Assume that $\Y$ is the $d$\hyp{}dimensional output process of the state space model \labelcref{eq-ssm-QML} with $(A,B,C)$ being a minimal realization of McMillan degree $N$. Then a sufficient condition for  the sampled process $\Y^{(h)}$ to have the same McMillan degree, is the {\it Kalman--Bertram criterion}
\begin{equation}
\label{eq-KalBerCrit}
\lambda-\lambda'\neq 2h^{-1}\pi \ii k,\qquad \forall (\lambda, \lambda')\in\sigma(A)\times\sigma(A),\qquad \forall k\in\mathbb{Z}\backslash\{0\}.
\end{equation}
\end{proposition}
\begin{proof}
We will prove the assertion by showing that the $N$\hyp{}dimensional state space representation \labelcref{eq-sampledSSM} is both controllable and observable, and thus, by \cref{theominimality}, minimal. Observability has been shown in \citet[Proposition 5.2.11]{sontag1998mathematical} using the Hautus criterion \citep{hautus1969controllability}. The key ingredient in the proof of controllability is \cref{coro-fullrankcov}, where we showed that the autocovariance matrix $\cancel{\Sigma}^{(h)}$ of $\NN^{(h)}_n$, given in \cref{prop-sampledSSM}, has full rank; this shows that the representation \labelcref{eq-sampledSSM} is indeed minimal and completes the proof.
\end{proof}
Since, by \citet[Theorem 2.3.4]{hannan1987stl}, minimal realizations are unique up to a change of basis $(A,B,C)\mapsto (TAT^{-1},TB,CT^{-1})$, for some non\hyp{}singular $N\times N$ matrix $T$, and such a transformation does not change the eigenvalues of $A$, the criterion \labelcref{eq-KalBerCrit} does not depend on what particular triple $(A,B,C)$ one chooses. Uniqueness of the principal logarithm
% \citep[Theorem 1.31]{higham2008functions}
implies the following.
\begin{lemma}
\label{princlog}
Assume that the matrices $A,B\in M_N(\R)$ satisfy $\ee^{hA}=\ee^{hB}$ for some $h>0$. If the spectra $\sigma_A,\sigma_B$ of $A,B$ satisfy $|\imag\lambda|<\pi/h$ for all $\lambda\in\sigma_A\cup\sigma_B$, then $A=B$.
\end{lemma}
\begin{lemma}
\label{injectivity}
Assume that $A\in M_N(\R)$ satisfies \cref{assumEigen1-QML}. For every $h>0$, the linear map $\mathscr{M}:M_N(\R)\to M_N(\R)$, $M\mapsto\int_0^h{\ee^{Au} M \ee^{A^Tu}\dd u}$ is injective.
\end{lemma}
\begin{proof}
If we apply the vectorization operator $\vec:M_N(\R)\to\R^{N^2}$ and use the well\hyp{}known identity \citep[Proposition 7.1.9]{bernstein2005matrix} $\vec(UVW)=(W^T\otimes U)\vec(V)$ for matrices $U,V$ and $W$ of appropriate dimensions, we obtain the induced linear operator
\begin{equation*}
\vec\circ\mathscr{M}\circ\vec^{-1}:\R^{N^2}\to\R^{N^2},\quad \vec M\mapsto \int_0^h{\ee^{Au}\otimes \ee^{Au}\dd u}\vec M.
\end{equation*}
To prove the claim that the operator $\mathscr{M}$ is injective, it is thus sufficient to show that the matrix $\mathscr{A}\coloneqq \int_0^h{\ee^{Au}\otimes \ee^{Au}\dd u}\in M_{N^2}(\R)$ is non\hyp{}singular. We write $A\oplus A\coloneqq A\otimes \I_N+\I_N\otimes A$. By \citet[Fact 11.14.37]{bernstein2005matrix}, $\mathscr{A} = \int_0^h{\ee^{(A\oplus A)u}\dd u}$ and since $\sigma(A\oplus A)=\{\lambda+\mu:\lambda,\mu\in\sigma(A)\}$ \citep[Proposition 7.2.3]{bernstein2005matrix}, \cref{assumEigen1-QML} implies that all eigenvalues of the matrix $A\oplus A$ have strictly negative real parts; in particular, $A\oplus A$ is invertible. Consequently, it follows from \citet[Fact 11.13.14]{bernstein2005matrix} that $\mathscr{A} = (A\oplus A)^{-1}\left[\ee^{(A\oplus A)h}-\I_{N^2}\right]$. Since, for any matrix $M$, it holds that $\sigma(\ee^M)=\{\ee^\lambda, \lambda\in\sigma(M)\}$ \citep[Proposition 11.2.3]{bernstein2005matrix}, the spectrum of $\ee^{(A\oplus A)h}$ is a subset of the open unit disk, and it follows that $\mathscr{A}$ is invertible.
\end{proof}

\subsection{Overcoming the aliasing effect}
\label{section-identifiability}

One goal in this paper is the estimation of multivariate CARMA processes or, equivalently, continuous\hyp{}time state space models, based on discrete observations. In this brief section we concentrate on the issue of identifiability, and we derive sufficient conditions that prevent redundancies from being introduced into an otherwise properly specified model by the process of sampling, an effect known as aliasing \citep{hansen1983dimensionality}.
% mccrorie2003theproblem

For ease of notation we choose to parametrize the state matrix, the input matrix, and the observation matrix of the state space model \labelcref{eq-ssm-QML}, as well as the driving L\'evy process $\Lb$; from these one can always obtain an autoregressive and a moving average polynomial which describe the same process by applying a left matrix fraction decomposition to the corresponding transfer function
% , see \citet{patel1981computation} and the upcoming \cref{EchelonSSM,EchelonMCARMA}.
We hence assume that there is some compact parameter set $\Theta\subset\R^r$, and that, for each $\bth\in\Theta$, one is given matrices $A_{\bth} $, $B_{\bth} $ and $C_{\bth} $ of matching dimensions, as well as a L\'evy process $\Lb_{\bth}$. A basic assumption is that we always work with second order processes (cf. \cref{assum-levy2}).
\begin{assumptionparaC}
\label{assum-2momentsCT}
For each $\bth\in\Theta$, it holds that $\E\Lb_{\bth}=\bzero_m$, that $\E\left\|\Lb_{\bth}(1)\right\|^2$ is finite, and that the covariance matrix $\SigL_{\bth}=\E\Lb_{\bth}(1)\Lb_{\bth}(1)^T$ is non\hyp{}singular.
\end{assumptionparaC}
To ensure that the model corresponding to $\bth$ describes a stationary output process we impose the analogue of \cref{assumEigen1-QML}.
\begin{assumptionparaC}
\label{assum-stabilityCT}
For each $\bth\in\Theta$, the eigenvalues of $A_{\bth}$ have strictly negative real parts.
\end{assumptionparaC}
Next, we restrict the model class to minimal algebraic realizations of a fixed McMillan degree.
\begin{assumptionparaC}
\label{assum-minimalityCT}
For all $\bth\in\Theta$, the triple $\left(A_{\bth} ,B_{\bth} ,C_{\bth} \right)$ is minimal with McMillan degree $N$.
\end{assumptionparaC}
Since we shall base the inference on a QML approach and thus on second\hyp{}order properties of the observed process, we require the model class to be identifiable from these available information according to the following definitions.
\begin{definition}
Two stochastic processes, irrespective of whether their index sets are continuous or discrete, are {\it $L^2$\hyp{}observationally equivalent} if their spectral densities are the same. 
\end{definition}
\begin{definition}
\label{IdentSpecDens}
A family $\left(\Y_{\bth},\bth\in\Theta\right)$ of continuous\hyp{}time stochastic processes is {\it identifiable from the spectral density} if, for every $\bth_1\neq\bth_2$, the two processes $\Y_{\bth_1}$ and $\Y_{\bth_2}$ are not $L^2$\hyp{}observationally equivalent. It is {\it $h$\hyp{}identifiable from the spectral density}, $h>0$, if, for every $\bth_1\neq\bth_2$, the two sampled processes $\Y_{\bth_1}^{(h)}$ and $\Y_{\bth_2}^{(h)}$ are not $L^2$\hyp{}observationally equivalent.
\end{definition}
\begin{assumptionparaC}
\label{assum-identifiabilityCT}
The collection of output processes $K(\Theta)\coloneqq\left(\Y_{\bth},\bth\in\Theta\right)$ corresponding to the state space models $\left(A_{\bth} ,B_{\bth} ,C_{\bth} ,\Lb_{\bth}\right)$ is  identifiable from the spectral density.
\end{assumptionparaC}
Since we shall use only discrete, $h$\hyp{}spaced observations of $\Y$, it would seem more natural to impose the stronger requirement that $K(\Theta)$ be $h$\hyp{}identifiable. We will see, however, that this is implied by the previous assumptions if we additionally assume that the following holds.
\begin{assumptionparaC}
\label{assum-spectrumCT}
For all $\bth\in\Theta$, the spectrum of $A_{\bth} $ is a subset of $\left\{z\in\C:-\pi/h<\imag z < \pi/h\right\}$.
\end{assumptionparaC}

\begin{theorem}[Identifiability]
\label{discident}
Assume that $\Theta\supset\bth\mapsto\left(A_{\bth} ,B_{\bth} ,C_{\bth} ,\SigL_{\bth}\right)$ is a parametrization of continuous\hyp{}time state space models satisfying \cref{assum-2momentsCT,assum-stabilityCT,assum-minimalityCT,assum-identifiabilityCT,assum-spectrumCT}. Then the corresponding collection of output processes $K(\Theta)$ is $h$\hyp{}identifiable from the spectral density.
\end{theorem}
\begin{proof}
We will show that for every $\bth_1,\bth_2\in\Theta$, $\bth_1\neq\bth_2$, the sampled output processes $\Y^{(h)}_{\bth_1}$ and $\Y^{(h)}_{\bth_2}{(h)}$ are not $L^2$\hyp{}observationally equivalent. Suppose, for the sake of contradiction, that the spectral densities of the sampled output processes were the same. Then the Spectral Factorization Theorem (\cref{specfac}) would imply that there exists an orthogonal $N\times N$ matrix $O$ such that
\begin{equation*}
C_{\bth_1} (\ee^{\ii \omega}\I_N-\ee^{A_{\bth_1} h})\cancel\Sigma^{(h),1/2}_{\bth_1}O=C_{\bth_2} (\ee^{\ii \omega}\I_N-\ee^{A_{\bth_2} h})\cancel\Sigma^{(h),1/2}_{\bth_2},\quad -\pi\leq\omega\leq\pi,
\end{equation*}
where $\cancel\Sigma^{(h),1/2}_{\bth_i}$ are the unique positive definite matrix square roots of the matrices $\int_0^h{\ee^{A_{\bth_i} u}B_{\bth_i} \SigL_{\bth_i}B_{\bth_i}^T\ee^{A_{\bth_i}^Tu}\dd u}$, defined by spectral calculus. This means that the two triples
\begin{equation*}
\left(\ee^{A_{\bth_1} h},\cancel\Sigma^{(h),1/2}_{\bth_1}O,C_{\bth_1} \right)\quad\text{and}\quad\left(\ee^{A_{\bth_2} h},\cancel\Sigma^{(h),1/2}_{\bth_2},C_{\bth_2} \right)
\end{equation*}
are algebraic realizations of the same rational matrix function. Since \cref{assum-spectrumCT} clearly implies the Kalman--Bertram criterion \labelcref{eq-KalBerCrit}, it follows from \cref{kalBerThm} in conjunction with \cref{assum-minimalityCT} that these realizations are minimal, and hence from \citet[Theorem 2.3.4]{hannan1987stl} that there exists an invertible matrix $T\in M_N(\R)$ satisfying
\begin{equation}
\label{eq-threeconditions}
\ee^{A_{\bth_1} h}=T^{-1}\ee^{A_{\bth_2} h}T,\qquad\cancel\Sigma^{(h),1/2}_{\bth_1}O =T^{-1}\cancel\Sigma^{(h),1/2}_{\bth_2},\qquad C_{\bth_1}  =C_{\bth_2} T.
\end{equation}
It follows from the power series representation of the matrix exponential that $T^{-1}\ee^{A_{\bth_2} h}T$ equals $\ee^{T^{-1}A_{\bth_2}T h}$. Under \cref{assum-spectrumCT}, the first equation in conjunction with \cref{princlog} therefore implies that $A_{\bth_1} =T^{-1}A_{\bth_2} T$. Using this, the second of the three equations \labelcref{eq-threeconditions} gives \begin{equation*}
\cancel\Sigma^{(h)}_{\bth_1} =\int_0^h{\ee^{A_{\bth_1} u}\left(T^{-1}B_{\bth_2}\right)\SigL_{\bth_2}\left(T^{-1}B_{\bth_2}\right)^T\ee^{A_{\bth_1}^Tu}\dd u},
\end{equation*}
which, by \cref{injectivity}, implies that $(T^{-1}B_{\bth_2} )\SigL_{\bth_2}(T^{-1}B_{\bth_2} )^T = B_{\bth_1}\Sigma_{\bth_1}^{\Lb}B_{\bth_1}^T$. Together with the last of the equations \labelcref{eq-threeconditions} and \cref{prop-sampledSSM} it follows that $f_{\bth_1}=f_{\bth_2}$, which contradicts \cref{assum-identifiabilityCT} that $\Y_{\bth_1}$ and $\Y_{\bth_2}$ are not $L^2$\hyp{}observationally equivalent.
\end{proof}

\subsection{Asymptotic properties of the QML estimator}
\label{section-asymptoticmcarma}

In this section we apply the theory that we developed in \cref{section-QMLDTSSM} for the QML estimation of general discrete\hyp{}time linear state space models to the estimation of continuous\hyp{}time linear state space models or, equivalently, multivariate CARMA processes. We have already seen that a discretely observed MCARMA process can be represented by a discrete\hyp{}time state space model and that, thus, a parametric family of MCARMA processes induces a parametric family of discrete\hyp{}time state space models. \Cref{eq-sampledSSM} show that sampling with spacing $h$ maps the continuous\hyp{}time state space models $\left(A_{\bth},B_{\bth},C_{\bth},\Lb_{\bth}\right)_{\bth\in\Theta}$ to the discrete\hyp{}time state space models
\begin{equation}
\label{eq-sampledSSmfamiliy}
\left(\ee^{A_{\bth}h},C_{\bth},\NN^{(h)}_{\bth},\bzero\right)_{\bth\in\Theta},\quad \NN^{(h)}_{\bth,n}=\int_{(n-1)h}^{nh}\ee^{A_{\bth}u}B_{\bth}\dd\Lb_{\bth}(u).
\end{equation}
which are not in the innovations form \labelcref{eq-ssminnoform}. The QML estimator $\hat\bth^{L,(h)}$ is defined by \cref{eq-hatlikelihood}, applied to the state space model \labelcref{eq-sampledSSmfamiliy}, that is
\begin{subequations}
\label[pluralequation]{eq-DefhatbthLh}
\begin{align}
\hat\bth^{L,(h)} =& \argmin_{\bth\in\Theta}{\widehat{\mathscr{L}}^{(h)}(\bth,\y^{L,(h)})},\\
\widehat{\mathscr{L}}^{(h)}(\bth,\y^{L,(h)})=&\sum_{n=1}^L{\left[d\log{2\pi} + \log\det V^{(h)}_{\bth} + \hat\beps^{(h),T}_{\bth,n}V_{\bth}^{(h),-1}\hat\beps^{(h)}_{\bth,n}\right]},
\end{align}
\end{subequations}
where $\hat\beps^{(h)}_{\bth}$ are the pseudo\hyp{}innovations of the observed process $\Y^{(h)}=\Y^{(h)}_{\bth_0}$, which are computed from the sample $\y^{L,(h)} = (\Y^{(h)}_1,\ldots,\Y^{(h)}_L)$ via the recursion
\begin{equation*}
\hat\X_{\bth,n} = \left(\ee^{A_{\bth}h}-K^{(h)}_{\bth}C_{\bth}\right)\hat\X_{\bth,n-1}+K^{(h)}_{\bth}\Y^{(h)}_{n-1},\quad \hat\beps^{(h)}_{\bth,n} = \Y^{(h)}_n - C_{\bth}\hat\X_{\bth,n},\quad n\in\N.
\end{equation*}
The initial value $\hat\X_{\bth,1}$ may be chosen in the same ways as in the discrete\hyp{}time case. The steady-state Kalman gain matrices $K^{(h)}_{\bth}$ and pseudo-covariances $V^{(h)}_{\bth}$ are computed as functions of the unique positive definite solution $\Omega^{(h)}_{\bth}$ to the discrete\hyp{}time algebraic Riccati equation
\begin{equation*}
\Omega^{(h)}_{\bth} = \ee^{A_{\bth}h}\Omega^{(h)}_{\bth}\ee^{A^T_{\bth}h}+\cancel\Sigma^{(h)}_{\bth}-\left[\ee^{A_{\bth}h}\Omega^{(h)}_{\bth} C_{\bth}^T\right]\left[C_{\bth}\Omega^{(h)}_{\bth} C_{\bth}^T\right]^{-1}\left[\ee^{A_{\bth}h}\Omega^{(h)}_{\bth} C_{\bth}^T\right]^T,
\end{equation*}
namely
\begin{equation*}
K^{(h)}_{\bth} = \left[\ee^{A_{\bth}h}\Omega^{(h)}_{\bth} C_{\bth}^T\right]\left[C_{\bth}\Omega^{(h)}_{\bth} C_{\bth}^T\right]^{-1},\quad V^{(h)}_{\bth} = C_{\bth}\Omega^{(h)}_{\bth}C_{\bth}^T.
\end{equation*}
In order to obtain the asymptotic normality of the QML estimator for multivariate CARMA processes, it is therefore only necessary to make sure that \cref{assum-compact,assum-smoothparam,assum-smoothparam3,assum-stability,assum-2moments,assum-4moments,assum-mixing,assum-interior,assum-identifiability1,assum-identifiabilityFisher} hold for the model \labelcref{eq-sampledSSmfamiliy}. The discussion of identifiability in the previous section allows us to specify accessible conditions on the parametrization of the continuous\hyp{}time model under which the QML estimator is strongly consistent. In addition to the identifiability assumptions \labelcref{assum-minimalityCT,assum-identifiabilityCT,assum-spectrumCT}, we impose the following conditions.

\begin{assumptionparaC}
\label{assum-compactCT}
The parameter space $\Theta$ is a compact subset of $\R^r$.
\end{assumptionparaC}

\begin{assumptionparaC}
\label{assum-smoothparamCT}
The functions $\bth\mapsto A_{\bth}$, $\bth\mapsto B_{\bth}$, $\bth\mapsto C_{\bth}$, and $\bth\mapsto \SigL_{\bth}$ are continuous. Moreover, for each $\bth\in\Theta$, the matrix $C_{\bth}$ has full rank.
\end{assumptionparaC}

\begin{lemma}
\label{lemma-C47D14}
\Cref{assum-minimalityCT,assum-compactCT,assum-smoothparamCT,assum-stabilityCT,assum-2momentsCT} imply that the family $\left(\ee^{A_{\bth}h},C_{\bth},\NN^{(h)}_{\bth},\bzero\right)_{\bth\in\Theta}$ of discrete\hyp{}time state space models  satisfies \cref{assum-compact,assum-smoothparam,assum-stability,assum-2moments}.
\end{lemma}
\begin{proof}
\Cref{assum-compact} is clear. \Cref{assum-smoothparam} follows from the observation that the functions $A\mapsto \ee^A$ and $(A,B,\Sigma)\mapsto\int_0^h{\ee^{Au}B\Sigma B^T\ee^{A^Tu}}\dd u$ are continuous. By \cref{assum-stabilityCT,assum-compactCT,assum-smoothparamCT}, and the fact that the eigenvalues of a matrix are continuous functions of its entries, it follows that there exists a positive real number $\epsilon$ such that, for each $\bth\in\Theta$, the eigenvalues of $A_{\bth}$ have real parts less than or equal to $-\epsilon$. The observation that the eigenvalues of $\ee^A$ are given by the exponentials of the eigenvalues of $A$ thus shows that \cref{assum-stability}, \labelcref{assum-stabilityF} holds with $\rho\coloneqq\ee^{-\epsilon h}<1$. \Cref{assum-2momentsCT} that the matrices $\SigL_{\bth}$ are non\hyp{}singular and the minimality assumption \labelcref{assum-minimalityCT} imply by \cref{coro-fullrankcov} that the noise covariance matrices $\cancel\Sigma^{(h)}_{\bth}=\E\NN^{(h)}_{\bth,n}\NN^{(h),T}_{\bth,n}$ are non\hyp{}singular, and thus \cref{assum-stability}, \labelcref{assum-stabilityQS} holds. Further, by \cref{prop-Kalmanfilter}, the matrices $\Omega_{\bth}$ are non\hyp{}singular, and so are, because the matrices $C_{\bth}$ are assumed to be of full rank, the matrices $ V_{\bth}$; this means that \cref{assum-stability}, \labelcref{assum-stabilityV} is satisfied. \Cref{assum-2moments} is a consequence of \cref{prop-sampledSSM}, which states that the noise sequences $\NN_{\bth}$ are i.\,i.\,d.\, and in particular ergodic; their second moments are finite because of \cref{assum-2momentsCT}.
\end{proof}

%% CLT conditions

In order to be able to show that the QML estimator $\hat\bth^{L,(h)}$ is asymptotically normally distributed, we impose the following conditions in addition to the ones described so far.

\begin{assumptionparaC}
\label{assum-interiorCT}
The true parameter value $\bth_0$ is an element of the interior of $\Theta$.
\end{assumptionparaC}

\begin{assumptionparaC}
\label{assum-smoothparam3CT}
The functions $\bth\mapsto A_{\bth}$, $\bth\mapsto B_{\bth}$, $\bth\mapsto C_{\bth}$, and $\bth\mapsto \SigL_{\bth}$ are three times continuously differentiable.
\end{assumptionparaC}

\begin{assumptionparaC}
\label{assum-4momentsCT}
There exists a positive number $\delta$ such that $\E\left\|\Lb_{\bth_0}(1)\right\|^{4+\delta}<\infty$. 
\end{assumptionparaC}

\begin{lemma}
\label{lemma-C89D78}
\Cref{assum-interiorCT,assum-smoothparam3CT,assum-4momentsCT} imply that \cref{assum-interior,assum-smoothparam3,assum-4moments} hold for the model \labelcref{eq-sampledSSmfamiliy}.
\end{lemma}
\begin{proof}
\Cref{assum-interior} is clear. \Cref{assum-smoothparam3} follows from the fact that the functions $A\mapsto \ee^A$ and $(A,B,\Sigma)\mapsto\int_0^h{\ee^{Au}B\Sigma B^T\ee^{A^Tu}}\dd u$ are not only continuous, but infinitely often differentiable. For \cref{assum-4moments} we need to show that the random variables $\NN\coloneqq\NN_{\bth_0,1}$ have bounded $(4+\delta)$th absolute moments. It follows from \citet[Theorem 2.7]{rajput1989spectral} that $\NN$ is infinitely divisible with characteristic triplet $(\bgamma,\Sigma,\nu)$, and that
\begin{equation*}
\int_{\left\|\bx\right\|\geq 1}{\left\|\bx\right\|^{4+\delta}\nu(\dd\bx)} \leq \int_0^1\left\|\ee^{A_{\bth_0}(h-s)}B_{\bth}\right\|^{4+\delta}\dd s\int_{\left\|\bx\right\|\geq 1}{\left\|\bx\right\|^{4+\delta}\nu^{\Lb_{\bth_0}}{\bth}(\dd\bx)}.
\end{equation*}
The first factor on the right side is finite by \cref{assum-compactCT,assum-smoothparam3CT}, the second by \cref{assum-4momentsCT} and the equivalence of finiteness of the $\alpha$th absolute moment of an infinitely divisible distribution and finiteness of the $\alpha$th absolute moments of the corresponding L\'evy measure restricted to the exterior of the unit ball \citep[Corollary 25.8]{sato1991lpa}. The same corollary shows that $\E\left\|\NN\right\|^{4+\delta}<\infty$ and thus \cref{assum-4moments}.
\end{proof}
Our final assumption is the analogue of \cref{assum-identifiabilityFisher}. It will ensure that the Fisher information matrix of the QML estimator $\hat\bth^{L,(h)}$ is non\hyp{}singular by imposing a non\hyp{}degeneracy condition on the parametrization of the model.

\begin{assumptionparaC}
\label{assum-identifiabilityFisherCT}
There exists a positive index $j_0$ such that the $\left[(j_0+2)d^2\right]\times r$ matrix
\begin{equation*}
\nabla_{\bth}\left(\begin{array}{c}\left[\I_{j_0+1}\otimes K^{(h),T}_{\bth}\otimes C_{\bth}\right]\left[\begin{array}{cccc}\left(\vec \ee^{\I_N h}\right)^T& \left(\vec \ee^{A_{\bth}h}\right)^T & \cdots & \left(\vec \ee^{A_{\bth}^{j_0}h}\right)^T\end{array}\right]^T\\\vec V_{\bth}\end{array}\right)_{\bth=\bth_0}
\end{equation*}
has rank $r$.
\end{assumptionparaC}

\begin{theorem}[Consistency and asymptotic normality of $\hat\bth^{L,(h)}$]
\label{theorem-CLTmcarma}
Assume that $\left(A_{\bth},B_{\bth},C_{\bth},\Lb_{\bth}\right)_{\bth\in\Theta}$ is a parametric family of continuous\hyp{}time state space models, and denote by $\y^{L,(h)}=(\Y^{(h)}_{\bth_0.1},\ldots,\Y^{(h)}_{\bth_0.L})$ a sample of length $L$ from the discretely observed output process corresponding to the parameter value $\bth_0\in\Theta$. Under \cref{assum-compactCT,assum-smoothparamCT,assum-stabilityCT,assum-2momentsCT,assum-minimalityCT,assum-identifiabilityCT,assum-spectrumCT} the QML estimator $\hat\bth^{L,(h)}=\argmin_{\bth\in\Theta}\widehat{\mathscr{L}}(\bth,\y^{L,(h)})$ is strongly consistent, i.\,e.\
\begin{equation}
\hat\bth^{L,(h)}\xrightarrow[L\to\infty]{\text{a.\,s.}}\bth_0.
\end{equation}
If, moreover, \cref{assum-interiorCT,assum-smoothparam3CT,assum-4momentsCT,assum-identifiabilityFisherCT} hold, then $\hat\bth^{L,{(h)}}$ is asymptotically normally distributed, i.\,e.\
\begin{equation}
\label{eq-CLTmcarma}
\sqrt{L}\left(\hat\bth^{L,(h)}-\bth_0\right)\xrightarrow[L\to\infty]{d}\mathscr{N}(\bzero,\Xi),
\end{equation}
where the asymptotic covariance matrix $\Xi=J^{-1}IJ^{-1}$ is given by
\begin{align}
I =& \lim_{L\to\infty}L^{-1}\Var\left(\nabla_{\bth}\mathscr{L}\left(\bth_0,\y^L\right)\right),\quad J = \lim_{L\to\infty}L^{-1}\nabla^2_{\bth}\mathscr{L}\left(\bth_0,\y^L\right).
\end{align}
\end{theorem}
\begin{proof}
Strong consistency of $\hat\bth^{L,(h)}$ is a consequence of \cref{theorem-consistency} if we can show that the parametric family $\left(\ee^{A_{\bth}h},C_{\bth},\NN_{\bth},\bzero\right)_{\bth\in\Theta}$ of discrete\hyp{}time state space models satisfies \cref{assum-compact,assum-smoothparam,assum-stability,assum-2moments,assum-identifiability1}. The first four of these are shown to hold in \cref{lemma-C47D14}. For the last one, we observe that, by \cref{lemma-identifiabilityspectralDT}, \cref{assum-identifiability1} is equivalent to the family of state space models \labelcref{eq-sampledSSmfamiliy} being identifiable from the spectral density. Under \cref{assum-minimalityCT,assum-identifiabilityCT,assum-spectrumCT} this is guaranteed by \cref{discident}.

In order to prove \cref{eq-CLTmcarma}, we shall apply \cref{theorem-hatbthLCLT} and therefore need to verify \cref{assum-interior,assum-smoothparam3,assum-4moments,assum-mixing,assum-identifiabilityFisher} for the state space models $\left(\ee^{A_{\bth}h},C_{\bth},\NN_{\bth},\bzero\right)_{\bth\in\Theta}$. The first three hold by \cref{lemma-C89D78}, the last one as a reformulation of \cref{assum-identifiabilityFisherCT}. \Cref{assum-mixing}, that the strong mixing coefficients $\alpha$ of a sampled multivariate CARMA process satisfy $\sum_{m}[\alpha(m)]^{\delta/(2+\delta)}<\infty$, follows from \cref{assum-2momentsCT} and \citet[Proposition 3.34]{marquardt2007multivariate}, where it was shown that MCARMA processes with a finite logarithmic moment are exponentially strongly mixing.
\end{proof}

\section{Practical applicability}
\label{section-practicalapplicability}
In this section we complement the theoretical results from \cref{section-QMLDTSSM,section-QMLMCARMA} by commenting on their applicability in practical situations. Canonical parametrizations are a classical subject of research about discrete\hyp{}time dynamical systems, and most of the results apply also to the continuous\hyp{}time case; without going into detail we present the basic notions and results about these parametrizations. The assertions of \cref{theorem-CLTmcarma} are confirmed by a simulation study for a bivariate non\hyp{}Gaussian CARMA process. Finally, we estimate the parameters of a CARMA model for a bivariate time series from economics using our QML approach.
\subsection{Canonical parametrizations}
\label{section-Canonical}
We present parametrizations of multivariate CARMA processes that satisfy the identifiability conditions \labelcref{assum-minimalityCT,assum-identifiabilityCT}, as well as the smoothness conditions \labelcref{assum-smoothparamCT,assum-smoothparam3CT}; if, in addition, the parameter space $\Theta$ is restricted so that \cref{assum-stabilityCT,assum-compactCT,assum-interiorCT,assum-spectrumCT} hold, and the driving L\'evy process satisfies \cref{assum-2momentsCT}, the canonically parametrized MCARMA model can be estimated consistently. In order for this estimate to be asymptotically normally distributed, one must additionally impose \cref{assum-4momentsCT} on the L\'evy process and check that \cref{assum-identifiabilityFisherCT} holds -- a condition which we are unable to verify analytically for the general model; for explicit parametrizations, however, it can be checked numerically with moderate computational effort. The parametrizations are well\hyp{}known from the discrete\hyp{}time setting; detailed descriptions with proofs can be found in \citet{hannan1987stl}
% reinsel1997elements,lutkepohl1996sef,deistler1983ppa
or, from a slightly different perspective, in the control theory literature \citep[][and references therein]{gevers1986arma}.
% gevers1984uniquely, guidorzi1975canonical
We begin with a canonical decomposition for rational matrix functions.
\begin{theorem}[{{\citet[Theorem 4.7.5]{bernstein2005matrix}}}]
\label{thmSmith}
Let $H\in M_{d,m}(\R\{z\})$ be a rational matrix function of rank $r$. There exist matrices $S_1\in M_d(\R[z])$ and $S_2\in M_m(\R[z])$ with constant determinant, such that $H=S_1 M S_2$, where
\begin{equation}
M = \left[\begin{array}{cc}\operatorname{diag}\left\{\epsilon_i/\psi_i\right\}_{i=1}^r & 0_{r,m-r}\\0_{d-r,r} & 0_{d-r,m-r}\end{array}\right]\in M_{d,m}(\R\{z\}),
\end{equation}
and $\epsilon_1,\ldots\epsilon_r$, $\psi_1,\ldots,\psi_r\in\R[z]$ are monic polynomials uniquely determined by $H$ satisfying the following conditions: for each $i=1,\ldots,r$, the polynomials $\epsilon_i$ and $\psi_i$ have no common roots, and for each $i=1,\ldots,r-1$, the polynomial $\epsilon_i$ ($\psi_{i+1}$) divides the polynomial $\epsilon_{i+1}$ ($\psi_i$). The triple $(S_1,M,S_2)$ is called the Smith--McMillan decomposition of $H$.
\end{theorem}
The degrees $\nu_i$ of the denominator polynomials $\psi_i$ in the Smith--McMillan decomposition of a rational matrix function $H$ are called the  Kronecker indices of $H$, and they define the vector $\bnu=(\nu_1,\ldots,\nu_d)\in\N^d$, where we set $\nu_k=0$ for $k=r+1,\ldots,d$. They satisfy the important relation $\sum_{i=1}^d{\nu_i} = \delta_M(H)$, where $\delta_M(H)$ denotes the McMillan degree of $H$, i.\,e.\ the smallest possible dimension of an algebraic realization of $H$, see \cref{def-minreal}. For $1\leq i,j\leq d$, we also define the integers $\nu_{ij}=\min\{\nu_i+I_{\{i>j\}},\nu_j\}$, and if the Kronecker indices of the transfer function of an MCARMA process $\Y$ are $\bnu$, we call $\Y$ an $\operatorname{MCARMA}_{\bnu}$ process.
\begin{theorem}[{Echelon state space realization, \citet[Section 3]{guidorzi1975canonical}}]
\label{EchelonSSM}
For natural numbers $d$ and $m$, let $H\in M_{d,m}(\R\{z\})$ be a rational matrix function with Kronecker indices $\bnu=(\nu_1,\ldots,\nu_d)$. Then a unique minimal algebraic realization $(A,B,C)$ of $H$ of dimension $N=\delta_M(H)$ is given by the following structure.
\begin{enumerate}[(i)]
 \item The matrix $A=(A_{ij})_{i,j=1,\ldots,d}\in M_N(\R)$ is a block matrix with blocks $A_{ij}\in M_{\nu_i,\nu_j}(\R)$ given by
\begin{subequations}
\label[pluralequation]{canSSMpara}
\begin{equation}
\label{canSSMparaA}
 A_{ij} =\left(\begin{array}{cccccc}
                0 & \cdots & \cdots & \cdots & \cdots & 0\\
		\vdots & & & & & \vdots \\
		 0 & \cdots & \cdots & \cdots & \cdots & 0\\
		\alpha_{ij,1} & \cdots & \alpha_{ij,\nu_{ij}} & 0 & \cdots & 0
               \end{array}\right) + \delta_{i,j}\left(\begin{array}{ccc}
          0 &  \multicolumn{2}{c}{\multirow{3}{*}{$\I_{\nu_i-1}$}} \\
	  \vdots & &\\
	  0 & &\\
	  0 &\cdots & 0
         \end{array}\right),
 \end{equation}
\item $B=(b_{ij})\in M_{N,m}(\R)$ unrestricted,
\item if $\nu_i>0$, $i=1,\ldots,d$, then
\begin{equation}
C=\left(\begin{array}{cccccccccccccccc}
         1 & 0 & \ldots & 0 					&\vdots& 0 & 0 & \ldots & 0 &\vdots & & \vdots&  \multicolumn{4}{c}{\multirow{2}{*}{$0_{(d-1),\nu_d}$}}	\\
	 \multicolumn{4}{c}{\multirow{2}{*}{$0_{(d-1),\nu_1}$}}&\vdots 	& 1 & 0 & \ldots & 0&\vdots & &\vdots& &&&			\\
	 &&&							&\vdots& \multicolumn{4}{c}{0_{(d-2),\nu_2}}&\vdots & &\vdots& 1 & 0 &\ldots & 0
        \end{array}
\right).
\end{equation}
\end{subequations}
\end{enumerate}
\end{theorem}
If $\nu_i=0$, the elements of the $i$th row of $C$ are also freely varying, but we concentrate here on the case where all Kronecker indices $\nu_i$ are positive. To compute $\bnu$ as well as the coefficients $\alpha_{ij,k}$ and $b_{ij}$ for a given rational matrix function $H$, several numerically stable and efficient algorithms are available in the literature \citep[see, e.\,g.,][and the references therein]{rozsa1975minimal}. The orthogonal invariance inherent in spectral factorization (see \cref{specfac}) implies that this parametrization alone does not ensure identifiability. One remedy is to restrict the parametrization to transfer functions $H$ satisfying $H(0)=H_0$, for a non\hyp{}singular matrix $H_0$. To see how one must constrain the parameters $\alpha_{ij,k},b_{ij}$ in order to ensure this normalization, we work in terms of left matrix fraction descriptions.
\begin{theorem}[{Echelon MCARMA realization, \citet[Section 3]{guidorzi1975canonical}}]
\label{EchelonMCARMA}
For positive integers $d$ and $m$, let $H\in M_{d,m}(\R\{z\})$ be a rational matrix function with Kronecker indices $\bnu=(\nu_1,\ldots,\nu_d)$. Assume that $(A,B,C)$ is a realization of $H$, parametrized as in \cref{canSSMpara}. Then a unique left matrix fraction description $P^{-1}Q$ of $H$ is given by $ P(z)=\left[p_{ij}(z)\right]$, $Q(z)=\left[q_{ij}(z)\right]$, where
\begin{equation}
p_{ij}(z)=\delta_{i,j}z^{\nu_i}-\sum_{k=1}^{\nu_{ij}}{\alpha_{ij,k}z^{k-1}},\quad q_{ij}(z)=\sum_{k=1}^{\nu_i}{\kappa_{\nu_1+\ldots+\nu_{i-1}+k,j}z^{k-1}},
\end{equation}
and the coefficient $\kappa_{i,j}$ is the $(i,j)$th entry of the matrix $K=TB$, where the matrix $T=(T_{ij})_{i,j=1,\ldots,d}\in M_N(\R)$ is a block matrix with blocks $T_{ij}\in M_{\nu_i,\nu_j}(\R)$ given by
\begin{equation}
T_{ij} = \left(\begin{array}{cccccc}
                 -\alpha_{ij,2}		&\ldots	&-\alpha_{ij,\nu_{ij}}	&0	&\ldots		&0	\\
		 \vdots			&\iddots	&			&	&		&\vdots	\\
		 -\alpha_{ij,\nu_{ij}}	&	&			&	&		&\vdots	\\
		 0			&	&			&	&		&\vdots	\\
		 \vdots			&	&			&	&		&\vdots	\\
		 0			&\ldots	&	\ldots		&\ldots	&\ldots		&0
                \end{array}
\right)+\delta_{i,j}\left(\begin{array}{cccccc}
                 0	&0	&\ldots	&\ldots	&0	& 1 	\\
		 0	&0	&\ldots	&	&1	& 0	\\
		 \vdots	&\vdots	&	&\iddots&	& \vdots\\
		 \vdots	&	&\iddots&	& \vdots& \vdots\\
		 0	&1	&	&\ldots	& 0	& 0	\\
		 1	&0	&\ldots	&\ldots	& 0	& 0
                \end{array}
\right).
\end{equation}
\end{theorem}
The orders $p,q$ of the polynomials $P,Q$ satisfy $p=\max\{\nu_1,\ldots,\nu_d\}$ and $q\leq p-1$. Using this parametrization, there are different ways to impose the normalization $H(0)=H_0\in M_{d,m}(\R)$. One first observes that the special structure of the polynomials $P$ and $Q$ implies that $H(0)=P(0)^{-1}Q(0) = -(\alpha_{ij,1})_{ij}^{-1}(\kappa_{\nu_1+\ldots+\nu_{i-1}+1,j})_{ij}$. The canonical state space parametrization $(A,B,C)$ given by \cref{canSSMpara} therefore satisfies $H(0)=-CA^{-1}B=H_0$ if one makes the coefficients $\alpha_{ij,1}$ functionally dependent on the free parameters $\alpha_{ij,m}$, $m=1,\ldots\nu_{ij}$ and $b_{ij}$ by setting $\alpha_{ij,1}=-[(\kappa_{\nu_1+\ldots+\nu_{k-1}+1,l})_{kl}H_0^{\sim1}]_{ij}$, where $\kappa_{ij}$ are the entries of the matrix $K$ appearing in \cref{EchelonMCARMA} and $H_0^{\sim 1}$ is a right inverse of $H_0$. Another possibility, which has the advantage of preserving the multi-companion structure of the matrix $A$, is to keep the $\alpha_{ij,1}$ as free parameters, and to restrict some of the entries of the matrix $B$ instead. Since $|\det K| = 1$ and the matrix $T$ is thus invertible, the coefficients $b_{ij}$ can be written as $B=T^{-1}K$. Replacing the $(\nu_1+\ldots+\nu_{i-1}+1,j)$th entry of $K$ by the $(i,j)$th entry of the matrix $-(\alpha_{kl,1})_{kl}H_0$ makes some of the $b_{ij}$ functionally dependent on the entries of the matrix $A$, and results in a state space representation with prescribed Kronecker indices and satisfying $H(0)=H_0$. This latter method has also the advantage that it does not require the matrix $H_0$ to possess a right inverse. In the special case that $d=m$ and $H_0=-\I_d$, it suffices to set $\kappa_{\nu_1+\ldots+\nu_{i-1}+1,j} = \alpha_{ij,1}$.
% , for $i,j=1,\ldots,d$. 
Examples of normalized low\hyp{}order canonical parametrizations are given in \cref{tablecanparassm,tablecanparamcarma}.
\begin{table}
{\centering
\small
\begin{tabular}{|c|c|c|c|c|}
\hline
$\bnu$	& $n(\bnu)$ 	& $A$ 	& $B$ 	& $C$	\\
\hline
$(1,1)$ 		& $7$			&$\left(\begin{array}{cc}	
        		     			        \vartheta_1 & \vartheta_2\\
							 \vartheta_3 & \vartheta_4
        		     			       \end{array}\right)$		&$\left(\begin{array}{cc}	
        		     			        \vartheta_1 & \vartheta_2\\
							 \vartheta_3 & \vartheta_4
        		     			       \end{array}\right)$	&$\left(\begin{array}{cc}	
        		     			        1 & 0 \\ 0 & 1
        		     			       \end{array}\right)$	\\
\hline
$(1,2)$			& $10$		&$\left(\begin{array}{ccc}
					    \vartheta_1 & \vartheta_2 & 0 \\
					    0 & 0 & 1 \\
					    \vartheta_3 & \vartheta_4 & \vartheta_5
					\end{array}\right)$	& $\left(\begin{array}{cc}
								    \vartheta_1 & \vartheta_2 \\
								    \vartheta_6 & \vartheta_7 \\
								    \vartheta_3+\vartheta_5\vartheta_6 & \vartheta_4+\vartheta_5\vartheta_7
								  \end{array}\right)$ & $\left(\begin{array}{ccc}1 & 0 & 0 \\ 0 & 1 & 0\end{array}\right)$\\
\hline
$(2,1)$			& $11$			&$\left(\begin{array}{ccc}
                0 & 1 & 0 \\
		\vartheta_1 & \vartheta_2 & \vartheta_3 \\
		\vartheta_4 & \vartheta_5 & \vartheta_6
               \end{array}\right)$	& 	$\left(\begin{array}{cc}
		\vartheta_7 & \vartheta_8 \\
		\vartheta_1+\vartheta_2\vartheta_7 & \vartheta_3+\vartheta_2\vartheta_8 \\
		\vartheta_4+\vartheta_5\vartheta_7 & \vartheta_6+\vartheta_5\vartheta_8
		\end{array}\right)$	& $\left(\begin{array}{ccc}1 & 0 & 0 \\ 0 & 0 & 1\end{array}\right)$\\
\hline
$(2,2)$			&$15$		& $\left(\begin{array}{cccc}
       			 		         0 & 1 & 0 & 0 \\
						 \vartheta_1 & \vartheta_2 & \vartheta_3 & \vartheta_4 \\
						 0 & 0 & 0 & 1 \\
						 \vartheta_5 & \vartheta_6 & \vartheta_7 & \vartheta_8  
       			 		         \end{array}\right)$	& $\left(\begin{array}{cc}
									  \vartheta_9 & \vartheta_{10}\\
									   \vartheta_1+\vartheta_{4}\vartheta_{11}+\vartheta_2\vartheta_9 & \vartheta_3+\vartheta_2\vartheta_{10}+\vartheta_4\vartheta_{12}\\
									    \vartheta_{11}	& \vartheta_{12}\\
									    \vartheta_5+\vartheta_8\vartheta_{11}+\vartheta_6\vartheta_9 & \vartheta_7+\vartheta_6\vartheta_{10}+\vartheta_8\vartheta_{12}	
									  \end{array}\right)$ & $\left(\begin{array}{cccc}
												    1 & 0 & 0 & 0 \\ 0 & 0 & 1 & 0
												  \end{array}\right)$\\
\hline
\end{tabular}}
\vspace{.5cm}\\
\caption{Canonical state space realizations $(A,B,C)$ of normalized ($H(0)=-\I_2$) rational transfer functions in $M_2(\R\{z\})$ with different Kronecker indices $\bnu$; the number of parameters, $n(\bnu)$, includes three parameters for a covariance matrix $\SigL $.}
\label{tablecanparassm}
\end{table}

\begin{table}
{\centering
\small
\begin{tabular}{|c|c|c|c|c|}
\hline
$\bnu$	& $n(\bnu)$ 	& $P(z)$ 	& $Q(z)$	& $(p,q)$	\\
\hline
  $(1,1)$ 		& $7$				&$\left(\begin{array}{cc}z-\vartheta_1&-\vartheta_2\\-\vartheta_3&z-\vartheta_4\end{array}\right)$	&$\left(\begin{array}{cc}\vartheta_1&\vartheta_2\\\vartheta_3&\vartheta_4\end{array}\right)$	& $(1,0)$		\\
\hline
    $(1,2)$		& $10$				&$\left(\begin{array}{cc}z-\vartheta_1&-\vartheta_2\\-\vartheta_3&z^2-\vartheta_4z-\vartheta_5\end{array}\right)$	&$\left(\begin{array}{cc}\vartheta_1&\vartheta_2\\\vartheta_6z+\vartheta_3&\vartheta_7z+\vartheta_5\end{array}\right)$	& $(2,1)$		\\
\hline
$(2,1)$			& $11$				&$\left(\begin{array}{cc}z^2-\vartheta_1z-\vartheta_2&-\vartheta_3\\-\vartheta_4z-\vartheta_5&z-\vartheta_6\end{array}\right)$	&$\left(\begin{array}{cc}\vartheta_7z+\vartheta_2&\vartheta_8z+\vartheta_3\\\vartheta_5&\vartheta_6\end{array}\right)$	& $(2,1)$		\\
\hline
$(2,2)$			&$15$				&$\left(\begin{array}{cc}z^2-\vartheta_1z-\vartheta_2&-\vartheta_3z-\vartheta_4\\-\vartheta_5z-\vartheta_6&z^2-\vartheta_7z-\vartheta_8\end{array}\right)$		&$\left(\begin{array}{cc}\vartheta_9z+\vartheta_2&\vartheta_{10}z+\vartheta_4\\\vartheta_{11}z+\vartheta_6&\vartheta_{12}z+\vartheta_8\end{array}\right)$	& $(2,1)$		 \\
\hline
\end{tabular}}
\vspace{.5cm}\\
\caption{Canonical MCARMA realizations $(P,Q)$ with order $(p,q)$ of normalized ($H(0)=-\I_2$) rational transfer functions in $M_2(\R\{z\})$ with different Kronecker indices $\bnu$; the number of parameters, $n(\bnu)$, includes three parameters for a covariance matrix $\SigL $.} 
\label{tablecanparamcarma}
\end{table}

\subsection{A simulation study}
\label{section-Simulation}
We present a simulation study for a bivariate CARMA process with Kronecker indices $(1,2)$, i.\,e.\ CARMA indices $(p,q)=(2,1)$. As the driving L\'evy process we chose a zero\hyp{}mean normal\hyp{}inverse Gaussian (NIG) process $(\Lb(t))_{t\in\R}$. Such processes have been found to be useful in the modelling of stock returns and stochastic volatility, as well as turbulence data \citep[see, e.\,g.,][]{barndorff1997normal,rydberg1997normal}. 
% barndorff2004parsimonious, barndorff1997processes
The distribution of the increments $\Lb(t)-\Lb(t-1)$ of a bivariate normal\hyp{}inverse Gaussian L\'evy process is characterized by the density
\begin{equation*}
f_{\operatorname{NIG}}(\boldsymbol{x};\boldsymbol{\mu},\alpha,\boldsymbol{\beta},\delta,\Delta)=\frac{\delta\exp(\delta\kappa)}{2\pi}\frac{\exp(\langle\boldsymbol{\beta}\boldsymbol{x}\rangle)}{\exp(\alpha g(\boldsymbol{x}))}\frac{1+\alpha g(\boldsymbol{x})}{g(\boldsymbol{x})^3},\quad\boldsymbol{x}\in\R^2,
\end{equation*}
where
\begin{equation*}
g(\boldsymbol{x})=\sqrt{\delta^2+\langle\boldsymbol{x}-\boldsymbol{\mu},\Delta(\boldsymbol{x}-\boldsymbol{\mu}\rangle},\quad \kappa^2=\alpha^2-\langle\boldsymbol{\beta},\Delta\boldsymbol{\beta}\rangle>0,
\end{equation*}
and $\boldsymbol{\mu}\in\R^2$ is a location parameter, $\alpha\geq 0$ is a shape parameter, $\boldsymbol{\beta}\in\R^2$ is a symmetry parameter, $\delta\geq 0$ is a scale parameter and $\Delta\in M_2^+(\R)$, $\det\Delta=1$, determines the dependence between the two components of $(\Lb(t))_{t\in\R}$. For our simulation study we chose parameters 
\begin{equation}
\label{NIGparams}
\delta=1,\quad \alpha= 3,\quad \boldsymbol{\beta}=(1,1)^T,\quad \Delta=\left(\begin{array}{cc}5/4 & -1/2 \\ -1/2 & 1\end{array}\right),\quad \boldsymbol{\mu}=-\frac{1}{2\sqrt{31}}(3,2)^T,
\end{equation}
resulting in a skewed distribution with mean zero and covariance $\SigL \approx\left(\begin{array}{cc}0.4751 & -0.1622 \\-0.1622 & 0.3708\end{array}\right)$. A sample of $350$ independent replicates of the bivariate $\text{CARMA}_{1,2}$ process $(\Y(t))_{t\in\R}$ driven by a normal\hyp{}inverse Gaussian L\'evy process $(\Lb(t))_{t\in\R}$ with parameters given in \cref{NIGparams} were simulated on the equidistant time grid $0,0.01,\ldots,2000$ by applying an Euler scheme to the stochastic differential equation \labelcref{eq-ssm-QML} making use of the canonical parametrization given in \cref{tablecanparassm}. For the simulation, the initial value $\X(0)=\bzero_3$ and parameters $\vartheta_{1:7}=(-1, -2, 1, -2, -3, 1, 2)$ was used. Each realization was sampled at integer times ($h=1$), and QML estimates of $\vartheta_1,\ldots,\vartheta_7$ as well as $(\vartheta_8,\vartheta_9,\vartheta_{10})\coloneqq\vech\SigL $ were computed by numerical maximization of the quasi log-likelihood function using a differential evolution optimization routine \citep{price2005differential} in conjunction with a subspace trust-region method 
% \citep{branch2000subspace,byrd1988approximate}.
In \cref{tablesimstudy} the sample means and sampled standard deviations of the estimates are reported. Moreover, the standard deviations were estimated using the square roots of the diagonal entries of the asymptotic covariance matrix \labelcref{asympCov} with $s(L)=\lfloor L/\log L\rfloor^{1/3}$, and the estimates are also displayed in \cref{tablesimstudy}.
\begin{table}
{\centering
\small
\begin{tabular}{|c|c|c|c|c|}
 \hline
parameter	& sample mean 	& bias 		& sample std. dev. & mean est. std. dev.\\
\hline
$\vartheta_1$ 	& -1.0001	& 0.0001	&0.0354		&0.0381	\\
$\vartheta_2$ 	&  -2.0078	& 0.0078	&0.0479		&0.0539	\\
$\vartheta_3$ 	&  1.0051	&-0.0051	&0.1276		&0.1321	\\
$\vartheta_4$ 	&  -2.0068	& 0.0068	&0.1009		&0.1202	\\
$\vartheta_5$ 	&  -2.9988	&-0.0012	&0.1587		&0.1820	\\
$\vartheta_6$ 	&  1.0255	&-0.0255	&0.1285		&0.1382	\\
$\vartheta_7$ 	&  2.0023	&-0.0023	&0.0987		&0.1061	\\
$\vartheta_8$ 	&  0.4723	&-0.0028	&0.0457		&0.0517	\\
$\vartheta_9$ 	&  -0.1654	& 0.0032	&0.0306		&0.0346	\\
$\vartheta_{10}$&  0.3732	& 0.0024	&0.0286		&0.0378	\\
\hline
\end{tabular}}
\vspace{.5cm}\\
\caption{QML estimates for the parameters of a bivariate NIG\hyp{}driven $\text{CARMA}_{1,2}$ process observed at integer times over the time horizon $[0,2000]$. The second column reports the empirical mean of the estimators as obtained from 350 independent paths; the third and fourth columns contain the resulting bias and the sample standard deviation of the estimators, respectively, while the last column reports the average of the expected standard deviations of the estimators as obtained from the asymptotic normality result \cref{theorem-CLTmcarma}.}
\label{tablesimstudy}
\end{table}
One sees that the bias, the difference between the sample mean and the true parameter value, is very small in accordance with the asymptotic consistency of the estimator. Moreover, the estimated standard deviation is always slightly larger than the sample standard deviation, yet close enough to provide a useful approximation for, e.\,g., the construction of confidence regions. In order not to underestimate the uncertainty in the estimate, such a conservative approximation to the true standard deviations is desirable in practice. Overall, the estimation procedure performs very well in the simulation study.

\section*{Acknowledgements}
ES acknowledges financial support from the International Graduate School of Science and Engineering of the Technische Universit\"at M\"unchen. RS is grateful for the support of Deutsche Forschungsgemeinschaft (DFG) via research grant STE 2005/1-1. Both authors acknowledge financial support from the TUM Institute for Advanced Study, funded by the German Excellence Initiative.

{\small
\bibliographystyle{imsart-nameyear}

}

\end{document}